\documentclass[leqno,12pt]{amsart}
\usepackage{amssymb, amsmath}
\usepackage{color}

\newcommand{\eps}{\varepsilon}

\newcommand{\cS}{\mathcal{S}}

\newcommand{\Z}{\Bbb Z}

\usepackage{amsmath,amssymb,amsfonts}
\usepackage{graphicx}

\usepackage[utf8]{inputenc}
\usepackage[colorlinks=true,linkcolor=black,citecolor=black]{hyperref}

\newcommand{\RR}{\mathbb{R}}

\newcommand{\ZZ}{\mathbb{Z}}
\newcommand{\NN}{\mathbb{N}}

\newcommand{\DD}{\Delta}

\newcommand{\Bcal}{\mathcal{B}}

\newcommand{\Rcal}{\mathcal{R}}
\newcommand{\Scal}{\mathcal{S}}

\newcommand{\dd}{\partial}
\newcommand{\divv}{\mbox{div\,}}

\newcommand{\supp}{\mbox{supp }}

\newcommand{\abs}[1]{\left\vert#1\right\vert}
\newcommand{\set}[1]{\left\{#1\right\}}
\newcommand{\psca}[1]{\left\langle#1\right\rangle}

\newcommand{\pare}[1]{\left(#1\right)}
\newcommand{\norm}[1]{\left\Vert#1\right\Vert}

\newcommand{\f}{\frac}

\newcommand{\pa}{\partial}
\newcommand{\p}{\partial}
\newcommand\na{\nabla}

\newtheorem{theorem}{Theorem}[section]
\newtheorem{definition}[theorem]{Definition}
\newtheorem{lemma}[theorem]{Lemma}
\newtheorem{proposition}[theorem]{Proposition}

\newtheorem{remark}[theorem]{Remark}

\title{\textsc{Hydrostatic approximation of the 2D MHD system in a thin strip with a small analytic data}}

\author{Nacer Aarach }

\date{}

\begin{document}

\maketitle

\begin{abstract}

In this paper, we study the global well-posedness of the two dimensional rescaled anisotropic MHD system and also to the hydrostatic MHD  with no slip boundary condition in a  strip,  for small initial data, which are analytic in the horizontal variable. We also justify the limit of the MHD system (when $\epsilon\to 0$ where $\epsilon$ denotes the height of the initial strip) is given by a couple of a Prandtl like system for the velocity and a magnetic field equation when the initial data are small and analytic.

\end{abstract}

\section{Introduction}
The dynamics of an electrically conducting liquid near a wall has been a topic of constant interest, at least since the pioneering work of Hartmann \cite{JH}. It is relevant to many domains of active research, such as dynamo theory \cite{BEE2004} or nuclear fusion \cite{J}. In this paper, we consider the global well-posednss of the following two-dimensional Navier-Stokes system coupled with an approximation of the Maxwell equation, this system is called MHD system. We studied this system in a thin striped domain and provided with Dirichlet boundary conditions. Then let $\mathcal{S}^\eps = \lbrace (x,y) \in \mathbb{R}^2 : 0<y<\eps \rbrace $ where $\eps$ is the width of the strip. So, our system is of the following form:
\begin{equation}
	\label{NS2}
	\left\{\;
	\begin{aligned}
		&\partial_t U^\eps + U^\eps .\nabla U^\eps-\eps^2 \Delta U^\eps + \nabla P^\eps = B^\eps .\nabla B^\eps && \qquad \mbox{in }\,\mathcal{S} \times ]0,\infty[\\
		&\partial_t B^\eps + U^\eps.\nabla B^\eps - \eps^2 \Delta B^\eps = B^\eps .\nabla B^\eps && \qquad \mbox{in }\,\mathcal{S} \times ]0,\infty[\\
		&\divv U^\eps = \divv B^\eps = 0 && \qquad \mbox{in }\,\mathcal{S} \times ]0,\infty[ \\
		&U^\eps_{/t=0} = U^\eps_0, \text{ \ \ \ } B^\eps_{/t=0} = B^\eps_0 && \qquad \mbox{in }\,\mathcal{S} \times ]0,\infty[  		
	\end{aligned}
	\right.
\end{equation}
where $U^\eps(t,x,y) = \pare{u^\eps(t,x,y), v^\eps(t,x,y)}$ is the velocity vector of the fluid and $P^\eps(t,x,y)$ represent the pressure function which guarantees the free divergence property of the velocity field $U^\eps$; $B^\eps(t,x,y) = \pare{b^\eps(t,x,y), c^\eps(t,x,y)}$ represents the magnetic field of the fluid. The system $\eqref{NS2}$ is completed by the following boundary condition
\begin{equation*}
	U^\eps_{\vert y=0} = U^\eps_{\vert y=\eps} = 0 \quad\mbox{and}\quad B^\eps_{\vert y=0} = B^\eps_{\vert y=\eps} = 0.
\end{equation*}
In our equation \eqref{NS2} the Laplace is given by $\Delta = \partial_x^2 + \partial_y^2$. 

The purpose of this work is to rely on this recent progress to gain insight and perspective in the study of the Magnetohydrodynamic system (MHD), which is coupled the Navier-Stokes equation with small viscosity and an approximation of the Maxwell equation for the electromagnetic field. this note is to gain some insight into the analysis of MHD boundary layer models. It is primarily intended to mathematicians, either applied or interested in the theory of fluid PDE’s. The goal is twofold. First, we wish to provide a clear picture of the various models available, depending on the asymptotic under consideration. 
We hope that this work will serve as a starting point for more complete mathematical and numerical analysis. Between the mechanics of "classic" fluids and the magnetohydrodynamics, lies the electrohydrodynamics or mechanical fluids ionized in the presence of electric fields (electrostatic), but without magnetic field. This system describes the evolution of a conducting fluid under the effect of a magnetic field (such as, for instance, the liquid iron in the Earth's core under the influence of the Earth's magnetic field) in a strip  (larger than the atmosphere) which is called the magnetosphere.


In order to describe hydrodynamical flows on the earth, in geophysics, it is usually assumed that vertical motion is much smaller than horizontal motion and that the fluid layer depth is small compared to the radius of the sphere, thus they are good approximation of global magnetosphere flow and oceanic flow. The thin-striped domain in the system \eqref{NS2} is considered to take into account this anisotropy between horizontal and vertical directions. Under this assumption, it is believed that the dynamics of fluids in large scale tends towards a geostrophic balance (see \cite{G1982}, \cite{H2004} or \cite{PZ2005}). In a formal way, as in \cite{PZZ2019}, taking into account this anisotropy, we also consider the initial data in the following form, 
 
$$ U^\eps_{\vert t=0} = U_{0}^{\eps} = \left( u_0\pare{x,\frac{y}{\eps}},\eps v_{0}\pare{x,\frac{y}{\eps}}\right) \text{ \ \ in \ \ } \mathcal{S}^{\eps}$$ and $$B^\eps_{\vert t=0} = B^\eps_0 = \left( b_0\pare{x,\frac{y}{\eps}},\eps c_{0}\pare{x,\frac{y}{\eps}}\right).$$

Then, In our paper, we look for solutions of the system $\eqref{NS2}$ in the following form  : 
\begin{equation} 
	\label{S1eq3}
	\left\{\;
	\begin{aligned}
		&U^\eps(t,x,y)=\pare{u^\eps \pare{t,x,\frac{y}{\eps}}, \eps v^\eps \pare{t,x,\frac{y}{\eps}}}\\
		&B^\eps(t,x,y)= \pare{b^\eps \pare{t,x,\frac{y}{\eps}}, \eps c^\eps \pare{t,x,\frac{y}{\eps}}} \\
		&P^\eps(t,x,y)=p^\eps \pare{t,x,\frac{y}{\eps}}.
	\end{aligned}
	\right.
\end{equation}

Let $\mathcal{S} = \left\{(x,z)\in\mathbb{R}^2:\ 0<y<1\right\}$, we start by giving the two equations obtained by replacing $U^\eps$ by $u^\eps$ and $v^\eps$ and $B^\eps$ by $b^\eps$ and $c^\eps$
\begin{equation}
	\label{eq:hydroPE}
	\left\{\;
	\begin{aligned}
		&\partial_t u^\eps +u^\eps\partial_x u^\eps + v^\eps\partial_yu^\eps-\eps^2\partial_x^2u^\eps-\partial_y^2u^\eps+\partial_x p^\eps=b^\eps\pa_xb^\eps + c^\eps\pa_yb^\eps && \mbox{dans } \; \mathcal{S}\times ]0,\infty[\\
		&\eps^2\left(\partial_tv^\eps+u^\eps\partial_x v^\eps+v^\eps\partial_yv^\eps-\eps^2\partial_x^2v^\eps-\partial_y^2v^\eps\right)+\partial_yp^\eps=\eps^2 \left( b^\eps \pa_x c^\eps + c^\eps \pa_y c^\eps \right) && \mbox{dans } \; \mathcal{S}\times ]0,\infty[\\
		&\partial_t b^\eps +u^\eps\partial_x b^\eps + v^\eps\partial_yb^\eps-\eps^2\partial_x^2b^\eps-\partial_y^2b^\eps = b^\eps\pa_xu^\eps + c^\eps\pa_yu^\eps && \mbox{dans } \; \mathcal{S}\times ]0,\infty[\\
		&\eps\left(\partial_tc^\eps+u^\eps\partial_x c^\eps+v^\eps\partial_yc^\eps-\eps^2\partial_x^2c^\eps-\partial_y^2c^\eps\right) = \eps\left(b^\eps \pa_x v^\eps + c^\eps \pa_y v^\eps \right)  && \mbox{dans } \; \mathcal{S}\times ]0,\infty[\\
		&\partial_x u^\eps+\partial_yv^\eps=0 && \mbox{dans } \; \mathcal{S}\times ]0,\infty[\\
		&\partial_x b^\eps+\partial_yc^\eps=0 && \mbox{dans } \; \mathcal{S}\times ]0,\infty[\\
		&\left(u^\eps, v^\eps, b^\eps, c^\eps \right)|_{t=0}=\left(u_0, v_0,b_0, c_0 \right) && \mbox{dans } \; \mathcal{S}\\
		&\left(u^\eps, v^\eps,b^\eps, c^\eps \right)|_{z=0}=\left(u^\eps, v^\eps,b^\eps, c^\eps \right)|_{z=1}=0.
	\end{aligned}
	\right.
\end{equation}
Formally taking $\eps\to 0$ in the scaled system $\eqref{eq:hydroPE}$, we obtain the Prandtl system on u and also for b which is of the following form :
\begin{equation}
	\label{eq:hydrolimit}
	\left\{\;
	\begin{aligned}
		&\p_tu+u\p_x u+v\p_yu-\p_y^2u+\p_xp=b\pa_xb + c\pa_yb && \mbox{dans } \; \cS\times ]0,\infty[\\
		&\p_yp= 0 && \mbox{dans } \; \cS\times ]0,\infty[\\
		&\p_t b+u\partial_{x} b+v\partial_{y} b-\pa_y^2 b= b\pa_x u + c\pa_y u && \mbox{dans } \; \cS\times ]0,\infty[\\
		&\p_xu+\p_yv=0  && \mbox{dans } \; \cS\times ]0,\infty[\\
		&\p_xb+\p_yc=0  && \mbox{dans } \; \cS\times ]0,\infty[\\
		&u|_{t=0}=u_0  && \mbox{dans } \; \cS\\
		&b|_{t=0}=b_0 && \mbox{dans } \; \cS,
	\end{aligned}
	\right.
\end{equation}
where $(u,v)$ and $(b,c)$ satisfying the Dirichlet boundary condition  
\begin{equation} 
	\label{S1eq7}
	\left(u, v,b,c\right)|_{y=0}=\left(u, v,b,c\right)|_{y=1}=0.
\end{equation}

Our goal is to achieve the global existence of the solution for systems $\eqref{eq:hydroPE}$ et $\eqref{eq:hydrolimit}$, then we want to show the convergence of the scaled MHD system $\eqref{eq:hydroPE}$ to the limit system when $\eps$ tends towards zero.
\begin{remark}\label{remk}
We can also get the equation satisfied by $c$ when $\eps \to 0$. We remark that the system $\eqref{eq:hydrolimit}$ implies the following equation for $c$:
$$ \p_t c+u\partial_{x} c+v\partial_{y} c-\pa_y^2 c= b\pa_x v + c\pa_y v \ \ \ \  \mbox{dans } \; \cS\times ]0,\infty[. $$
  Indeed, by taking the derivative of the equation satisfied by $b$ with respect to $x$ variable, and using the free divergence of $U$ and $B$,  we obtain 
 \begin{equation*}
 \pa_x \left(  \p_t b+u\partial_{x} b+v\partial_{y} b-\pa_y^2 b- b\pa_x u - c\pa_y u \right) =-\pa_y \left(  \p_t c+u\partial_{x} c+v\partial_{y} c-\pa_y^2 c- b\pa_x v - c\pa_y v \right). 
 \end{equation*}
 From which we obtain
 $$ \p_t c+u\partial_{x} c+v\partial_{y} c-\pa_y^2 c- b\pa_x v - c\pa_y v = c(t,x), $$ but as $c$ it zero into the boundary then $c(t,x) = 0$ this ensure the equation satisfied by $c$. 
\end{remark}

We remark that in the system \eqref{eq:hydrolimit}, we have to deal with the same difficulty as for Prandtl equations due to its degenerate form and the nonlinear term $v\dd_yu$ which will lead to the loss of one derivative in the tangential direction in the process of energy estimates. For a more complete survey on this very challenging problem and we suggest the reader to the works \cite{awxy, e-1,e-2,GV-D,M-W} and references therein. To overcome this difficulty, one has to impose a monotony hypothesis on the normal derivative of the velocity or an analytic regularity on the velocity. After the pioneer works of Oleinik \cite{oleinik} using the Croco transformation under the monotony hypothesis. Lately, this result was proved via energy method in \cite{awxy,M-W} indepzndently by taking care of the cancellation property in the convection terms of (PE). Sammartino and Caflisch \cite{samm} solved the problem for analytic solutions on a half space and later, the analyticity in normal variable $y$ was removed by Lombardo, Cannone and Sammartino in \cite{LCS2003}. The main argument used in \cite{samm, LCS2003} is to apply the abstract Cauchy-Kowalewskaya (CK) theorem. We also mention a well-posedness result of Prandtl system for a class of data with Gevrey regularity \cite{GvM2015}. Lately, for a class of convex data, G\'erard-Varet, Masmoudi and Vicol \cite{GvMV2019} proved the well-posedness of the Prandtl system in the Gevrey class. We also want to remark that unlike the case of Prandtl equations, in the system \eqref{eq:hydrolimit}, the pressure term is defined by $\dd_y p = 0$. One of the novelties of the paper is to find a way to treat the pressure term.

We also want to recall some results concerning the system $\eqref{eq:hydroPE}$. This system was studied in the 90's by Lions-Temam-Wang \cite{Temam,Wang,Lions}, where the authors considered full viscosity and diffusivity, and establish the global existence of weak solutions. Concerning the strong solutions for the 2D case, the locale existence result was established by Guillén-Gonzàlez, Masmoudi and Rodriguez-Bellido \cite{30}, while the global existence for 2D case was achieved by Bresch, Kazhikhov and Lemoine in \cite{6} and by Temam and Ziane in \cite{Ziane}. In our paper we also want to establish the global well posedness of the system \eqref{eq:hydroPE} in 2D case but in a thin strip.

\section{Littlewood-Paley Theory and Functional Framework}

To introduce the result of this paper, we will recall some elements of the Littlewood-Paley theory and also introduce the function space and technique using for the proof of our result. So we define the dyadic operator in the horizontal variable, (of $x$ variable) and for all $q \in \ZZ$, we recalle from \cite{BCD2011book} that  
	\begin{align*}
	\Delta_q^h a(x,y) &= \mathcal{F}_h^{-1}\pare{\varphi(2^{-q}\abs{\xi})\widehat{a}(\xi,y)},\\
	S_q^h a(x,y) &= \mathcal{F}_h^{-1}\pare{\psi(2^{-q}\abs{\xi})\widehat{a}(\xi,y)}.
	\end{align*}
	
where $\psi$ and $\varphi$ are a smooth function such that 
\begin{align*}
     &supp \ \varphi \subset \lbrace z\in \RR/ \text{ \ } \f34 \leq |z| \leq \f83 \rbrace \text{ \ and \ } \forall z >0, \text{ \ } \sum_{q\in \ZZ} \varphi (2^{-q} z) =1,\\
     & supp \ \psi \subset \lbrace z\in \RR/ \text{ \ } |z| \leq \f43 \rbrace \text{ \ and \ } \text{ \ } \psi(z) +  \sum_{q \geq 0} \varphi (2^{-q} z) =1.
 \end{align*}
 and 
 \begin{equation*}
\forall\, q,q' \in \NN, \; \abs{q-q'}\geq 2, \quad \supp \varphi(2^{-q} \cdot) \cap \supp \varphi(2^{-q'} \cdot) = \emptyset.
\end{equation*}
And in all that follows, $\mathcal{F}a$ and $\widehat{a}$ always denote the partial Fourier transform of the distribution $a$ with respect to the horizontal variable (of $x$ variable), that is, $\widehat{a}(\xi,y) = \mathcal{F}_{x \rightarrow \xi} (a)(\xi,y).$ We refer to \cite{BCD2011book} and \cite{B1981} for a more detailed construction of the dyadic decomposition. Combined the definition of the dyadic operator to 
\begin{equation}
\label{decomposition} \forall z\in \RR, \quad\psi(z) + \sum_{j\in\NN} \varphi (2^{-j} z) = 1,
\end{equation}
implies that all tempered distributions can be decomposed with respect to the horizontal frequencies as $$ u = \sum_{q \in \ZZ} \Delta_q^v u.$$
We now introduce the function spaces used throughout the paper. As in \cite{PZZ2019}, we define the Besov-type spaces $\Bcal^s$, $s\in\RR$ as follows.
\begin{definition}
	Let $s \in \RR$ and $\Scal = \RR \times ]0,1[$. For any $u \in S'_h(\Scal)$, \emph{i.e.}, $u$ belongs to $S'(\Scal)$ and $\lim_{q\to -\infty} \norm{S^h_q u}_{L\infty} = 0$, we set
	\begin{equation*}
	\norm{u}_{\Bcal^s} \stackrel{\tiny def}{=} \; \sum_{q \in \ZZ} 2^{qs} \norm{\DD_q^h u}_{L^2}.
	\end{equation*}
	\begin{enumerate}
		\item[(i)] For $s \leq \frac{1}2$, we define 
		\begin{equation*}
		\Bcal^s(\Scal) \stackrel{\tiny def}{=} \; \set{u \in S'_h(\Scal) \ : \ \norm{u}_{\Bcal^s} < +\infty}.
		\end{equation*}
		\item[(ii)] For $s \in \; ]k-\frac12, k+\frac12]$, with $k \in \NN^*$, we define $\Bcal^s(\Scal)$ as the subset of distributions $u$ in $S'_h(\Scal)$ such that $\dd_x^ku \in \Bcal^{s-k}(\Scal)$.
	\end{enumerate}
\end{definition}
For a better use of the smoothing effect given by the diffusion terms, we will work in the following Chemin-Lerner type spaces and also the time-weighted Chemin-Lerner type spaces.
\begin{definition} \label{def:CLspaces}
	Let $p \in [1,+\infty]$ and $T \in ]0,+\infty]$. Then, the space $\tilde{L}^p_T(\Bcal^s(\Scal))$ is the closure of $C([0,T];S(\Scal))$ under the norm
	\begin{equation*}
	\norm{u}_{\tilde{L}^p_T(\Bcal^s(\Scal))} \stackrel{\tiny def}{=} \; \sum_{q \in \ZZ} 2^{qs} \pare{\int_0^T \norm{\DD^h_q u(t)}_{L^2}^p dt}^{\frac{1}{p}},
	\end{equation*}
	with the usual change if $p = +\infty$.
\end{definition} \label{def:wCLspaces}
\begin{definition}
	\label{def:CLweight}
	Let $p \in [1,+\infty]$ and let $f \in L^1_{\tiny loc}(\RR_+)$ be a nonnegative function. Then, the space $\tilde{L}^p_{t,f(t)}(\Bcal^s(\Scal))$ is the closure of $C([0,T];S(\Scal))$ under the norm
	\begin{equation*}
	\norm{u}_{\tilde{L}^p_{t,f(t)}(\Bcal^s(\Scal))} \stackrel{\tiny def}{=} \; \sum_{q\in\ZZ} 2^{qs} \pare{\int_0^t f(t') \norm{\DD^h_q u(t')}_{L^2}^p dt'}^{\frac{1}{p}}.
	\end{equation*}
\end{definition}

The following Bernstein lemma gives important properties of a distribution $u$ when its Fourier transform is well localized. We refer the reader to \cite{C1995book} for the proof of this lemma.
\begin{lemma}
	\label{lem:Bernstein}
	Let $k\in\NN$, $d \in \NN^*$ and $r_1, r_2 \in \RR$ satisfy $0 < r_1 < r_2$. There exists a constant $C > 0$ such that, for any $a, b \in \RR$, $1 \leq a \leq b \leq +\infty$, for any $\lambda > 0$ and for any $u \in L^a(\RR^d)$, we have
	\begin{equation*} 
	\supp\pare{\widehat{u}} \subset \set{\xi \in \RR^d \;\vert\; \abs{\xi} \leq r_1\lambda} \quad \Longrightarrow \quad \sup_{\abs{\alpha} = k} \norm{\dd^\alpha u}_{L^b} \leq C^k\lambda^{k+ d \pare{\frac{1}a-\frac{1}b}} \norm{u}_{L^a},
	\end{equation*}
	and
	\begin{equation*} 
	\supp\pare{\widehat{u}} \subset \set{\xi \in \RR^d \;\vert\; r_1\lambda \leq \abs{\xi} \leq r_2\lambda} \quad \Longrightarrow \quad C^{-k} \lambda^k\norm{u}_{L^a} \leq \sup_{\abs{\alpha} = k} \norm{\dd^\alpha u}_{L^a} \leq C^k \lambda^k\norm{u}_{L^a}.
	\end{equation*}
\end{lemma}

Finally to deal with the estimate concerning the product of two distribution, we shall frequently use the Bony's decomposition (see \cite{B1981} ) in the horizontal variable ( of the $x$ variable ) that for $f,g$ two tempered distribution  :
\begin{equation} \label{eq:Bony}
	fg = T^h_f g + T^h_g f + R^h(f,g),
	\end{equation}
	where
	$$ T^h_f g = \sum_{q} S^h_{q-1} f \Delta_q^h g, \text{ \ } T^h_g f = \sum_{q} S^h_{q-1} g \Delta_q^h f $$ and the rest term satisfied
	
	$$ R^h(f,g) = \sum_{q} \tilde{\Delta}^h_{q} f \Delta_q^h g \text{ \ with \ } \tilde{\Delta}^h_{q}f = \sum_{|q-q'| \leq 1} \Delta_{q'}^h f. $$
\subsection{Main results}

Our main difficulty relies in finding a way to estimate the nonlinear terms, which allows to exploit the smoothing effect given by the above function spaces. Using the method introduced by Chemin in \cite{C2004} (see also \cite{CGP2011}, \cite{PZ2011} or \cite{PZZ2019}), for any $f \in L^2(\Scal)$, we define the following auxiliary function, which allows to control the analyticity of $f$ in the horizontal variable $x$, 
\begin{align} \label{eq:AnPhi}
	f_{\phi}(t,x,y) = e^{\phi(t,D_x)} f(t,x,y) \stackrel{\tiny def}{=} \mathcal{F}^{-1}_h(e^{\phi(t,\xi)} \widehat{f}(t,\xi,y)) \text{ \ \ with \ \ } \phi(t,\xi) = (a- \lambda \theta(t))|\xi|,
\end{align}
where the quantity $\theta(t)$, which describes the evolution of the analytic band of $f$, satisfies
\begin{equation}
\label{eq:AnTheta} \forall\, t > 0, \ \dot{\theta}(t) \geq 0 \text{ \ \ and \ \ } \theta(0) = 0. 
\end{equation}

The main idea of this technique consists in the fact that if we differentiate, with respect to the time variable, a function of the type $e^{\phi(t,D_x)} f(t,x,y)$, we obtain an additional ``good term'' which plays the smoothing role. More precisely, we have
\begin{displaymath}
	\frac{d}{dt} \pare{e^{\phi(t,D_x)} f(t,x,y)} = -\dot{\theta}(t) \abs{D_x} e^{\phi(t,D_x)} f(t,x,y) + e^{\phi(t,D_x)} \dd_t f(t,x,y),
\end{displaymath}
where $-\dot{\theta}(t) \abs{D_x} e^{\phi(t,D_x)} f(t,x,y)$ gives a smoothing effect if $\dot{\theta}(t) \geq 0$. This smoothing effect allows to obtain our global existence and stability results in the analytic framework. Remark that the existence in the Prandtl case, we only have the local existence and the convergence is still an open question! Besides, MHD system is known to be very unstable.

\bigskip

Our main results are the following theorems.

The first result is the global well-posedness of the limit MHD system \eqref{eq:hydrolimit}, with small analytic data in the horizontal variable.
\begin{theorem} \label{th:1} [ Global well-posedness of the MHD limit system ].
Let $a>0$, we assume that for some constant $c_0$ sufficiently small independent of $\eps$, and for any initial data $(u_0,v_0,b_0,c_0) = (U_0,B_0)$ satisfying
\begin{align}\label{u0b0}
\Vert e^{a|D_x|} u_0 \Vert_{\mathcal{B}^\f12} + \Vert e^{a|D_x|} b_0 \Vert_{\mathcal{B}^\f12} \leq c_0 a,
\end{align} 
and there holds the compatibility condition $\int_0^1 (u_0,b_0) dy= 0$. Then the system limit \eqref{eq:hydrolimit} has a unique global solution $(U,B) = (u,v,b,c)$ satisfying 
\begin{align}\label{ub}
\Vert e^{\mathcal{R}t}(u_\phi,b_\phi)\Vert_{\tilde{L}^\infty(\mathbb{R}^+,\mathcal{B}^\f12)} + \Vert e^{\mathcal{R}t}\pa_y (u_\phi,b_\phi)\Vert_{\tilde{L}^2(\mathbb{R}^+,\mathcal{B}^\f12)} \leq C\Vert e^{a|D_x|} (u_0,b_0) \Vert_{\mathcal{B}^\f12}
\end{align} 
where $(u_\phi,v_\phi)$ will given by \eqref{eq:AnPhi} and $\mathcal{R}$ is a constant determined by the Poincarr\'e inequalityon the strip domain $\mathcal{S} = \left\{(x,y)\in\mathbb{R}^2:\ 0<y<1\right\}$. 
\end{theorem}

The second result is the global will-posedness of the scaled system \eqref{eq:hydroPE} with small analytic initial data in the horizontal variable $x$. The main interesting point here is that the smallness of data is independent of $\eps$ and there holds the global uniform estimate \eqref{ubvc} with respect to the parameter $\eps$.
\begin{theorem} \label{th:2} [ Global well-posedness of the scaled system ] Let $a>0$, there exist a constant $c_1$ sufficiently small independent of $\eps$, such that for any initial data $(u^\eps_0,v^\eps_0,b^\eps_0,c^\eps_0) = (U^\eps_0, B^\eps_0)$ satisfying 
\begin{align}\label{u0v0b0c0}
\Vert e^{a|D_x|} (u_0,\eps v_0) \Vert_{\mathcal{B}^\f12} + \Vert e^{a|D_x|} (b_0,\eps c_0) \Vert_{\mathcal{B}^\f12} \leq c_1 a,
\end{align}
then the system \eqref{eq:hydroPE} has a unique global solution $(U^\eps,B^\eps)$ so that 
 \begin{align}\label{ubvc}
&\Vert e^{\mathcal{R}t}(u_\varphi,\eps v_\varphi)\Vert_{\tilde{L}^\infty(\mathbb{R}^+,\mathcal{B}^\f12)} + \Vert e^{\mathcal{R}t}(b_\varphi,\eps c_\varphi)\Vert_{\tilde{L}^\infty(\mathbb{R}^+,\mathcal{B}^\f12)} + \Vert e^{\mathcal{R}t}\pa_y (u_\varphi,\eps v_\varphi)\Vert_{\tilde{L}^2(\mathbb{R}^+,\mathcal{B}^\f12)} \notag \\ & +\Vert e^{\mathcal{R}t}\pa_y (b_\varphi,\eps c_\varphi)\Vert_{\tilde{L}^2(\mathbb{R}^+,\mathcal{B}^\f12)} + \eps \Vert e^{\mathcal{R}t}\pa_x (u_\varphi,\eps v_\varphi)\Vert_{\tilde{L}^2(\mathbb{R}^+,\mathcal{B}^\f12)} + \eps \Vert e^{\mathcal{R}t}\pa_x (b_\varphi,\eps c_\varphi)\Vert_{\tilde{L}^2(\mathbb{R}^+,\mathcal{B}^\f12)} \\  &\leq C\left(\Vert e^{a|D_x|} (u_0,\eps v_0) \Vert_{\mathcal{B}^\f12} + \Vert e^{a|D_x|} (b_0,\eps c_0) \Vert_{\mathcal{B}^\f12} \right), \notag
\end{align} 
where $(u_\varphi,v_\varphi)$ and $(b_\varphi,c_\varphi)$ will given also by \eqref{eq:AnPhi}.
\end{theorem}

The main idea to prove the above two theorems is to control the new unknown $(u_\phi,b_\phi)$ defined by \eqref{eq:AnPhi}, where $u$ is the horizontal velocity and $b$ is the horizontal component of the fluid's magnetic field. $u_\phi $ is a weighted function of $u$ in the dual Fourier variable with an exponential function of $(a - \lambda \theta(t))|\xi|$ (also $b_\phi$ is a weighted of $b$ in the dual Fourier), which is equivalent to the analytically of the solution in the horizontal variable. By the classical Cauchy-Kowalewskaya theorem, one expects the radius of the analyticity of the solutions decay in time and so the exponent, which corresponds to the width of the analyticity strip, is allowed to vary with time. Using energy estimates on the equation satisfied by $(u_\phi,b_\phi) $ and the control of the quantity which describes " the loss of the analyticity radius ", we shall show that the analyticity strip persists globally in time. Consequently, our result is a global Cauchy-Kowalewskaya  type theorem.

The third result in concern the study of the convergence from the scaled anisotropic MHD system \eqref{eq:hydroPE} to the limit system  \eqref{eq:hydrolimit}, so in this theorem we proved that the convergence is globally in time. 
\begin{theorem} \label{th:3} [ Convergence to the MHD limit system ]
Let $a>0$, and $(u_0^\eps,v_0^\eps,b_0^\eps,c_0^\eps)$, satisfying \eqref{u0v0b0c0}. Let $u_0$ and $b_0$ satisfying $e^{a|D|_x} u_0 \in \mathcal{B}^\f12 \cap \mathcal{B}^\f52$, $e^{a|D|_x} \pa_y u_0 \in \mathcal{B}^\f32$, (the some thing for $b_0$) and there holds the  compatibility condition $ \int_0^1 (u_0,b_0) dy =0$ and 
\begin{align*}
\Vert e^{a|D_x|} u_0 \Vert_{\mathcal{B}^\f12} + \Vert e^{a|D_x|} b_0 \Vert_{\mathcal{B}^\f12} \leq \frac{c_2 a}{1+ \Vert e^{a|D_x|} (u_0,b_0) \Vert_{\mathcal{B}^\f32}},
\end{align*}
for some $c_2$ sufficiently small independent of $\eps$, then we have 
\begin{align} 
&\Vert (\Psi^1_\Theta,\eps \Psi^2_\Theta) \Vert_{\tilde{L}^\infty_t(\mathcal{B}^\f12)}+ \norm{(\Phi_{\Theta}^1,\eps \Phi_\Theta^2)(t)}_{L^\infty_t (\mathcal{B}^\f12)} \Vert \pa_y (\Psi^1_\Theta,\eps \Psi^2_\Theta) \Vert_{\tilde{L}^2_t(\mathcal{B}^\f12)} \notag \\ &  +\eps \Vert (\Psi^1_\Theta,\eps \Psi^2_\Theta) \Vert_{\tilde{L}^2_t(\mathcal{B}^\f32)}+ \Vert \pa_y (\Phi^1_\Theta,\eps \Phi^2_\Theta) \Vert_{\tilde{L}^2_t(\mathcal{B}^\f12)} +\eps \Vert (\Phi^1_\Theta,\eps \Phi^2_\Theta) \Vert_{\tilde{L}^2_t(\mathcal{B}^\f32)}  \\
&\leq C\left( \Vert e^{a|D_x|}(u_0^\eps-u_0, \eps(v_0^\eps  -v_0)) \Vert_{\mathcal{B}^\f12} + \Vert e^{a|D_x|}(b_0^\eps-b_0, \eps(c_0^\eps  -c_0)) \Vert_{\mathcal{B}^\f12} +M\eps \right). \notag
\end{align}
where 
\begin{align} \label{eq:psiphiqbis}
	\left\{
	\begin{aligned}
		(\Psi^{1,\eps},\Psi^{2,\eps},q^\eps) &= (u^\eps - u,v^\eps - v,p^\eps-p ), \\
		(\Phi^{1,\eps},\Phi^{2,\eps}) &= (b^\eps -b,c^\eps - c) .
	\end{aligned}
		\right.
\end{align}
and $v_0$ is determined from $u_0$ via $\pa_x u +\pa_y v =0 $ and $v_0|_{y=0} = v_0|_{y=1}= 0$, and $(\Psi_\Theta^1,\eps\Psi_\Theta^2)$, $(\Phi_\Theta^1,\eps\Phi_\Theta^2)$ will be given by \eqref{analytiqueff}.
\end{theorem}

\section{Global well posedness of the limit system}
The goal of this paper is to prove the global well posedness of the limit system of the MHD equation, we remark that the local smooth solution of the limit system follow a standard parabolic regularisation method similar to the MHD system, First we remark that the Dirichlet boundary condition $$ (u,v)_{/y=0} = (u,v)_{/y=1} = (b,c)_{/y=0} =(b,c)_{/y=1} =0,$$ and the incompressible condition $\pa_x u +\pa_y v =0$ and $\pa_x b + \pa_y c =0$ imply that :
\begin{align}\label{vv}
    v(t,x,y) &= \int_0^y \pa_y v(t,x,s) ds = - \int_0^y \pa_x u(t,x,s)ds \\
    c(t,x,y) &= \int_0^y \pa_y c(t,x,s) ds = - \int_0^y \pa_x b(t,x,s)ds.
\end{align}
We want now to find the equation for the pressure. Due to the Dirichlet boundary condition $(u,v,b,c)_{/y=0} = (u,v,b,c)_{/y=1} =0$ we deduce from the incompressibility condition $\pa_xu+\pa_yv=0$ and $\pa_xb+\pa_y c=0$ that
\begin{align*}
	\pa_x\int_0^1u(t,x,y)\, dy = -\int_0^1 \dd_y v(t,x,y)\, dy = v(t,x,1) - v(t,x,0) = 0 \\
	\pa_x\int_0^1b(t,x,y)\, dy = -\int_0^1 \dd_y c(t,x,y)\, dy = c(t,x,1) - c(t,x,0) = 0.
\end{align*}
Due to the compatibility condition $\pa_x \int_0^1 (u_0,b_0) dy = 0$ and the fact that $u(t,x,y) \rightarrow 0$ and $b(t,x,y) \rightarrow 0$ as $|x| \rightarrow 0$, ensure that 
 $$ \int_0^1 u(t,x,y) dy = 0 \text{ \ \ and \ \ } \int_0^1 b(t,x,y) dy = 0. $$ 
 Then by integrating the equations $\p_tu+u\p_x u+v\p_yu-\p_y^2u+\p_xp=b\pa_xb + c\pa_yb$ and $\p_tb+u\p_x b+v\p_yb-\p_y^2b = b\pa_xu + c\pa_yu$, for $y\in[0,1]$ and using the fact that $\pa_y p=0$, we obtain 
 $$ \pa_x p = \pa_y u(t,x,1) - \pa_yu(t,x,0) - \f12 \pa_x \int_0^1 (u)^2(t,x,y) dy + \f12 \pa_x \int_0^1 (b)^2(t,x,y) dy. $$ 
 
 \begin{remark}
We consider the equation which represent the magnetic field $B=(b,c)$,
$$ \pa_tB + U\nabla B - B\nabla U - \pa_y^2B =0 ,$$  and taking the trace of this equation on the boundary we find that $\pa_y^2 B=0$ for a smooth solutions of $\eqref{eq:hydrolimit}$. Using that $B$ satisfies the divergence free condition $\pa_xb +\pa_yc =0$, we obtain $$ \pa_x(\pa_y b) +\pa_y^2 c=0. $$ While $\pa_y^2c = 0$ on the boundaries of the strip, then $\pa_x(\pa_yb)=0$ also on the boundary. This implies that that $(\pa_y b)(t,x,0) = m(t)$ and using the fact that $\p_yb(t,x,y) \rightarrow 0$ as $|x| \rightarrow 0$, then we deduce that $(\pa_y b)(t,x,0)=0$. Using a similar argument we obtain  $(\pa_y b)(t,x,1) =0$.
 \end{remark}
 So, we have obtained
 $$ \pa_yb(t,x,1)=\pa_yb(t,x,0) = 0,$$
 for smooth solutions of $\eqref{eq:hydrolimit}$.
 \begin{remark}
 Using the previous boundary condition, we obtain that the average $\int_0^1 b(t,x,y )dy=0$ for all $t\geq 0$. Indeed, we have 
 \begin{align*}
     &\p_t \int_0^1b(t,x,y) dy+\int_0^1 u\partial_{x} bdy +\int_0^1 v\partial_{y} b dy - \int_0^1 \pa_y^2 b(t,x,y) dy- \int_0^1 b\pa_x u dy- \int_0^1 c \pa_y udy\\
     &=\p_t\int_0^1 b(t,xy)dy  -\pa_yb(t,x,1)+\pa_yb(t,x,0)=0
      \end{align*}
 \end{remark}

Let $(u_\phi, v_\phi, b_\phi, c_\phi)$ be define as in \eqref{eq:AnPhi} and \eqref{eq:AnTheta}. By an direct calculations from the limit system \eqref{eq:hydrolimit} show that $(u_\phi, v_\phi, b_\phi, c_\phi)$ verify the following system

\begin{equation} \label{S1eq8}
	\left\{ \;
	\begin{aligned}
		&\displaystyle \p_tu_{\phi}+ \lambda \dot{\theta}(t)|D_x| u_{\phi}+(u\p_x u)_{\phi} + (v\p_yu)_{\phi}-\p_y^2u_{\phi} + \p_xp_{\phi}=(b\p_x b)_{\phi} + (c\p_yb)_{\phi} \\
		&\displaystyle\p_yp_{\phi}= 0\\
		&\p_t b_{\phi}+\lambda \dot{\theta}(t)|D_x| b_{\phi}+(u\pa_x b)_{\phi}+(v\pa_y b)_{\phi}-\pa_y^2 b_{\phi}=(b\p_x u)_{\phi} + (c\p_yu)_{\phi}\\
		&\displaystyle \p_xu_{\phi}+\p_yv_{\phi}=0,\\
		&\displaystyle \p_xb_{\phi}+\p_yc_{\phi}=0,\\
		&\displaystyle u_{\phi}|_{t=0}=u_0, \\
		&\displaystyle b_{\phi}|_{t=0}=b_0,
	\end{aligned}
	\right.
\end{equation}

where $|D_x|$ denotes the Fourier multiplier of symbol $|\xi|$. In what follows, we recall that we use ``$C$'' to denote a generic positive constant which can change from line to line.

Applying the dyadic operator $\DD^h_q$ to the system \eqref{S1eq8}, then taking the $L^2(\Scal)$ scalar product of the first and the third equations of the obtained system with $\Delta_q^h u_{\phi}$ and $\Delta_q^h b_{\phi}$ respectively, we get

\begin{multline*}
	  \psca{ \Delta_q^h \pa_t  u_{\phi}, \Delta_q^h u_{\phi} }_{L^2} + \lambda \dot{\theta}(t) \psca{|D_x| \Delta_q^h u_{\phi},\Delta_q^h u_{\phi}}_{L^2} - \psca{\Delta_q^h \pa_y^2 u_{\phi},\Delta_q^h u_\phi }_{L^2}+ \psca{\Delta_q^h \pa_x p_\phi, \Delta_q^h u_\phi}_{L^2}\\
	= -\psca{\Delta_q^h (u\p_x u)_{\phi}, \Delta_q^h u_{\phi})}_{L^2} - \psca{\Delta_q^h (v\p_y u)_{\phi}, \Delta_q^h u_{\phi}}_{L^2} + \psca{\Delta_q^h (b\p_x b + c\p_yb)_{\phi},\Delta_q^h u_\phi},  \\
\end{multline*}
and
\begin{multline*}
 \psca{ \Delta_q^h \pa_t  b_{\phi}, \Delta_q^h b_{\phi} }_{L^2} + \lambda \dot{\theta}(t) \psca{|D_x| \Delta_q^h b_{\phi},\Delta_q^h b_{\phi}}_{L^2} - \psca{ \Delta_q^h \pa_y^2 b_{\phi},\Delta_q^h b_\phi }_{L^2}  \\
	= -\psca{\Delta_q^h (u\p_x b)_{\phi}, \Delta_q^h b_{\phi}}_{L^2} - \psca{\Delta_q^h (v\p_y b)_{\phi}, \Delta_q^h b_{\phi}}_{L^2} + \psca{\Delta_q^h(b\p_x u + c\p_yu)_{\phi}, \Delta_q^h b_\phi}. \hspace{1cm}
\end{multline*}

Thanks to the Dirichlet boundary condition and due to the free divergence of $U$ (its mean that $\divv U = \pa_xu +\pa_y v =0$ ), we get by using the integration by part that 
\begin{align*}
    \psca{\Delta_q^h \pa_x p_\phi, \Delta_q^h u_\phi}_{L^2} &= - \psca{\Delta_q^h p_\phi, \Delta_q^h \pa_x u_\phi}_{L^2}  \\
    & = \psca{\Delta_q^h p_\phi, \Delta_q^h \pa_y v_\phi}_{L^2} \\
    &= - \psca{\Delta_q^h \pa_y p_\phi, \Delta_q^h v_\phi}_{L^2} =0. \ \ \  (\text{ because } \pa_y p_\phi =0) 
\end{align*}
we recall that we have by integrating by part that 
\begin{align*}
    \psca{\Delta_q^h \pa_y^2 u_{\phi},\Delta_q^h u_\phi }_{L^2} & = - \psca{\Delta_q^h \pa_y u_{\phi},\Delta_q^h \pa_y u_\phi }_{L^2} = - \norm{\Delta_q^h \pa_y u_\phi}_{L^2}^2 \\
    \psca{ \Delta_q^h \pa_y^2 b_{\phi},\Delta_q^h b_\phi }_{L^2} & = -\psca{ \Delta_q^h \pa_y b_{\phi},\Delta_q^h \pa_y b_\phi }_{L^2} = - \norm{\Delta_q^h \pa_y b_\phi}_{L^2}^2
\end{align*}
we replace in the obtained estimate we get  
\begin{multline} \label{S4eq5}
	\frac{1}{2} \frac{d}{dt} \Vert \Delta_q^h u_{\phi}(t) \Vert_{L^2}^2 + \lambda \dot{\theta}(t) \norm{|D_x|^{\frac12} \Delta_q^h u_{\phi}}_{L^2}^2 + \Vert \Delta_q^h \pa_y u_{\phi}(t) \Vert_{L^2}^2 
	= -\psca{\Delta_q^h (u\p_x u)_{\phi}, \Delta_q^h u_{\phi})}_{L^2} \\ - \psca{\Delta_q^h (v\p_y u)_{\phi}, \Delta_q^h u_{\phi}}_{L^2} +  \psca{\Delta_q^h (b\p_x b)_{\phi},\Delta_q^h u_\phi} + \psca{\Delta_q^h ( c\p_yb)_{\phi},\Delta_q^h u_\phi},
\end{multline}
and
\begin{multline} \label{S4eq6}
	\frac{1}{2} \frac{d}{dt} \Vert \Delta_q^h b_{\phi}(t) \Vert_{L^2}^2 + \lambda \dot{\theta}(t) \norm{|D_x|^{\frac12} \Delta_q^h b_{\phi}}_{L^2}^2 + \Vert \Delta_q^h \pa_y b_{\phi}(t) \Vert_{L^2}^2 
	= -\psca{\Delta_q^h (u\p_x b)_{\phi}, \Delta_q^h b_{\phi}}_{L^2} \\ - \psca{\Delta_q^h (v\p_y b)_{\phi}, \Delta_q^h b_{\phi}}_{L^2} + \psca{\Delta_q^h(b\p_x u)_{\phi}, \Delta_q^h b_\phi} +\psca{\Delta_q^h( c\p_yu)_{\phi}, \Delta_q^h b_\phi}. \hspace{1cm}
\end{multline}
Multiplying \eqref{S4eq5} and \eqref{S4eq6} with $e^{2\Rcal t}$ and then integrating with respect to the time variable, we have
\begin{multline} \label{S4eq5bis}
	\norm{e^{\Rcal t} \Delta_q^h u_{\phi}(t)}_{L^\infty_t (L^2)}^2 + \lambda \int_0^t \dot{\theta}(t') \norm{e^{\Rcal t'} |D_x|^{\frac12} \Delta_q^h u_{\phi}}_{L^2}^2 dt' + \norm{e^{\Rcal t} \Delta_q^h \pa_y u_{\phi}(t)}_{L^2_t(L^2)}^2 \\
	= \norm{\Delta_q^h u_{\phi}(0)}_{L^2}^2 + I_1 + I_2 + I_3 + I_4, \hspace{2cm}
\end{multline}
and
\begin{multline} \label{S4eq6bis}
	\norm{e^{\Rcal t} \Delta_q^h b_{\phi}(t)}_{L^\infty_t (L^2)}^2 + \lambda \int_0^t \dot{\theta}(t') \norm{e^{\Rcal t} |D_x|^{\frac12} \Delta_q^h b_{\phi}}_{L^2}^2 dt' + \norm{e^{\Rcal t} \Delta_q^h \pa_y^2 b_{\phi}(t)}_{L^2_t(L^2)}^2\\
	= \norm{\Delta_q^h b_{\phi}(0)}_{L^2}^2 + D_1 + D_2+ D_3+D_4. \hspace{2.5cm}
\end{multline}

Next, by using the Lemmas \ref{lem C+B+A}, \ref{lem:ApaxA} and \ref{lem:cbb} yield that
\begin{align*}
	\abs{I_1} &= \abs{\int_0^t \psca{e^{\Rcal t'}\Delta_q^h (u\p_x u)_{\phi}, e^{\Rcal t'}\Delta_q^h u_{\phi})} dt'} \leq C d_q^2 2^{-2qs} \Vert e^{\mathcal{R}t} u_{\phi} \Vert_{\tilde{L}^2_{t,\dot{\theta}(t)}(\mathcal{B}^{s+\frac{1}{2}})}^2\\
	\abs{I_2} &= \abs{\int_0^t \psca{e^{\Rcal t'}\Delta_q^h (v\p_y u)_{\phi}, e^{\Rcal t'}\Delta_q^h u_{\phi}} dt'} \leq C d_q^2 2^{-2qs} \Vert e^{\mathcal{R}t} u_{\phi} \Vert_{\tilde{L}^2_{t,\dot{\theta}(t)}(\mathcal{B}^{s+\frac{1}{2}})}^2\\
	\abs{I_3} &= \abs{\int_0^t \psca{e^{\Rcal t'}\Delta_q^h (b\p_x b)_{\phi}, e^{\Rcal t'}\Delta_q^h u_{\phi}} dt'} \leq C d_q^2 2^{-2qs} \norm{e^{\Rcal t} u_\phi}_{\tilde{L}^2_{t,\dot{\theta}(t)}(\mathcal{B}^{s+\f12})} \norm{e^{\Rcal t} b_\phi}_{\tilde{L}^2_{t,\dot{\theta}(t)}(\mathcal{B}^{s+\f12})}\\
	\abs{I_4} &= \abs{\int_0^t \psca{e^{\Rcal t'}\Delta_q^h (c\p_y b)_{\phi}, e^{\Rcal t'}\Delta_q^h u_{\phi}} dt'} \leq C d_q^2 2^{-2qs} \norm{e^{\Rcal t} u_\phi}_{\tilde{L}^2_{t,\dot{\theta}(t)}(\mathcal{B}^{s+\f12})} \norm{e^{\Rcal t} b_\phi}_{\tilde{L}^2_{t,\dot{\theta}(t)}(\mathcal{B}^{s+\f12})}
\end{align*}
and
\begin{align*}
	\abs{D_1} &= \abs{\int_0^t \psca{e^{\Rcal t'}\Delta_q^h (u\p_x b)_{\phi}, e^{\Rcal t'}\Delta_q^h b_{\phi}} dt'} \leq C d_q^2 2^{-2qs} \Vert e^{\Rcal t} b_{\phi} \Vert_{\tilde{L}^2_{t,\dot{\theta}(t)}(\mathcal{B}^{s+\frac{1}{2}})}^2\\
	\abs{D_2} &= \abs{\int_0^t \psca{e^{\Rcal t'}\Delta_q^h (v\p_y b)_{\phi}, e^{\Rcal t'}\Delta_q^h b_{\phi}} dt'} \leq C d_q^2 2^{-2qs} \norm{e^{\Rcal t} u_{\phi}}_{\tilde{L}^2_{t(,\dot{\theta}(t))}(\mathcal{B}^{s+\f12})} \norm{e^{\Rcal t} b_\phi}_{\tilde{L}^2_{t(,\dot{\theta}(t))}(\mathcal{B}^{s+\f12})}\\
	\abs{D_3} &= \abs{\int_0^t \psca{e^{\Rcal t'}\Delta_q^h (b\p_x u)_{\phi}, e^{\Rcal t'}\Delta_q^h b_{\phi}} dt'} \leq C d_q^2 2^{-2qs} \norm{e^{\Rcal t} u_\phi}_{\tilde{L}^2_{t,\dot{\theta}(t)}(\mathcal{B}^{s+\f12})} \norm{e^{\Rcal t} b_\phi}_{\tilde{L}^2_{t,\dot{\theta}(t)}(\mathcal{B}^{s+\f12})}\\
	\abs{D_4} &= \abs{\int_0^t \psca{e^{\Rcal t'}\Delta_q^h (c\p_y u)_{\phi}, e^{\Rcal t'}\Delta_q^h b_{\phi}} dt'} \leq C d_q^2 2^{-2qs} \norm{e^{\Rcal t} b_\phi}_{\tilde{L}^2_{t,\dot{\theta}(t)}(\mathcal{B}^{s+\f12})}^2 .
\end{align*}

Multiplying \eqref{S4eq5bis} and \eqref{S4eq6bis} by $2^{2qs}$ and summing with respect to $q \in \ZZ$, we obtain 
\begin{multline} \label{S4eq5ter}
	\norm{e^{\Rcal t} u_{\phi}}_{\tilde{L}^\infty_t(\Bcal^s)}^2 + \lambda \norm{e^{\Rcal t} u_{\phi}}_{\tilde{L}^2_{t,\dot{\theta}(t)} (\Bcal^{s + \frac12})}^2 + \norm{e^{\Rcal t} \pa_y u_{\phi}}_{\tilde{L}^2_t(\Bcal^s)}^2 \\
	\leq \norm{u_{\phi}(0)}_{\Bcal^s}^2 + C \Vert e^{\mathcal{R}t} u_{\phi} \Vert_{\tilde{L}^2_{t,\dot{\theta}(t)}(\mathcal{B}^{s+\frac{1}{2}})}^2 + C \norm{e^{\Rcal t} (u,b)_\phi}_{\tilde{L}^2_{t,\dot{\theta}(t)}(\mathcal{B}^{s+\frac{1}{2}})}^2 ,
\end{multline}
and
\begin{multline} \label{S4eq6ter}
	\norm{e^{\Rcal t} b_{\phi}}_{\tilde{L}^\infty_t(\Bcal^s)}^2 + \lambda \norm{e^{\Rcal t} b_{\phi}}_{\tilde{L}^2_{t,\dot{\theta}(t)} (\Bcal^{s + \frac12})}^2 + \norm{e^{\Rcal t} \pa_y b_{\phi}}_{\tilde{L}^2_t(\Bcal^s)}^2\\
	\leq \norm{b_{\phi}(0)}_{\Bcal^s}^2 + C \Vert e^{\Rcal t} b_{\phi} \Vert_{\tilde{L}^2_{t,\dot{\theta}(t)}(\mathcal{B}^{s+\frac{1}{2}})}^2 + C \norm{e^{\Rcal t} u_\phi}_{\tilde{L}^2_{t,\dot{\theta}(t)}(\mathcal{B}^{s+\f12})} \norm{e^{\Rcal t} b_\phi}_{\tilde{L}^2_{t,\dot{\theta}(t)}(\mathcal{B}^{s+\f12})}.
\end{multline}

Thus, choosing 
\begin{equation}
	\label{eq:C} C \geq \max\set{4,\frac{1}{2\Rcal}},
\end{equation}
and taking the sum of \eqref{S4eq5ter} and \eqref{S4eq6ter}, we have
\begin{multline*}
	\norm{e^{\Rcal t} \pare{u_{\phi},b_{\phi}}}_{\tilde{L}^\infty_t(\Bcal^s)}^2 + \lambda \norm{e^{\Rcal t} \pare{u_{\phi},b_{\phi}}}_{\tilde{L}^2_{t,\dot{\theta}(t)} (\Bcal^{s + \frac12})}^2 + \frac{1}{2} \norm{e^{\Rcal t} \pa_y u_{\phi}}_{\tilde{L}^2_t(\Bcal^s)}^2 + \norm{e^{\Rcal t} \pa_y b_{\phi}}_{\tilde{L}^2_t(\Bcal^s)}^2\\
	\leq 2C \norm{e^{a\abs{D_x}} (u_0,b_0)}_{\mathcal{B}^{s}}^2 + 2C^2 \norm{e^{\mathcal{R}t} \pare{u_{\phi},b_\phi}}_{\tilde{L}^2_{t,\dot{\theta}(t)}(\mathcal{B}^{s+\frac{1}{2}})}^2.
\end{multline*}

We get the square root of our result equation we obtain

\begin{multline}
	\label{eq:u+T}
	\norm{e^{\Rcal t} \pare{u_{\phi},b_{\phi}}}_{\tilde{L}^\infty_t(\Bcal^s)} + \sqrt{\lambda} \norm{e^{\Rcal t} \pare{u_{\phi},b_{\phi}}}_{\tilde{L}^2_{t,\dot{\theta}(t)} (\Bcal^{s + \frac12})} + \frac{1}{2} \norm{e^{\Rcal t} \pa_y u_{\phi}}_{\tilde{L}^2_t(\Bcal^s)} + \norm{e^{\Rcal t} \pa_y b_{\phi}}_{\tilde{L}^2_t(\Bcal^s)}\\
	\leq C \norm{e^{a\abs{D_x}} (u_0,b_0)}_{\mathcal{B}^{s}} + 2C \norm{e^{\mathcal{R}t} \pare{u_{\phi},b_\phi}}_{\tilde{L}^2_{t,\dot{\theta}(t)}(\mathcal{B}^{s+\frac{1}{2}})}.
\end{multline}

We set 
\begin{equation} \label{eq:Tstar} 
	T^\star \stackrel{\tiny def}{=} \sup\set{t>0\ : \ \norm{u_{\phi}}_{\mathcal{B}^{\frac{1}{2}}}+ \norm{b_{\phi}}_{\mathcal{B}^{\frac{1}{2}}}  \leq \frac{1}{2C^2} \ \mbox{ and } \ \theta(t) \leq  \frac{a}{\lambda}}.
\end{equation}
We choose initial data such that 
\begin{align*}
	C\left( \Vert e^{a|D_x|} u_0 \Vert_{\mathcal{B}^{\frac12}} + \Vert e^{a|D_x|} b_0 \Vert_{\mathcal{B}^{\frac12}} \right) < \min\set{\frac{1}{2C^2},\frac{a}{2\lambda}},
\end{align*}
then, combining with the fact that $\theta(0) = 0$, we deduce that $T^\star > 0$. We choose now $\sqrt{\lambda} = 2 C$. For any $0 < t < T^\star$, we deduce from \eqref{eq:u+T} that
\begin{align}
	\label{eq:u+T02}
	\norm{e^{\Rcal t} \pare{u_{\phi},b_{\phi}}}_{\tilde{L}^\infty_t(\Bcal^s)} + \norm{e^{\Rcal t} \pa_y (u_{\phi},b_\phi) }_{\tilde{L}^2_t(\Bcal^s)} \leq C \norm{e^{a\abs{D_x}} (u_0,b_0)}_{\mathcal{B}^{s}}.
\end{align}
We then deduce from \eqref{eq:u+T02}, using \eqref{eq:C}, that, for any $0 < t < T^\star$, 
\begin{equation*}
	\norm{u_{\phi}}_{\mathcal{B}^{\frac{1}{2}}} \leq \Vert e^{\Rcal t} \pare{u_{\phi},b_{\phi}}\Vert_{\tilde{L}^\infty_t(\Bcal^s)} \leq C \norm{e^{a\abs{D_x}} (u_0,b_0)}_{\mathcal{B}^{s}} \leq C\left( \Vert e^{a|D_x|} u_0 \Vert_{\mathcal{B}^{\frac12}} + \Vert e^{a|D_x|} b_0 \Vert_{\mathcal{B}^{\frac12}} \right) < \frac{1}{2C^2}.
\end{equation*} 
Now, we recall that we already defined $\dot{\theta}(t) = \Vert \pa_y (u_{\phi},b_\phi)(t) \Vert_{\mathcal{B}^{\frac{1}{2}}}$ with $\theta(0) = 0$. Then, for any $0 < t < T^\star$, Inequality \eqref{eq:u+T02} yields
\begin{align*}
	\theta (t) &= \int_0^t \Vert \pa_y (u_{\phi},b_\phi)(t) \Vert_{\mathcal{B}^\f12} dt' \\ &\leq \int_0^t  e^{-\mathcal{R}t'} \Vert e^{\mathcal{R}t'} \pa_y (u_{\phi},b_\phi)(t') \Vert_{\mathcal{B}^\f12} dt' \\
	&\leq \left( \int_0^t  e^{-2\mathcal{R}t'} dt' \right)^{\f12} \left( \int_0^t \Vert e^{\mathcal{R}t'} \pa_y (u_{\phi},b_\phi)(t') \Vert_{\mathcal{B}^\f12}^2  dt' \right)^{\f12} \\
	&\leq C \left( \Vert e^{a|D_x|} u_0 \Vert_{\mathcal{B}^\f12} +  \Vert e^{a|D_x|}  b_{0} \Vert_{\mathcal{B}^\f12}\right) \\ &< \frac{a}{2\lambda}.
\end{align*}
A continuity argument implies that $T^\star = +\infty$ and we have \eqref{eq:u+T02} is valid for any $t \in \RR_+$.
\begin{proposition}\label{prop}
Let $a >0$. If $e^{a|D_x|}u_0 \in \mathcal{B}^{\f52} $, $ e^{a|D_x|} \pa_y u_0 \in \mathcal{B}^{\f32}$, $e^{a|D_x|}b_0 \in \mathcal{B}^{\f32}$ and 
\begin{align} \label{1.11} 
		\Vert e^{a|D_x|} u_0 \Vert_{\mathcal{B}^\f12} \leq  \frac{c_1 a}{1+  \Vert e^{a|D_x|} u_0 \Vert_{\mathcal{B}^\f32} + \Vert e^{a|D_x|}  b_0\Vert_{\mathcal{B}^\f32}} 
	\end{align}
	for some $c_1$ sufficiently small, then there exists a positive constant $C$ so that for $\lambda = C^2(1+  \Vert e^{a|D_x|}  u_0\Vert_{\mathcal{B}^\f32} +  \Vert e^{a|D_x|}  b_0\Vert_{\mathcal{B}^\f32}),$ and $1\leq s \leq \f52,$ one has 
	
		\begin{align}
		\label{eq:MhdlimitU}
	\norm{e^{\Rcal t} \pare{u_{\phi},b_{\phi}}}_{\tilde{L}^\infty(\mathbb{R}^+,\Bcal^s)} +  \norm{e^{\Rcal t} \pa_y (u_{\phi},b_\phi) }_{\tilde{L}^2(\mathbb{R}^+,\Bcal^s)} \leq 2C \norm{e^{a\abs{D_x}} (u_0,b_0)}_{\mathcal{B}^{s}}.
	\end{align}
	
	\begin{align}
		\label{eq:hydrolimitdtU}
		\Vert e^{\mathcal{R}t} (\pa_tu)_\phi \Vert_{\tilde{L}^2_t(\mathcal{B}^\f32)} +\Vert e^{\mathcal{R}t} (\pa_tb)_\phi \Vert_{\tilde{L}^2_t(\mathcal{B}^\f32)} &+\Vert e^{\mathcal{R}t}  \pa_y (u_\phi,b_\phi) \Vert_{\tilde{L}^\infty_t(\mathcal{B}^\f32)} \\ &\lesssim C\Big( \Vert e^{a|D_x|}\pa_y (u_0,b_0) \Vert_{\mathcal{B}^\f32} + \Vert e^{a|D_x|} (u_0,b_0) \Vert_{\mathcal{B}^\f52}  \Big). \notag
	\end{align}
\end{proposition}
\begin{proof}
We start by giving proof of \eqref{eq:MhdlimitU}. In the same way, we have by applying the dyadic operator to the system \eqref{S1eq8} and  taking the $L^2$ inner product, then integrating with respect to the time variable, we have
\begin{multline} \label{S4eq5bisbis}
	\norm{e^{\Rcal t} \Delta_q^h u_{\phi}(t)}_{L^\infty_t (L^2)}^2 + \lambda \int_0^t \dot{\theta}(t') \norm{e^{\Rcal t'} |D_x|^{\frac12} \Delta_q^h u_{\phi}}_{L^2}^2 dt' + \norm{e^{\Rcal t} \Delta_q^h \pa_y u_{\phi}(t)}_{L^2_t(L^2)}^2 \\
	= \norm{\Delta_q^h u_{\phi}(0)}_{L^2}^2 + I_1 + I_2, \hspace{2cm}
\end{multline}
and
\begin{multline} \label{S4eq6bisbis}
	\norm{e^{\Rcal t} \Delta_q^h b_{\phi}(t)}_{L^\infty_t (L^2)}^2 + \lambda \int_0^t \dot{\theta}(t') \norm{e^{\Rcal t} |D_x|^{\frac12} \Delta_q^h b_{\phi}}_{L^2}^2 dt' + \norm{e^{\Rcal t} \Delta_q^h \pa_y^2 b_{\phi}(t)}_{L^2_t(L^2)}^2\\
	= \norm{\Delta_q^h b_{\phi}(0)}_{L^2}^2 + I_3+I_4. \hspace{2.5cm}
\end{multline}
 where 
 \begin{align*}
 \abs{I_1} &= \abs{\int_0^t \psca{e^{\Rcal t'}\Delta_q^h (u\p_x u+ v\pa_yu)_{\phi}, e^{\Rcal t'}\Delta_q^h u_{\phi})} dt'} \\
	\abs{I_2} &= \abs{\int_0^t \psca{e^{\Rcal t'}\Delta_q^h (b\p_x b+c\pa_yb)_{\phi}, e^{\Rcal t'}\Delta_q^h u_{\phi}} dt'} \\
	\abs{I_3} &= \abs{\int_0^t \psca{e^{\Rcal t'}\Delta_q^h (u\p_x b+ v\pa_yb)_{\phi}, e^{\Rcal t'}\Delta_q^h b_{\phi}} dt'} \\
	\abs{I_4} & = \abs{\int_0^t \psca{e^{\Rcal t'}\Delta_q^h (b\p_x u+ c\pa_yu)_{\phi}, e^{\Rcal t'}\Delta_q^h b_{\phi}} dt'}
 \end{align*}
 All the estimate getting in the proof of the theorem $\ref{th:1} $,  satisfied when $s\in]0,1]$, so we need now give the proof of all this estimate when our $s>0$. We first remark from the proof of the lemma \ref{lem C+B+A}, that for any $s>0$ we have 
 $$ \int_0^t \abs{ \psca{e^{\Rcal t'}\Delta_q^h (T^h_{u}\pa_xu +R^h(u,\p_x u))_{\phi}, e^{\Rcal t'}\Delta_q^h u_{\phi})} } dt' \leq C 2^{-2qs}d_q^2 \Vert e^{\mathcal{R}t} u_\phi \Vert_{\tilde{L}^2_{t,\dot{\theta}(t)}(\mathcal{B}^{s+\frac{1}{2}})}^2, $$
 then we only need to prove that 
 $$ \int_0^t \abs{ \psca{e^{\Rcal t'}\Delta_q^h (T^h_{\pa_xu}u)_{\phi}, e^{\Rcal t'}\Delta_q^h u_{\phi})} } dt' \leq C 2^{-2qs}d_q^2 \Vert e^{\mathcal{R}t} u_\phi \Vert_{\tilde{L}^2_{t,\dot{\theta}(t)}(\mathcal{B}^{s+\frac{1}{2}})}^2, \text{ \ \ for \ \ } s>0. $$
 Indeed we recall that 
 \begin{align} \label{eq:normu1}
	\Vert \Delta^h_q u_{\phi}(t') \Vert_{L^{\infty}} \lesssim 2^{\frac{q}{2}} \Vert \Delta_q^h u_{\phi}(t') \Vert_{L^2_h(L^{\infty}_v)} \lesssim 2^{\frac{q}{2}} \Vert \Delta_q^h \pa_y u_{\phi}(t') \Vert_{L^2} \lesssim d_q(u_\phi) \Vert  \pa_y u_{\phi}(t') \Vert_{\mathcal{B}^{\f12}},
\end{align}
then we infer,
\begin{align*}
 \int_0^t \abs{ \psca{e^{\Rcal t'}\Delta_q^h (T^h_{\pa_xu}u)_{\phi}, e^{\Rcal t'}\Delta_q^h u_{\phi})} } dt' &\lesssim \sum_{|q'-q|\leq 4} \int_0^t e^{2\Rcal t'} \Vert S^h_{q'-1} \pa_xu_\phi \Vert_{L^\infty} \Vert \Delta_{q'}^h u_\phi \Vert_{L^2} \Vert \Delta_{q}^h u_\phi \Vert_{L^2} dt'\\
 &\lesssim \sum_{|q'-q|\leq 4} 2^{q'} \int_0^t e^{2\Rcal t'} \Vert \pa_yu_\phi \Vert_{\mathcal{B}^\f12} \Vert \Delta_{q'}^h u_\phi \Vert_{L^2} \Vert \Delta_{q}^h u_\phi \Vert_{L^2} dt'\\
 &\lesssim C2^{-2qs} d_q^2 \Vert e^{\Rcal t}u_\phi \Vert_{\tilde{L}^2_{t,\dot{\theta}(t)}(\mathcal{B}^{s+\frac{1}{2}})}^2.
\end{align*}
While it follows from the proof of Lemma \ref{lem:ApaxA} that
 $$ \int_0^t \abs{ \psca{e^{\Rcal t'}\Delta_q^h (T^h_{\pa_yu}v +R^h(v,\p_y u))_{\phi}, e^{\Rcal t'}\Delta_q^h u_{\phi})} } dt' \leq C 2^{-2qs}d_q^2 \Vert e^{\mathcal{R}t} u_\phi \Vert_{\tilde{L}^2_{t,\dot{\theta}(t)}(\mathcal{B}^{s+\frac{1}{2}})}^2, $$
 In views at the proof of the lemme \ref{lem:ApaxA} (the estimate \eqref{eq:normK1}), we have 
 $$ \Vert \Delta_q^h v_\phi(t) \Vert_{L^\infty} \leq d_q 2^{\frac{q}{2}} \Vert u_\phi \Vert_{\mathcal{B}^\f32}^\f12 \Vert \pa_y u_\phi \Vert_{\mathcal{B}^\f12}^\f12 $$
 so we replace in the inequality of the term $T^h_{v} \pa_yu$, we obtain 
 \begin{align*}
  \int_0^t \abs{ \psca{e^{\Rcal t'}\Delta_q^h (T^h_{v}\pa_yu)_{\phi}, e^{\Rcal t'}\Delta_q^h u_{\phi})} } dt' &\lesssim \sum_{|q'-q|\leq 4} \int_0^t e^{2\Rcal t'} \Vert S^h_{q'-1} v_\phi \Vert_{L^\infty} \Vert \Delta_{q'}^h \pa_y u_\phi \Vert_{L^2} \Vert \Delta_{q}^h u_\phi \Vert_{L^2} dt'\\
 &\lesssim \sum_{|q'-q|\leq 4} 2^{\frac{q'}{2}} \int_0^t e^{2\Rcal t'} \Vert u_\phi \Vert_{\mathcal{B}^\f32}^\f12 \Vert \pa_yu_\phi \Vert_{\mathcal{B}^\f12}^\f12 \Vert \Delta_{q'}^h \pa_yu_\phi \Vert_{L^2} \Vert \Delta_{q}^h u_\phi \Vert_{L^2} dt'\\
 &\lesssim C2^{-2qs} d_q^2 \Vert u_\phi \Vert_{\tilde{L}^\infty_t(\mathcal{B}^\f32)}^\f12 \Vert e^{\Rcal t}\pa_y u_\phi \Vert_{\tilde{L}^2_t(\mathcal{B}^s)} \Vert e^{\Rcal t}u_\phi \Vert_{\tilde{L}^2_{t,\dot{\theta}(t)}(\mathcal{B}^{s+\frac{1}{2}})}.
 \end{align*}
 As a result, it comes out 
 \begin{align*}
   \int_0^t &\abs{\psca{e^{\Rcal t'}\Delta_q^h (u\p_x u+ v\pa_yu)_{\phi}, e^{\Rcal t'}\Delta_q^h u_{\phi})}} dt' \lesssim C 2^{-2qs}d_q^2 \Vert e^{\mathcal{R}t} u_\phi \Vert_{\tilde{L}^2_{t,\dot{\theta}(t)}(\mathcal{B}^{s+\frac{1}{2}})} \\
   & \times \left( \Vert e^{\mathcal{R}t} u_\phi \Vert_{\tilde{L}^2_{t,\dot{\theta}(t)}(\mathcal{B}^{s+\frac{1}{2}})} + \Vert u_\phi \Vert_{\tilde{L}^\infty_t(\mathcal{B}^\f32)}^\f12 \Vert e^{\Rcal t}\pa_y u_\phi \Vert_{\tilde{L}^2_t(\mathcal{B}^s)} \right).
 \end{align*}
    Now, we will get the estimate of the second term $I_2$. In the same way, we remark from the proof of the lemma \ref{lem C+B+A}, that we only need to prove 
    $$ \int_0^t \abs{ \psca{e^{\Rcal t'}\Delta_q^h (T^h_{\pa_xb}b)_{\phi}, e^{\Rcal t'}\Delta_q^h u_{\phi})} } dt' \leq C 2^{-2qs}d_q^2 \Vert e^{\mathcal{R}t} b_\phi \Vert_{\tilde{L}^2_{t,\dot{\theta}(t)}(\mathcal{B}^{s+\frac{1}{2}})}\Vert e^{\mathcal{R}t} b_\phi \Vert_{\tilde{L}^2_{t,\dot{\theta}(t)}(\mathcal{B}^{s+\frac{1}{2}})}, \text{ \ \ for \ \ } s>0. $$
    So, in views of \eqref{eq:normu1}, we infer 
    \begin{align*}
 \int_0^t \abs{ \psca{e^{\Rcal t'}\Delta_q^h (T^h_{\pa_xb}b)_{\phi}, e^{\Rcal t'}\Delta_q^h u_{\phi})} } dt' &\lesssim \sum_{|q'-q|\leq 4} \int_0^t e^{2\Rcal t'} \Vert S^h_{q'-1} \pa_xb_\phi \Vert_{L^\infty} \Vert \Delta_{q'}^h b_\phi \Vert_{L^2} \Vert \Delta_{q}^h u_\phi \Vert_{L^2} dt'\\
 &\lesssim \sum_{|q'-q|\leq 4} 2^{q'} \int_0^t e^{2\Rcal t'} \Vert \pa_yb_\phi \Vert_{\mathcal{B}^\f12} \Vert \Delta_{q'}^h b_\phi \Vert_{L^2} \Vert \Delta_{q}^h u_\phi \Vert_{L^2} dt'\\
 &\lesssim C2^{-2qs} d_q^2 \Vert e^{\Rcal t}b_\phi \Vert_{\tilde{L}^2_{t,\dot{\theta}(t)}(\mathcal{B}^{s+\frac{1}{2}})}  \Vert e^{\Rcal t}u_\phi \Vert_{\tilde{L}^2_{t,\dot{\theta}(t)}(\mathcal{B}^{s+\frac{1}{2}})}.
\end{align*}
 In views at the proof of the lemme \ref{lem:ApaxA} (the estimate \eqref{eq:normK1}), we have 
 $$ \Vert \Delta_q^h c_\phi(t) \Vert_{L^\infty} \leq d_q 2^{\frac{q}{2}} \Vert b_\phi \Vert_{\mathcal{B}^\f32}^\f12 \Vert \pa_y b_\phi \Vert_{\mathcal{B}^\f12}^\f12 $$
 so that there holds 
 \begin{align*}
  \int_0^t \abs{ \psca{e^{\Rcal t'}\Delta_q^h (T^h_{c}\pa_yb)_{\phi}, e^{\Rcal t'}\Delta_q^h u_{\phi})} } dt' &\lesssim \sum_{|q'-q|\leq 4} \int_0^t e^{2\Rcal t'} \Vert S^h_{q'-1} c_\phi \Vert_{L^\infty} \Vert \Delta_{q'}^h \pa_y b_\phi \Vert_{L^2} \Vert \Delta_{q}^h u_\phi \Vert_{L^2} dt'\\
 &\lesssim \sum_{|q'-q|\leq 4} 2^{\frac{q'}{2}} \int_0^t e^{2\Rcal t'} \Vert b_\phi \Vert_{\mathcal{B}^\f32}^\f12 \Vert \pa_yb_\phi \Vert_{\mathcal{B}^\f12}^\f12 \Vert \Delta_{q'}^h \pa_yb_\phi \Vert_{L^2} \Vert \Delta_{q}^h u_\phi \Vert_{L^2} dt'\\
 &\lesssim C2^{-2qs} d_q^2 \Vert b_\phi \Vert_{\tilde{L}^\infty_t(\mathcal{B}^\f32)}^\f12 \Vert e^{\Rcal t}\pa_y b_\phi \Vert_{\tilde{L}^2_t(\mathcal{B}^s)} \Vert e^{\Rcal t}u_\phi \Vert_{\tilde{L}^2_{t,\dot{\theta}(t)}(\mathcal{B}^{s+\frac{1}{2}})}.
 \end{align*}
 As a result, it comes out 
 \begin{align*}
   \int_0^t &\abs{\psca{e^{\Rcal t'}\Delta_q^h (b\p_x b+ c\pa_yb)_{\phi}, e^{\Rcal t'}\Delta_q^h u_{\phi})}} dt' \lesssim C 2^{-2qs}d_q^2 \Vert e^{\mathcal{R}t} u_\phi \Vert_{\tilde{L}^2_{t,\dot{\theta}(t)}(\mathcal{B}^{s+\frac{1}{2}})} \\
   & \times \left( \Vert e^{\mathcal{R}t} b_\phi \Vert_{\tilde{L}^2_{t,\dot{\theta}(t)}(\mathcal{B}^{s+\frac{1}{2}})} + \Vert b_\phi \Vert_{\tilde{L}^\infty_t(\mathcal{B}^\f32)}^\f12 \Vert e^{\Rcal t}\pa_y b_\phi \Vert_{\tilde{L}^2_t(\mathcal{B}^s)} \right).
   \end{align*}
   
   We do the same think to estimate $I_3$ and $I_4$, then we have 
   \begin{align*}
   I_3 + I_4 &\lesssim C 2^{-2qs}d_q^2 \Vert e^{\mathcal{R}t} u_\phi \Vert_{\tilde{L}^2_{t,\dot{\theta}(t)}(\mathcal{B}^{s+\frac{1}{2}})} \\
   &\times \left( \Vert e^{\mathcal{R}t} b_\phi \Vert_{\tilde{L}^2_{t,\dot{\theta}(t)}(\mathcal{B}^{s+\frac{1}{2}})} + \Vert b_\phi \Vert_{\tilde{L}^\infty_t(\mathcal{B}^\f32)}^\f12 \Vert e^{\Rcal t}\pa_y u_\phi \Vert_{\tilde{L}^2_t(\mathcal{B}^s)}+ \Vert u_\phi \Vert_{\tilde{L}^\infty_t(\mathcal{B}^\f32)}^\f12 \Vert e^{\Rcal t}\pa_y b_\phi \Vert_{\tilde{L}^2_t(\mathcal{B}^s)} \right)
   \end{align*}
   
   	So we replace all the results obtained in \ref{S4eq5bisbis} and \ref{S4eq6bisbis} and
   	Multiplying by $2^{2qs}$, then summing with respect to $q \in \ZZ$, we obtain 
\begin{align} \label{S4eq5terbis}
	\norm{e^{\Rcal t} u_{\phi}}_{\tilde{L}^\infty_t(\Bcal^s)}^2 &+ \lambda \norm{e^{\Rcal t} u_{\phi}}_{\tilde{L}^2_{t,\dot{\theta}(t)} (\Bcal^{s + \frac12})}^2 + \norm{e^{\Rcal t} \pa_y u_{\phi}}_{\tilde{L}^2_t(\Bcal^s)}^2
	\leq C\norm{u_{\phi}(0)}_{\Bcal^s}^2 + C\Vert e^{\mathcal{R}t} u_\phi \Vert_{\tilde{L}^2_{t,\dot{\theta}(t)}(\mathcal{B}^{s+\frac{1}{2}})} \\
   & \times \left( \Vert e^{\mathcal{R}t} b_\phi \Vert_{\tilde{L}^2_{t,\dot{\theta}(t)}(\mathcal{B}^{s+\frac{1}{2}})} + \Vert b_\phi \Vert_{\tilde{L}^\infty_t(\mathcal{B}^\f32)}^\f12 \Vert e^{\Rcal t}\pa_y b_\phi \Vert_{\tilde{L}^2_t(\mathcal{B}^s)}+ \Vert u_\phi \Vert_{\tilde{L}^\infty_t(\mathcal{B}^\f32)}^\f12 \Vert e^{\Rcal t}\pa_y u_\phi \Vert_{\tilde{L}^2_t(\mathcal{B}^s)} \right) \notag ,
\end{align}
and
\begin{align} \label{S4eq6terbis}
	\norm{e^{\Rcal t} b_{\phi}}_{\tilde{L}^\infty_t(\Bcal^s)}^2 &+ \lambda \norm{e^{\Rcal t} b_{\phi}}_{\tilde{L}^2_{t,\dot{\theta}(t)} (\Bcal^{s + \frac12})}^2 + \norm{e^{\Rcal t} \pa_y b_{\phi}}_{\tilde{L}^2_t(\Bcal^s)}^2
	\leq C\norm{b_{\phi}(0)}_{\Bcal^s}^2 + C\Vert e^{\mathcal{R}t} u_\phi \Vert_{\tilde{L}^2_{t,\dot{\theta}(t)}(\mathcal{B}^{s+\frac{1}{2}})} \\
   &\times \left( \Vert e^{\mathcal{R}t} b_\phi \Vert_{\tilde{L}^2_{t,\dot{\theta}(t)}(\mathcal{B}^{s+\frac{1}{2}})} + \Vert b_\phi \Vert_{\tilde{L}^\infty_t(\mathcal{B}^\f32)}^\f12 \Vert e^{\Rcal t}\pa_y u_\phi \Vert_{\tilde{L}^2_t(\mathcal{B}^s)}+ \Vert b_\phi \Vert_{\tilde{L}^\infty_t(\mathcal{B}^\f32)}^\f12 \Vert e^{\Rcal t}\pa_y u_\phi \Vert_{\tilde{L}^2_t(\mathcal{B}^s)} \right) \notag.
\end{align}

Thus, choosing 
\begin{equation}
	\label{eq:C} C \geq \max\set{4,\frac{1}{2\Rcal}},
\end{equation}
and taking the sum of \eqref{S4eq5terbis} and \eqref{S4eq6terbis}, we have
\begin{align*}
	\norm{e^{\Rcal t} \pare{u_{\phi},b_{\phi}}}_{\tilde{L}^\infty_t(\Bcal^s)}^2 &+ \lambda \norm{e^{\Rcal t} \pare{u_{\phi},b_{\phi}}}_{\tilde{L}^2_{t,\dot{\theta}(t)} (\Bcal^{s + \frac12})}^2 + \norm{e^{\Rcal t} \pa_y u_{\phi}}_{\tilde{L}^2_t(\Bcal^s)}^2
	\leq 2C \norm{e^{a\abs{D_x}} (u_0,b_0)}_{\mathcal{B}^{s}}^2 \\&+ C\Vert e^{\mathcal{R}t} u_\phi \Vert_{\tilde{L}^2_{t,\dot{\theta}(t)}(\mathcal{B}^{s+\frac{1}{2}})} 
  \times \Big( \Vert e^{\mathcal{R}t} b_\phi \Vert_{\tilde{L}^2_{t,\dot{\theta}(t)}(\mathcal{B}^{s+\frac{1}{2}})} \\ &+ \Vert b_\phi \Vert_{\tilde{L}^\infty_t(\mathcal{B}^\f32)}^\f12 \Vert e^{\Rcal t}\pa_y (b_\phi,u_\phi) \Vert_{\tilde{L}^2_t(\mathcal{B}^s)}+ \Vert u_\phi \Vert_{\tilde{L}^\infty_t(\mathcal{B}^\f32)}^\f12 \Vert e^{\Rcal t}\pa_y (u_\phi,b_\phi) \Vert_{\tilde{L}^2_t(\mathcal{B}^s)} \Big).
\end{align*}

Applying Young’s inequality yields
\begin{align*}
C\Vert b_\phi &\Vert_{\tilde{L}^\infty_t(\mathcal{B}^\f32)}^\f12 \Vert e^{\Rcal t}\pa_y (b_\phi,u_\phi) \Vert_{\tilde{L}^2_t(\mathcal{B}^s)}\Vert e^{\mathcal{R}t} u_\phi \Vert_{\tilde{L}^2_{t,\dot{\theta}(t)}(\mathcal{B}^{s+\frac{1}{2}})}\\ &\leq \Vert b_\phi \Vert_{\tilde{L}^\infty_t(\mathcal{B}^\f32)} \Vert e^{\mathcal{R}t} u_\phi \Vert_{\tilde{L}^2_{t,\dot{\theta}(t)}(\mathcal{B}^{s+\frac{1}{2}})}^2 + \frac{1}{2}  \Vert e^{\Rcal t}\pa_y (b_\phi,u_\phi) \Vert_{\tilde{L}^2_t(\mathcal{B}^s)}^2 
\end{align*}
and
\begin{align*}
C\Vert u_\phi &\Vert_{\tilde{L}^\infty_t(\mathcal{B}^\f32)}^\f12 \Vert e^{\Rcal t}\pa_y (u_\phi,b_\phi) \Vert_{\tilde{L}^2_t(\mathcal{B}^s)}\Vert e^{\mathcal{R}t} u_\phi \Vert_{\tilde{L}^2_{t,\dot{\theta}(t)}(\mathcal{B}^{s+\frac{1}{2}})}\\ &\leq \Vert u_\phi \Vert_{\tilde{L}^\infty_t(\mathcal{B}^\f32)} \Vert e^{\mathcal{R}t} u_\phi \Vert_{\tilde{L}^2_{t,\dot{\theta}(t)}(\mathcal{B}^{s+\frac{1}{2}})}^2 + \frac{1}{2}  \Vert e^{\Rcal t}\pa_y (b_\phi,u_\phi) \Vert_{\tilde{L}^2_t(\mathcal{B}^s)}^2.
\end{align*}
Therefore if we take
\begin{equation} \label{eq:lambda2}
\lambda \geq C\left(1+ \Vert u_\phi \Vert_{\tilde{L}^\infty_t(\mathcal{B}^\f32)} + \Vert b_\phi \Vert_{\tilde{L}^\infty_t(\mathcal{B}^\f32)}\right),
\end{equation}
then we obtain
		\begin{align}
		\label{eq:Mhdestimate}
	\norm{e^{\Rcal t} \pare{u_{\phi},b_{\phi}}}_{\tilde{L}^\infty_t(\Bcal^s)} +  \norm{e^{\Rcal t} \pa_y (u_{\phi},b_\phi) }_{\tilde{L}^2_t(\Bcal^s)} \leq 2C \norm{e^{a\abs{D_x}} (u_0,b_0)}_{\mathcal{B}^{s}}.
	\end{align}
	
which in particular implies that under the condition \eqref{eq:lambda2}, there holds
\begin{align*}
\Vert u_\phi \Vert_{\tilde{L}^\infty_t(\mathcal{B}^\f32)} + \Vert b_\phi \Vert_{\tilde{L}^\infty_t(\mathcal{B}^\f32)} \leq C \norm{e^{a\abs{D_x}} (u_0,b_0)}_{\mathcal{B}^{\f32}}
\end{align*} 
then by taking $\lambda = C(1+ \norm{e^{a\abs{D_x}} (u_0,b_0)}_{\mathcal{B}^{\f32}})$, \eqref{eq:lambda2} verified. So under the condition \eqref{1.11}, we have that \eqref{eq:lambda2} hold, and then \eqref{eq:Mhdestimate} is valid for any $t>0$. Thos complete the proof of the proposition.

Now we still have to prove the second estimate \eqref{eq:hydrolimitdtU}.. For that, we apply the dyadic operator $\Delta_q^h$ to \eqref{S1eq8} and taking the $L^2$ inner product of the resulting equation with $\Delta_q^h (\pa_tu)_\phi$ and $\Delta_q^h(\pa_tb)_\phi$. That yields 
\begin{multline*}
	  \psca{ \Delta_q^h (\pa_t u)_{\phi}, \Delta_q^h (\pa_tu)_\phi}_{L^2}  - \psca{\Delta_q^h \pa_y^2 u_{\phi},\Delta_q^h (\pa_tu)_\phi }_{L^2}+ \psca{\Delta_q^h \pa_x p_\phi, \Delta_q^h (\pa_tu)_\phi}_{L^2}\\
	= -\psca{\Delta_q^h (u\p_x u)_{\phi}, \Delta_q^h (\pa_tu)_\phi}_{L^2} - \psca{\Delta_q^h (v\p_y u)_{\phi}, \Delta_q^h (\pa_tu)_\phi}_{L^2} + \psca{\Delta_q^h (b\p_x b + c\p_yb)_{\phi},\Delta_q^h (\pa_tu)_\phi},  \\
\end{multline*}
and
\begin{multline*}
 \psca{ \Delta_q^h (\pa_t  b)_{\phi}, \Delta_q^h(\pa_tb)_\phi }_{L^2} - \psca{ \Delta_q^h \pa_y^2 b_{\phi},\Delta_q^h(\pa_tb)_\phi }_{L^2}  \\
	= -\psca{\Delta_q^h (u\p_x b)_{\phi}, \Delta_q^h(\pa_tb)_\phi}_{L^2} - \psca{\Delta_q^h (v\p_y b)_{\phi}, \Delta_q^h(\pa_tb)_\phi}_{L^2} + \psca{\Delta_q^h(b\p_x u + c\p_yu)_{\phi}, \Delta_q^h(\pa_tb)_\phi}. \hspace{1cm}
\end{multline*}
The fact that $(\pa_tu)_\phi = \pa_t u_\phi + \lambda \dot{\theta}(t) |D_x|u_\phi$ implies 
\begin{align*}
    \psca{ \Delta_q^h \pa_y^2 u_\phi,\Delta_q^h (\pa_tu)_\phi }_{L^2} &= -\pare{\f12 \frac{d}{dt} \Vert \Delta_q^h \pa_y u_\phi \Vert_{L^2}^2 + \lambda \dot{\theta}(t) 2^q \Vert \Delta_q^h \pa_y u_\phi \Vert_{L^2}^2}\\
    \psca{ \Delta_q^h \pa_y^2 b_\phi,\Delta_q^h (\pa_tb)_\phi }_{L^2} &= -\pare{\f12 \frac{d}{dt} \Vert \Delta_q^h \pa_yb_\phi \Vert_{L^2}^2 + \lambda \dot{\theta}(t) 2^q \Vert \Delta_q^h \pa_y b_\phi \Vert_{L^2}^2}
\end{align*}
from which, we deduce that 
\begin{align*}
\Vert \Delta_q^h (\pa_tu)_\phi \Vert_{L^2}^2 + \f12 \frac{d}{dt}\Vert \Delta_q^h \pa_y u_\phi \Vert_{L^2}^2 \leq  I_1 + I_2 + I_3,
\end{align*}
and
\begin{align*}
    \Vert \Delta_q^h (\pa_tb)_\phi \Vert_{L^2}^2 + \f12 \frac{d}{dt}\Vert \Delta_q^h \pa_y b_\phi \Vert_{L^2}^2 \leq  D_1 + D_2,
\end{align*}
where
\begin{align*}
I_1 &= \abs{\psca{ \Delta_q^h (u\pa_xu +b\pa_xb)_\phi,\Delta_q^h (\pa_tu)_\phi }_{L^2}}\\
I_2 &=  \abs{\psca{  \Delta_q^h (v\pa_yu+ c\pa_yb)_\phi,\Delta_q^h (\pa_tu)_\phi }_{L^2}}\\
I_3 &= \abs{\psca{  \Delta_q^h \pa_x p_\phi , \Delta_q^h (\pa_tu)_\phi }_{L^2}}.\\
D_1 &= \abs{\psca{ \Delta_q^h (u\pa_xb +b\pa_xu)_\phi,\Delta_q^h (\pa_tb)_\phi }_{L^2}}\\
D_2 &= \abs{\psca{ \Delta_q^h (v\pa_yb +c\pa_yu)_\phi,\Delta_q^h (\pa_tb)_\phi }_{L^2}}\\
\end{align*}

Since $\dd_x u + \dd_y v = 0$, using \eqref{vv} and integrations by parts, we find 
\begin{align*}
I_3 = \abs{\psca{  \Delta_q^h \pa_x p_\phi , \Delta_q^h (\pa_tu)_\phi }_{L^2}} &= \abs{\psca{  \Delta_q^h \pa_x p_\phi , \Delta_q^h (\pa_tu_\phi + \lambda \dot{\theta}(t) |D_x|u_\phi) }_{L^2}}\\
& = \abs{\psca{  \Delta_q^h p_\phi , \Delta_q^h (\pa_t \pa_x u_\phi + \lambda \dot{\theta}(t) |D_x| \pa_xu_\phi) }_{L^2}} \\
& =  \abs{\psca{  \Delta_q^h p_\phi , \Delta_q^h (\pa_t \pa_y v_\phi + \lambda \dot{\theta}(t) |D_x| \pa_y v_\phi) }_{L^2}} \\
& =  \abs{\psca{  \Delta_q^h \pa_y p_\phi , \Delta_q^h (\pa_t v)_\phi }_{L^2}} = 0 \ (\text{ \ because \ } \pa_y = 0)
\end{align*}

For $I_1$, $I_2$, $D_1$ and $D_2$ we have 
\begin{align*}
	I_1 = \abs{\psca{ \Delta_q^h (u\pa_xu+b\pa_xb)_\phi,\Delta_q^h (\pa_tu)_\phi }_{L^2}} &\leq \abs{\psca{  \Delta_q^h (u\pa_xu)_\phi,\Delta_q^h (\pa_tu)_\phi }_{L^2}} +\abs{\psca{  \Delta_q^h (b\pa_xb)_\phi,\Delta_q^h (\pa_tu)_\phi }_{L^2}} \\ &\leq \Vert \Delta_q^h (u\pa_xu)_\phi \Vert_{L^2}^2+ \Vert \Delta_q^h (b\pa_xb)_\phi \Vert_{L^2}^2  + \frac{1}{10} \Vert \Delta_q^h (\pa_tu)_\phi \Vert_{L^2}^2,
\end{align*}
and 
\begin{align*}
    I_2 &= \abs{\psca{  \Delta_q^h (v\pa_yu+c\pa_yb)_\phi,\Delta_q^h (\pa_tu)_\phi }_{L^2}} \leq \Vert \Delta_q^h (v\pa_yu)_\phi \Vert_{L^2}^2+\Vert \Delta_q^h (c\pa_yb)_\phi \Vert_{L^2}^2  + \frac{1}{10} \Vert \Delta_q^h (\pa_tu)_\phi \Vert_{L^2}^2,
\end{align*}
and 
\begin{align*}
    D_1 = \abs{\psca{ \Delta_q^h (u\pa_xb +b\pa_xu)_\phi,\Delta_q^h (\pa_tb)_\phi }_{L^2}} \leq \Vert \Delta_q^h (u\pa_xb)_\phi \Vert_{L^2}^2+ \Vert \Delta_q^h (b\pa_xu)_\phi \Vert_{L^2}^2  + \frac{1}{10} \Vert \Delta_q^h (\pa_tb)_\phi \Vert_{L^2}^2
\end{align*}
and
\begin{align*}
    D_2 = \abs{\psca{ \Delta_q^h (v\pa_yb +c\pa_yu)_\phi,\Delta_q^h (\pa_tb)_\phi }_{L^2}} \leq \Vert \Delta_q^h (v\pa_yb)_\phi \Vert_{L^2}^2+ \Vert \Delta_q^h (c\pa_yu)_\phi \Vert_{L^2}^2  + \frac{1}{10} \Vert \Delta_q^h (\pa_tb)_\phi \Vert_{L^2}^2.
\end{align*}
Then, we deduce that
\begin{align*}
	&\Vert \Delta_q^h (\pa_tu)_\phi \Vert_{L^2}^2 + \frac{d}{dt}\Vert \Delta_q^h \pa_y (u_\phi,b_\phi) \Vert_{L^2}^2 + \Vert \Delta_q^h (\pa_tb)_\phi \Vert_{L^2}^2 \leq C\Big(\Vert \Delta_q^h (u\pa_xu)_\phi \Vert_{L^2}^2  + \Vert \Delta_q^h (v\pa_yu)_\phi \Vert_{L^2}^2 \\ & +\Vert \Delta_q^h (b\pa_xb)_\phi \Vert_{L^2}^2 + \Vert \Delta_q^h (c\pa_yb)_\phi \Vert_{L^2}^2 + \Vert \Delta_q^h (u\pa_xb)_\phi \Vert_{L^2}^2 + \Vert \Delta_q^h (b\pa_xu)_\phi \Vert_{L^2}^2 +\Vert \Delta_q^h (c\pa_yu)_\phi \Vert_{L^2}^2+ \Vert \Delta_q^h (b\pa_xu)_\phi \Vert_{L^2}^2.
\end{align*}
Multiplying the result by $e^{2\mathcal{R}t}$ and integrating over $[0,t]$, we get 
\begin{multline*}
	\Vert e^{\mathcal{R}t}\Delta_q^h (\pa_tu)_\phi \Vert_{L^2_t(L^2)}^2 +\Vert e^{\mathcal{R}t}\Delta_q^h (\pa_tb)_\phi \Vert_{L^2_t(L^2)}^2+  \Vert e^{\mathcal{R}t} \Delta_q^h \pa_y (u_\phi,b_\phi) \Vert_{L^\infty_t(L^2)}^2
	\leq C \Big(\Vert \Delta_q^h \pa_ye^{a|D_x|} (u_0,b_0) \Vert_{L^2}^2 \\  + \Vert e^{\mathcal{R}t}\Delta_q^h (u\pa_xu)_\phi \Vert_{L^2_t(L^2)}^2+\Vert e^{\mathcal{R}t}\Delta_q^h (v\pa_yu)_\phi \Vert_{L^2_t(L^2)}^2 +  \Vert e^{\mathcal{R}t}\Delta_q^h (b\pa_xb)_\phi \Vert_{L^2_t(L^2)}^2 + \Vert e^{\mathcal{R}t}\Delta_q^h (c\pa_yb)_\phi \Vert_{L^2_t(L^2)}^2 \\ +\Vert e^{\mathcal{R}t}\Delta_q^h (u\pa_xb)_\phi \Vert_{L^2_t(L^2)}^2 +  \Vert e^{\mathcal{R}t}\Delta_q^h (v\pa_yb)_\phi \Vert_{L^2_t(L^2)}^2 +\Vert e^{\mathcal{R}t}\Delta_q^h (b\pa_xu)_\phi \Vert_{L^2_t(L^2)}^2 +  \Vert e^{\mathcal{R}t}\Delta_q^h (c\pa_yu)_\phi \Vert_{L^2_t(L^2)}^2\Big).
\end{multline*}
Multiplying the above inequality by $2^{3q}$, then taking the square root of the resulting estimate, and finally summing up the obtained equations with respect to $q \in \mathbb{Z},$ we obtain
\begin{multline} \label{4.14}
	\Vert e^{\mathcal{R}t} (\pa_tu)_\phi \Vert_{\tilde{L}^2_t(\mathcal{B}^\f32)} +\Vert e^{\mathcal{R}t} (\pa_tb)_\phi \Vert_{\tilde{L}^2_t(\mathcal{B}^\f32)}+  \Vert e^{\mathcal{R}t} \pa_y (u_\phi,b_\phi) \Vert_{\tilde{L}^\infty_t(\mathcal{B}^\f32)}
	\leq C \Big(\Vert \pa_ye^{a|D_x|} (u_0,b_0) \Vert_{\mathcal{B}^\f32} \\  + \Vert e^{\mathcal{R}t} (u\pa_xu)_\phi \Vert_{\tilde{L}^2_t(\mathcal{B}^\f32)}+\Vert e^{\mathcal{R}t} (v\pa_yu)_\phi \Vert_{\tilde{L}^2_t(\mathcal{B}^\f32)} +  \Vert e^{\mathcal{R}t} (b\pa_xb)_\phi \Vert_{\tilde{L}^2_t(\mathcal{B}^\f32)} + \Vert e^{\mathcal{R}t} (c\pa_yb)_\phi \Vert_{\tilde{L}^2_t(\mathcal{B}^\f32)} \\ +\Vert e^{\mathcal{R}t} (u\pa_xb)_\phi \Vert_{\tilde{L}^2_t(\mathcal{B}^\f32)} +  \Vert e^{\mathcal{R}t} (v\pa_yb)_\phi \Vert_{\tilde{L}^2_t(\mathcal{B}^\f32)} +\Vert e^{\mathcal{R}t}\ (b\pa_xu)_\phi \Vert_{\tilde{L}^2_t(\mathcal{B}^\f32)} +  \Vert e^{\mathcal{R}t} (c\pa_yu)_\phi \Vert_{\tilde{L}^2_t(\mathcal{B}^\f32)}. \Big).
\end{multline}
Next, it follows from the law of product in anisotropic Besov spaces and Poincaré inequality that 
\begin{align*}
	\Vert e^{\mathcal{R}t} (u\pa_xu)_\phi \Vert_{\tilde{L}^2_t(\mathcal{B}^\f32)} &\leq  \Vert  u_\phi \Vert_{\tilde{L}^\infty(\mathcal{B}^\f12)} \Vert e^{\mathcal{R}t} \pa_y u_\phi \Vert_{\tilde{L}^2_t(\mathcal{B}^\f52)} ; \\
	\Vert e^{\mathcal{R}t} (b\pa_xb)_\phi \Vert_{\tilde{L}^2_t(\mathcal{B}^\f32)} &\leq \Vert  b_\phi \Vert_{\tilde{L}^\infty(\mathcal{B}^\f12)} \Vert e^{\mathcal{R}t} \pa_y b_\phi \Vert_{\tilde{L}^2_t(\mathcal{B}^\f52)} \\
	\Vert e^{\mathcal{R}t} (v\pa_yu)_\phi \Vert_{\tilde{L}^2_t(\mathcal{B}^\f32)} &\leq  \Vert  u_\phi \Vert_{\tilde{L}^\infty(\mathcal{B}^\f12)} \Vert e^{\mathcal{R}t} \pa_y u_\phi \Vert_{\tilde{L}^2_t(\mathcal{B}^\f52)} + \Vert  u_\phi \Vert_{\tilde{L}^\infty(\mathcal{B}^\f52)} \Vert e^{\mathcal{R}t} \pa_y u_\phi \Vert_{\tilde{L}^2_t(\mathcal{B}^\f12)} \\
	\Vert e^{\mathcal{R}t} (c\pa_yb)_\phi \Vert_{\tilde{L}^2_t(\mathcal{B}^\f32)} &\leq  \Vert  b_\phi \Vert_{\tilde{L}^\infty(\mathcal{B}^\f12)} \Vert e^{\mathcal{R}t} \pa_y b_\phi \Vert_{\tilde{L}^2_t(\mathcal{B}^\f52)} + \Vert  b_\phi \Vert_{\tilde{L}^\infty(\mathcal{B}^\f52)} \Vert e^{\mathcal{R}t} \pa_y b_\phi \Vert_{\tilde{L}^2_t(\mathcal{B}^\f12)}.
\end{align*}
and 
\begin{align*}
    \Vert e^{\mathcal{R}t} (u\pa_xb)_\phi \Vert_{\tilde{L}^2_t(\mathcal{B}^\f32)} &+ \Vert e^{\mathcal{R}t} (b\pa_xu)_\phi \Vert_{\tilde{L}^2_t(\mathcal{B}^\f32)} \\ &\leq \Vert  u_\phi \Vert_{\tilde{L}^\infty(\mathcal{B}^\f12)} \Vert e^{\mathcal{R}t} \pa_y b_\phi \Vert_{\tilde{L}^2_t(\mathcal{B}^\f52)} + \Vert  b_\phi \Vert_{\tilde{L}^\infty(\mathcal{B}^\f12)} \Vert e^{\mathcal{R}t} \pa_y u_\phi \Vert_{\tilde{L}^2_t(\mathcal{B}^\f52)}
\end{align*}
and 
\begin{align*}
    \Vert e^{\mathcal{R}t} &(v\pa_yb)_\phi \Vert_{\tilde{L}^2_t(\mathcal{B}^\f32)} + \Vert e^{\mathcal{R}t} (c\pa_yu)_\phi \Vert_{\tilde{L}^2_t(\mathcal{B}^\f32)} \leq\Vert  u_\phi \Vert_{\tilde{L}^\infty(\mathcal{B}^\f12)} \Vert e^{\mathcal{R}t} \pa_y b_\phi \Vert_{\tilde{L}^2_t(\mathcal{B}^\f52)} \\ & + \Vert  u_\phi \Vert_{\tilde{L}^\infty(\mathcal{B}^\f52)} \Vert e^{\mathcal{R}t} \pa_y b_\phi \Vert_{\tilde{L}^2_t(\mathcal{B}^\f12)}+ \Vert  b_\phi \Vert_{\tilde{L}^\infty(\mathcal{B}^\f12)} \Vert e^{\mathcal{R}t} \pa_y u_\phi \Vert_{\tilde{L}^2_t(\mathcal{B}^\f52)} + \Vert  b_\phi \Vert_{\tilde{L}^\infty(\mathcal{B}^\f52)} \Vert e^{\mathcal{R}t} \pa_y u_\phi \Vert_{\tilde{L}^2_t(\mathcal{B}^\f12)}
\end{align*}
Inserting the above estimates into \eqref{4.14} and then using the smallness condition $\norm{u_{\phi}}_{\mathcal{B}^{\frac{1}{2}}} +\norm{b_{\phi}}_{\mathcal{B}^{\frac{1}{2}}}\leq \frac{1}{2C^2}$, we finally obtain 
\begin{align*}
 \Vert e^{\mathcal{R}t} (\pa_tu)_\phi \Vert_{\tilde{L}^2_t(\mathcal{B}^\f32)} &+ \Vert e^{\mathcal{R}t} (\pa_tb)_\phi \Vert_{\tilde{L}^2_t(\mathcal{B}^\f32)} + \Vert e^{\mathcal{R}t}  \pa_y (u_\phi,b_\phi) \Vert_{\tilde{L}^\infty_t(\mathcal{B}^\f32)} \\ &\leq C\Big( \Vert e^{a|D_x|}\pa_y (u_0,b_0) \Vert_{\mathcal{B}^\f32} + \Vert e^{a|D_x|} u_0 \Vert_{\mathcal{B}^\f52} +  \Vert e^{a|D_x|} b_0 \Vert_{\mathcal{B}^\f52} \Big).
 \end{align*}
 this complete the proof of the proposition \ref{prop}.
\end{proof}
\begin{lemma}\label{lem C+B+A}
Let A,B and C be a smooth function on $[0,T] \times \mathbb{R}^2_+$, and $s\in ]0,1]$, $T>0$ and $\phi$ be defined as in \eqref{eq:AnPhi}, with $\dot{\theta} (t) = \Vert \pa_y A_{\phi}(t) \Vert_{\mathcal{B}^{\frac{1}{2}}} $. There exist $C \geq 1$ such that, for any $t > 0$, $\phi(t,\xi) > 0$ and for any $B,C \in \tilde{L}^2_{t,\dot{\theta}(t)}(\mathcal{B}^{s+\frac{1}{2}})$, we have
	\begin{align}
	\sum_{q\in\ZZ} 2^{2qs} \int_0^t \abs{\psca{e^{\Rcal t'} \Delta_q^h (A\p_x B)_{\phi}, e^{\Rcal t'} \Delta_q^h C_{\phi}}_{L^2}} dt' \leq C \Vert e^{\Rcal t} B_{\phi} \Vert_{\tilde{L}^2_{t,\dot{\theta}(t)}(\mathcal{B}^{s+\frac{1}{2}})}\Vert e^{\Rcal t} C_{\phi} \Vert_{\tilde{L}^2_{t,\dot{\theta}(t)}(\mathcal{B}^{s+\frac{1}{2}})}.
	\end{align} 
\end{lemma}

\begin{proof}
We define $I$ the integral given by 

$$ I(t) = \int_0^t \abs{\psca{e^{\Rcal t'} \Delta_q^h (A\p_x B)_{\phi}, e^{\Rcal t'} \Delta_q^h C_{\phi}}_{L^2}} dt' .$$

As in \cite{BCD2011book}, using Bony's homogeneous decomposition into para-products $A\pa_x B$ in the horizontal variable and remainders as in Definition of a tempered distribution, we can write

$$ A\pa_x B = T^{h}_A \pa_{x} B + T^h_{\pa_x B} A + R^{h}(A,\pa_x B) $$
where,
\begin{align*}
	T_A \pa_xB = \sum_{q\in\ZZ} S_{q-1}^h A \DD_q^h \pa_x B \quad\mbox{ and }\quad R^h(A,\pa_x B) = \sum_{\abs{q'-q} \leq 1} \DD_q^h A \DD_{q'}^h \pa_x B.
\end{align*}

We have the following bound of $I$

\begin{align*}
	\int_0^t \abs{\psca{e^{\Rcal t'} \Delta_q^h (A\p_x B)_{\phi}, e^{\Rcal t'} \Delta_q^h C_{\phi}}_{L^2}} dt' \leq I_{1,q} + I_{2,q} + I_{3,q}, 
\end{align*}
where
\begin{align*}
	I_{1,q} &= \int_0^t \abs{\psca{ e^{\Rcal t'} \Delta_q^h ( T^h_{A}\pa_x B)_{\phi}, e^{\Rcal t'} \Delta_q^h C_{\phi}}_{L^2}} dt'\\
	I_{2,q} &= \int_0^t \abs{\psca{ e^{\Rcal t'} \Delta_q^h ( T^h_{\pa_xB} A )_{\phi}, e^{\Rcal t'} \Delta_q^h C_{\phi}}_{L^2}} dt'\\
	I_{3,q} &= \int_0^t \abs{\psca{ e^{\Rcal t'} \Delta_q^h (R^h(A,\pa_xB))_{\phi}, e^{\Rcal t'} \Delta_q^h C_{\phi}}_{L^2}} dt'.
\end{align*}

We start by getting the estimate of the first term $I_{1,q}$, for that we need to use the support properties given in  [\cite{B1981}, Proposition 2.10] and the definition of $T^{h}_{A} \pa_x B$, we infer 

\begin{equation} \label{eq:normBCE}
	I_{1,q} \leq \sum_{|q-q'| \leq 4} \int_0^t e^{2\Rcal t'} \Vert S^h_{q'-1} A_{\phi} (t') \Vert_{L^{\infty}}\Vert \Delta_{q'}^h \partial_xB_{\phi} (t') \Vert_{L^2} \Vert \Delta_q^h C_{\phi} (t') \Vert_{L^2} dt'.
\end{equation}

By the poincarrés inequality on the interval $\lbrace 0<y<1 \rbrace$, we have the inclusion $\dot{H}^1_y \hookrightarrow L^\infty_y$ and 
\begin{align} \label{eq:normB12}
	\Vert \Delta^h_q A_{\phi}(t') \Vert_{L^{\infty}} \lesssim 2^{\frac{q}{2}} \Vert \Delta_q^h A_{\phi}(t') \Vert_{L^2_h(L^{\infty}_v)} \lesssim 2^{\frac{q}{2}} \Vert \Delta_q^h \pa_y A_{\phi}(t') \Vert_{L^2} \lesssim d_q(A_\phi) \Vert  \pa_y A_{\phi}(t') \Vert_{\mathcal{B}^{\f12}},
\end{align}
where $\lbrace d_q(A_{\phi}) \rbrace$ is a square-summable sequence with $\sum d_q(u_\phi)^2 = 1$. Then,
\begin{align*}
	\Vert S^h_{q'-1} A_{\phi} (t') \Vert_{L^{\infty}} \lesssim  \Vert  \pa_y A_{\phi}(t') \Vert_{\mathcal{B}^{\f12}},
\end{align*}
then, we replace this result in our estimate \eqref{eq:normBCE}, and combining with H\"older inequality, imply that 

\begin{align*}
	I_{1,q} &\lesssim \sum_{|q-q'| \leq 4}  \int_0^t \Vert \pa_y A_{\phi}(t') \Vert_{\mathcal{B}^{\f12}} e^{\Rcal t'} \Vert \Delta_{q'}^h \pa_x B_{\phi}(t') \Vert_{L^2} e^{\Rcal t'} \Vert \Delta_q^h C_{\phi}(t') \Vert_{L^2} dt'\\
	&\lesssim \sum_{|q-q'| \leq 4} 2^{q'} \int_0^t \Vert \pa_y A_{\phi}(t') \Vert_{\mathcal{B}^{\f12}}^\f12 e^{\Rcal t'} \Vert \Delta_{q'}^h B_{\phi}(t') \Vert_{L^2} \Vert \pa_y A_{\phi}(t') \Vert_{\mathcal{B}^{\f12}}^\f12e^{\Rcal t'} \Vert \Delta_q^h C_{\phi}(t') \Vert_{L^2} dt'\\
	&\lesssim \sum_{|q-q'| \leq 4} 2^{q'} \left(\int_0^t \Vert \pa_y A_{\phi}(t') \Vert_{\mathcal{B}^{\f12}} e^{2\Rcal t'} \Vert \Delta_{q'}^h B_{\phi} \Vert_{L^2}^2 dt'\right)^{\f12} \left(\int_0^t \Vert  \pa_y A_{\phi}(t') \Vert_{\mathcal{B}^{\f12}} e^{2\Rcal t'} \Vert \Delta_q^h C_{\phi} \Vert_{L^2}^2 dt'\right)^{\f12}.
\end{align*}
we note that $\dot{\theta}(t) \simeq \Vert  \pa_y A_{\phi}(t') \Vert_{\mathcal{B}^{\f12}}$, using the definition \eqref{def:CLweight}, we have 

\begin{align*}
	\left(\int_0^t \Vert \pa_y A_{\phi}(t') \Vert_{\mathcal{B}^{\f12}} e^{2\Rcal t'} \Vert \Delta_q^h C_{\phi} \Vert_{L^2}^2 dt'\right)^{\frac{1}{2}} &\lesssim \left(\int_0^t \dot{\theta}(t') e^{2\Rcal t'} \Vert \Delta_q^h C_{\phi} \Vert_{L^2}^2 dt'\right)^{\frac{1}{2}} \\ &\lesssim 2^{-q(s+\frac{1}{2})} d_q(C_\phi) \Vert e^{\Rcal t} C_{\phi} \Vert_{\tilde{L}^{2}_{t,\dot{\theta}}(\mathcal{B}^{s+\frac{1}{2}})}.
\end{align*}
Then,
\begin{align} \label{eq:I_1,q}
	I_{1,q} \lesssim 2^{-2qs} d_q^2 \Vert e^{\Rcal t} B_{\phi} \Vert_{\tilde{L}^{2}_{t,\dot{\theta}}(\mathcal{B}^{s+\frac{1}{2}})} \Vert e^{\Rcal t} C_{\phi} \Vert_{\tilde{L}^{2}_{t,\dot{\theta}}(\mathcal{B}^{s+\frac{1}{2}})},
\end{align}
where 
$$ d_q^2 = d_q(C_\phi) \left( \sum_{|q-q'| \leq 4} d_{q'}(B_\phi) 2^{(q-q')(s-\f12)} \right) $$
if we multiply \eqref{eq:I_1,q} by $2^{2qs}$ and summing with respect to $q\in \mathbb{Z}$, we get 
\begin{align} \label{eq:I1q}
	\sum_{q\in\ZZ} 2^{2qs} I_{1,q} \lesssim C\Vert e^{\Rcal t} B_{\phi} \Vert_{\tilde{L}^{2}_{t,\dot{\theta}}(\mathcal{B}^{s+\frac{1}{2}})} \Vert e^{\Rcal t} C_{\phi} \Vert_{\tilde{L}^{2}_{t,\dot{\theta}}(\mathcal{B}^{s+\frac{1}{2}})}.
\end{align}

\medskip

Now we move to get the estimate of the second terme, by using the support properties given in [\cite{B1981}, Proposition 2.10] and the definition of $T^h_{\pa_x B} A$, we can estimate $I_{2,q}$ in a similar way as we did for $I_{1,q}$. 

\begin{align*}
I_{2,q} (t) &\leq \int_0^t \abs{\psca{ e^{\Rcal t'} \Delta_q^h ( T^h_{\pa_xB} A )_{\phi}, e^{\Rcal t'} \Delta_q^h C_{\phi}}_{L^2}} dt' \\
&\leq \sum_{|q-q'|\leq 4} \int_0^t e^{2\Rcal t'}\Vert S^h_{q'-1} \pa_x B_\phi \Vert_{L^\infty_h(L^2_v)} \Vert \Delta_{q'}^h A_\phi \Vert_{L^2_h(L^\infty_v)} \Vert \Delta_q^h C_\phi \Vert_{L^2} dt'.
\end{align*}

As in \eqref{eq:normB12}, we can write
\begin{equation*} 
	\Vert \Delta_{q'}^h A_{\phi} \Vert_{L^2_h(L^{\infty}_v)} \lesssim \Vert \Delta_{q'}^h \pa_y A_{\phi} \Vert_{L^2} \lesssim 2^{-\frac{q'}{2}} d_{q'}(A_\phi) \Vert  \pa_y A_{\phi}(t) \Vert_{\mathcal{B}^{\f12}}.
\end{equation*}

Since $ 0<s\leq 1 $, we have 
\begin{align*}
	 I_{2,q} &\leq  \sum_{|q-q'|\leq 4} \int_0^t e^{2\Rcal t'}\Vert S^h_{q'-1} \pa_x B_\phi \Vert_{L^\infty_h(L^2_v)} \Vert \Delta_{q'}^h A_\phi \Vert_{L^2_h(L^\infty_v)} \Vert \Delta_q^h C_\phi \Vert_{L^2} dt'\\
	&\leq  \sum_{|q-q'|\leq 4} \int_0^t 2^{-\frac{q'}{2}} d_{q'}(A_\phi) e^{\Rcal t'}\Vert S^h_{q'-1} \pa_x B_\phi \Vert_{L^\infty_h(L^2_v)} \Vert  \pa_y A_{\phi}(t)\Vert_{\mathcal{B}^{\f12}}e^{\Rcal t'}\Vert \Delta_q^h C_\phi \Vert_{L^2} dt'\\
	&\leq  \sum_{|q-q'|\leq 4} \int_0^t 2^{-\frac{q'}{2}} d_{q'}(A_\phi) e^{\Rcal t'} \sum_{l\leq q'-2} 2^l 2^{\frac{l}{2}} \Vert \Delta_l^h B_\phi \Vert_{L^2} \Vert  \pa_y A_{\phi}(t) \Vert_{\mathcal{B}^{\f12}}e^{\Rcal t'}\Vert \Delta_q^h C_\phi \Vert_{L^2} dt'\\
	&\leq  \sum_{|q-q'|\leq 4} \int_0^t 2^{-\frac{q'}{2}} d_{q'}(A_\phi) \sum_{l\leq q'-2} 2^{\frac{3l}{2}} 2^{-l(s+\f12)} d_l(B_{\phi}) \Vert e^{\Rcal t'} B_\phi \Vert_{B^{s+\f12}} \Vert  \pa_y A_{\phi}(t) \Vert_{\mathcal{B}^{\f12}}\Vert e^{\Rcal t'}\Delta_q^h C_\phi \Vert_{L^2} dt'\\
	&\leq \sum_{|q-q'|\leq 4} \int_0^t 2^{-\frac{q'}{2}} d_{q'}(A_\phi)2^{q'(1-s)}  \Vert e^{\Rcal t'} B_\phi \Vert_{B^{s+\f12}} \Vert  \pa_y A_{\phi}(t) \Vert_{\mathcal{B}^{\f12}}\Vert e^{\Rcal t'}\Delta_q^h C_\phi \Vert_{L^2} dt'\\
	&\leq \sum_{|q-q'|\leq 4} \int_0^t 2^{-\frac{q'}{2}} d_{q'}(A_\phi) d_{q}(C_\phi)2^{q'(1-s)}2^{-q(s+\f12)}  \Vert e^{\Rcal t'} B_\phi \Vert_{B^{s+\f12}} \Vert  \pa_y A_{\phi}(t) \Vert_{\mathcal{B}^{\f12}}\Vert e^{\Rcal t'} C_\phi \Vert_{B^{s+\f12}} dt'\\
	&\leq d_q^2 2^{-2qs}  \Vert e^{\Rcal t} B_\phi \Vert_{\tilde{L}^2_{t,\dot{\theta}(t)}(B^{s+\f12})}  \Vert e^{\Rcal t} C_\phi \Vert_{\tilde{L}^2_{t,\dot{\theta}(t)}(B^{s+\f12})}
\end{align*}
where
\begin{align*}
	d_q^2 = d_q(C_\phi) \left(\sum_{|q-q'|\leq 4}d_{q'} 2^{(q -q')(s-\frac{1}{2})}\right)
\end{align*}
is a summable sequence of positive constants. Summing with respect to $q \in \ZZ$, and using Fubini's theorem, we get
\begin{align} \label{eq:I2q}
	\sum_{q\in\ZZ} 2^{2qs} I_{2,q} \lesssim C\Vert e^{\Rcal t} B_{\phi} \Vert_{\tilde{L}^2_{t,\dot{\theta}}(\mathcal{B}^{s+\frac{1}{2}})} \Vert e^{\Rcal t} B_{\phi} \Vert_{\tilde{L}^2_{t,\dot{\theta}}(\mathcal{B}^{s+\frac{1}{2}})},
\end{align}
where we recall that $\dot{\theta} (t) \simeq \Vert \pa_y A_{\phi} \Vert_{\mathcal{B}^{\frac{1}{2}}}$.

To end this proof, it remains to estimate $I_{3,q}$(is the rest term). Using the support properties given in [\cite{B1981}, Proposition 2.10], the definition of $R^h(A,\pa_x B)$ and Bernstein lemma \ref{le:Bernstein}, we can write
\begin{align*}
	I_{3,q} &= \int_0^t \abs{\psca{ e^{\Rcal t'} \Delta_q^h (R^h(A,\pa_xB))_{\phi}, e^{\Rcal t'} \Delta_q^h C_{\phi}}_{L^2}} dt'\\ &\leq 2^{\frac{q}{2}} \sum_{q'\geq q-3}\int_0^t e^{2\Rcal t'} \Vert \DD^h_{q'} A_{\phi}\Vert_{L^{2}_h(L^{\infty}_v)} \Vert \Delta_{q'}^h \pa_x B_{\phi}  \Vert_{L^{2}} \Vert \Delta_{q}^h  C_{\phi} \Vert_{L^{2}}dt' \\
	&\leq 2^{\frac{q}{2}} \sum_{q'\geq q-3}\int_0^t e^{2\Rcal t'}   2^{q'(1-\frac{1}{2})}  \Vert \pa_y A_{\phi}\Vert_{\mathcal{B}^{\frac{1}{2}}} \Vert \Delta_{q'}^h B_{\phi}  \Vert_{L^{2}} \Vert \Delta_{q}^h C_{\phi} \Vert_{L^{2}} dt'  \\
	&\leq 2^{\frac{q}{2}} \sum_{q'\geq q-3} \int_0^t e^{2\Rcal t'} 2^{\frac{q'}{2}}  \Vert \pa_y A_{\phi}\Vert_{\mathcal{B}^{\frac{1}{2}}} \Vert \Delta_{q'}^h B_{\phi}  \Vert_{L^{2}} \Vert \Delta_{q}^h C_{\phi} \Vert_{L^{2}}dt'.
\end{align*}
Since $ 0<s\leq 1 $, we have
\begin{align*}
	 I_{3,q} &\leq 2^{\frac{q}{2}} \sum_{q'\geq q-3}\int_0^t  2^{\frac{q'}{2}}  \Vert \pa_y A_{\phi}\Vert_{\mathcal{B}^{\frac{1}{2}}} \Vert \Delta_{q'}^h  e^{\Rcal t'} B_{\phi}  \Vert_{L^{2}} \Vert \Delta_{q}^h e^{\Rcal t'} C_{\phi} \Vert_{L^{2}}dt' \\
	&\leq  2^{\frac{q}{2}} \sum_{q'\geq q-3}\int_0^t 2^{\frac{q'}{2}}  d_{q'}(B_\phi) 2^{-q'(s+\frac{1}{2})} \Vert e^{\Rcal t'} B_{\phi} \Vert_{\mathcal{B}^{s+\frac{1}{2}}}  \Vert \pa_y A_{\phi} \Vert_{\mathcal{B}^{\frac{1}{2}}}   d_{q}(C_\phi) 2^{-q(s+\frac{1}{2})} \Vert e^{\Rcal t'} C_{\phi} \Vert_{\mathcal{B}^{s+\frac{1}{2}}}dt'  \\
	&\leq d_q(C_\phi) 2^{-2qs} \int_0^t \Vert e^{\Rcal t} B_{\phi} \Vert_{\mathcal{B}^{s+\frac{1}{2}}} \Vert \pa_y u_{\phi} \Vert_{\mathcal{B}^{\frac{1}{2}}} \Vert e^{\Rcal t} C_{\phi} \Vert_{\mathcal{B}^{s+\frac{1}{2}}} \left(  \sum_{q'\geq q-3} d_{q'}(w_\phi) 2^{(q-q')s} \right)dt' \\
	&\leq d_q^2  2^{-2qs} \Vert e^{\Rcal t} B_\phi \Vert_{\tilde{L}^2_{t,\dot{\theta}(t)}(B^{s+\f12})}  \Vert e^{\Rcal t} C_\phi \Vert_{\tilde{L}^2_{t,\dot{\theta}(t)}(B^{s+\f12})},
\end{align*}
where
\begin{align*}
d_q^2 = d_q(C_\phi) \left(  \sum_{q'\geq k-3} d_{q'}(B_\phi) 2^{(q-q')s} \right)
\end{align*}
is a summable sequence of positive constants. Summing with respect to $q \in \ZZ$, and using Fubini's theorem, we finally obtain
\begin{align} \label{eq:I3q}
	\sum_{q\in\ZZ} 2^{2qs} I_{3,q} \lesssim \Vert e^{\Rcal t} B_\phi \Vert_{\tilde{L}^2_{t,\dot{\theta}(t)}(B^{s+\f12})}  \Vert e^{\Rcal t} C_\phi \Vert_{\tilde{L}^2_{t,\dot{\theta}(t)}(B^{s+\f12})}.
\end{align}
Lemma \ref{lem C+B+A} is then proved by summing Estimates \eqref{eq:I1q}, \eqref{eq:I2q} and \eqref{eq:I3q}. \hfill 
\end{proof}

\begin{lemma} \label{lem:ApaxA}
Let A,B and C be a smooth function on $[0,T] \times \mathbb{R}^2_+$, and $s\in ]0,1]$, $T>0$ and $\phi$ be defined as in \eqref{eq:AnPhi}, with $\dot{\theta} (t) = \Vert \pa_y A_{\phi}(t) \Vert_{\mathcal{B}^{\frac{1}{2}}} $ and $B(t,x,y) = -\int_0^y \pa_x A(t,x,y')dy'$. There exist $C \geq 1$ such that, for any $t > 0$, $\phi(t,\xi) > 0$ and for any $B,C \in \tilde{L}^2_{t,\dot{\theta}(t)}(\mathcal{B}^{s+\frac{1}{2}})$, we have
		\begin{align}
	\sum_{q\in\ZZ} 2^{2qs} \int_0^t \abs{\psca{e^{\Rcal t'} \Delta_q^h (B\p_y A)_{\phi}, e^{\Rcal t'} \Delta_q^h C_{\phi}}_{L^2}} dt' \leq C \Vert e^{\Rcal t} A_{\phi} \Vert_{\tilde{L}^2_{t,\dot{\theta}(t)}(\mathcal{B}^{s+\frac{1}{2}})}\Vert e^{\Rcal t} C_{\phi} \Vert_{\tilde{L}^2_{t,\dot{\theta}(t)}(\mathcal{B}^{s+\frac{1}{2}})}.
	\end{align}
\end{lemma}
\begin{proof}
We define the function $K(t)$ by the following formula
$$ K(t) = \int_0^t \abs{\psca{e^{\Rcal t'} \Delta_q^h (B\p_y A)_{\phi}, e^{\Rcal t'} \Delta_q^h C_{\phi}}_{L^2}} dt'. $$

As in \cite{BCD2011book}, using Bony's homogeneous decomposition into para-products $A\pa_x B$ in the horizontal variable and remainders as in Definition of a tempered distribution, we can write

$$ B\pa_y A = T^{h}_B \pa_{y} A + T^h_{\pa_y A} B + R^{h}(B,\pa_y A) $$
where,
\begin{align*}
	T_B \pa_yA = \sum_{q\in\ZZ} S_{q-1}^h B \DD_q^h \pa_y A \quad\mbox{ and }\quad R^h(B,\pa_y A) = \sum_{\abs{q'-q} \leq 1} \DD_q^h B \DD_{q'}^h \pa_y A.
\end{align*}

We have the following bound of $K$

\begin{align*}
	K(t) = \int_0^t \abs{\psca{e^{\Rcal t'} \Delta_q^h (B\p_y A)_{\phi}, e^{\Rcal t'} \Delta_q^h C_{\phi}}_{L^2}} dt' \leq K_{1,q} + K_{2,q} + K_{3,q}, 
\end{align*}
where
\begin{align*}
	K_{1,q} &= \int_0^t \abs{\psca{ e^{\Rcal t'} \Delta_q^h ( T^{h}_B \pa_{y} A)_{\phi}, e^{\Rcal t'} \Delta_q^h C_{\phi}}_{L^2}} dt'\\
	K_{2,q} &= \int_0^t \abs{\psca{ e^{\Rcal t'} \Delta_q^h ( T^h_{\pa_y A} B )_{\phi}, e^{\Rcal t'} \Delta_q^h C_{\phi}}_{L^2}} dt'\\
	K_{3,q} &= \int_0^t \abs{\psca{ e^{\Rcal t'} \Delta_q^h (R^{h}(B,\pa_y A))_{\phi}, e^{\Rcal t'} \Delta_q^h C_{\phi}}_{L^2}} dt'.
\end{align*}

We start by getting the estimate of the first term $K_{1,q}$, for that we need to use the support properties given in  [\cite{B1981}, Proposition 2.10] and the definition of $T^{h}_{B} \pa_y A$, we infer 

\begin{align*} 
	K_{1,q} &\leq \sum_{|q-q'| \leq 4} \int_0^t e^{2\Rcal t'} \Vert S^h_{q'-1} B_{\phi} (t') \Vert_{L^{\infty}}\Vert \Delta_{q'}^h \partial_yA_{\phi} (t') \Vert_{L^2} \Vert \Delta_q^h C_{\phi} (t') \Vert_{L^2} dt' \\
	&\leq \int_0^t \sum_{|q-q'| \leq 4}  e^{2\Rcal t'} \Vert S^h_{q'-1} B_{\phi} (t') \Vert_{L^{\infty}} d_{q'}(A_\phi) 2^{\frac{-q'}{2}} \Vert \pa_y A_{\phi}(t') \Vert_{\mathcal{B}^{\f12}} \Vert \Delta_q^h C_{\phi} (t') \Vert_{L^2} dt',
\end{align*}

then we have that 

\begin{equation} \label{eq:normK}
K_{1,q} \leq \int_0^t \sum_{|q-q'| \leq 4}d_{q'}(A_\phi) 2^{\frac{-q'}{2}} e^{2\Rcal t'} \Vert S^h_{q'-1} B_{\phi} (t') \Vert_{L^{\infty}}  \Vert \pa_y A_{\phi}(t') \Vert_{\mathcal{B}^{\f12}} \Vert \Delta_q^h C_{\phi} (t') \Vert_{L^2} dt'.
\end{equation}

By the poincarré's inequality on the interval $\lbrace 0<y<1 \rbrace$, we have the inclusion $\dot{H}^1_y \hookrightarrow L^\infty_y$ and 
\begin{align*} 
	\Vert \Delta^h_q B_{\phi}(t') \Vert_{L^{\infty}} &\lesssim \int_0^1 \Vert \Delta_q^h \pa_x A_{\phi}(t',x,y') \Vert_{L^{\infty}_{h}} dy'\lesssim 2^\frac{q}{2} \int_0^1 \Vert \Delta_q^h \pa_x A_{\phi}(t',x,y') \Vert_{L^2_{h}} dy' \\ &\lesssim 2^{\frac{q}{2}} 2^{q}  \int_0^1 \Vert \Delta_q^h A_{\phi}(t',x,y') \Vert_{L^2_{h}} dy' \lesssim 2^{\frac{3q}{2}}\Vert \Delta_q^h  A_{\phi}(t') \Vert_{L^2} 
\end{align*}
 Then,
\begin{align}\label{eq:normK1}
	\Vert S^h_{q'-1} B_{\phi} (t') \Vert_{L^{\infty}} \lesssim 2^{\frac{3q}{2}} \Vert \Delta_q^h A_{\phi}(t') \Vert_{L^2},
\end{align}
then, we replace this result in our estimate \eqref{eq:normK}, and combining with H\"older inequality, imply that 

\begin{align*}
	K_{1,q} &\lesssim \int_0^t \sum_{|q-q'| \leq 4}d_{q'}(A_\phi) 2^{\frac{-q'}{2}} e^{2\Rcal t'} \Vert S^h_{q'-1} B_{\phi} (t') \Vert_{L^{\infty}}  \Vert \pa_y A_{\phi}(t') \Vert_{\mathcal{B}^{\f12}} \Vert \Delta_q^h C_{\phi} (t') \Vert_{L^2} dt'\\
	&\lesssim \int_0^t \sum_{|q-q'| \leq 4}d_{q'}(A_\phi) 2^{\frac{-q'}{2}} \sum_{l\leq q'-2} 2^{\frac{3l}{2}} e^{\Rcal t'} \Vert \Delta_l^h A_{\phi}(t') \Vert_{L^2}  \Vert \pa_y A_{\phi}(t') \Vert_{\mathcal{B}^{\f12}}e^{\Rcal t'} \Vert \Delta_q^h C_{\phi} (t') \Vert_{L^2} dt' \\
	&\lesssim \int_0^t \sum_{|q-q'| \leq 4}d_{q'}(A_\phi) 2^{\frac{-q'}{2}} \sum_{l\leq q'-2} d_l 2^{\frac{3l}{2}}2^{-l(s+\f12)} \Vert e^{\Rcal t'} A_{\phi}(t') \Vert_{\mathcal{B}^{s+\f12}}  \Vert \pa_y A_{\phi}(t') \Vert_{\mathcal{B}^{\f12}} \Vert e^{\Rcal t'}\Delta_q^h C_{\phi} (t') \Vert_{L^2} dt' \\
	&\lesssim \int_0^t \sum_{|q-q'| \leq 4}d_{q'}(A_\phi) 2^{\frac{-q'}{2}} 2^{q'(1-s)}\Vert e^{\Rcal t'} A_{\phi}(t') \Vert_{\mathcal{B}^{s+\f12}}  \Vert \pa_y A_{\phi}(t') \Vert_{\mathcal{B}^{\f12}} \Vert e^{\Rcal t'}\Delta_q^h C_{\phi} (t') \Vert_{L^2} dt' \\
	&\lesssim \sum_{|q-q'| \leq 4} d_{q'}(A_\phi) 2^{\frac{-q'}{2}} 2^{q'(1-s)} \left(\int_0^t \Vert \pa_y A_{\phi}(t') \Vert_{\mathcal{B}^{\f12}}  \Vert e^{\Rcal t'} A_{\phi} \Vert_{\mathcal{B}^{s+\f12}}^2 dt'\right)^{\f12} \\ &\times \left(\int_0^t \Vert  \pa_y A_{\phi}(t') \Vert_{\mathcal{B}^{\f12}} e^{2\Rcal t'} \Vert \Delta_q^h C_{\phi} \Vert_{L^2}^2 dt'\right)^{\f12}.
\end{align*}
we note that $\dot{\theta}(t) \simeq \Vert  \pa_y A_{\phi}(t') \Vert_{\mathcal{B}^{\f12}}$, using the definition \eqref{def:CLweight}, we have 

\begin{align*}
	\left(\int_0^t \Vert \pa_y A_{\phi}(t') \Vert_{\mathcal{B}^{\f12}} e^{2\Rcal t'} \Vert \Delta_q^h C_{\phi} \Vert_{L^2}^2 dt'\right)^{\frac{1}{2}} &\lesssim \left(\int_0^t \dot{\theta}(t') e^{2\Rcal t'} \Vert \Delta_q^h C_{\phi} \Vert_{L^2}^2 dt'\right)^{\frac{1}{2}} \\ &\lesssim 2^{-q(s+\frac{1}{2})} d_q(C_\phi) \Vert e^{\Rcal t} C_{\phi} \Vert_{\tilde{L}^{2}_{t,\dot{\theta}}(\mathcal{B}^{s+\frac{1}{2}})}.
\end{align*}
Then,
\begin{align} \label{eq:K_1,q}
	I_{1,q} \lesssim 2^{-2qs} d_q^2 \Vert e^{\Rcal t} A_{\phi} \Vert_{\tilde{L}^{2}_{t,\dot{\theta}}(\mathcal{B}^{s+\frac{1}{2}})} \Vert e^{\Rcal t} C_{\phi} \Vert_{\tilde{L}^{2}_{t,\dot{\theta}}(\mathcal{B}^{s+\frac{1}{2}})},
\end{align}
where 
$$ d_q^2 = d_q(C_\phi) \left( \sum_{|q-q'| \leq 4} d_{q'}(A_\phi) 2^{(q-q')(s-\f12)} \right) $$
if we multiply $\eqref{eq:K_1,q}$ by $2^{2qs}$ and summing with respect to $q\in \mathbb{Z}$, we get 
\begin{align} \label{eq:K1q}
	\sum_{q\in\ZZ} 2^{2qs} K_{1,q} \lesssim C\Vert e^{\Rcal t} A_{\phi} \Vert_{\tilde{L}^{2}_{t,\dot{\theta}}(\mathcal{B}^{s+\frac{1}{2}})} \Vert e^{\Rcal t} C_{\phi} \Vert_{\tilde{L}^{2}_{t,\dot{\theta}}(\mathcal{B}^{s+\frac{1}{2}})}.
\end{align}

\medskip

Now we move to get the estimate of the second terme, by using the support properties given in [\cite{B1981}, Proposition 2.10] and the definition of $T^h_{\pa_y A} B$, we can estimate $K_{2,q}$ in a similar way as we did for $K_{1,q}$. 

\begin{align*}
K_{2,q} (t) &\leq \int_0^t \abs{\psca{ e^{\Rcal t'} \Delta_q^h ( T^h_{\pa_yA} B )_{\phi}, e^{\Rcal t'} \Delta_q^h C_{\phi}}_{L^2}} dt' \\
&\leq \sum_{|q-q'|\leq 4} \int_0^t e^{2\Rcal t'}\Vert S^h_{q'-1} \pa_y A_\phi \Vert_{L^\infty_h(L^2_v)} \Vert \Delta_{q'}^h B_\phi \Vert_{L^2_h(L^\infty_v)} \Vert \Delta_q^h C_\phi \Vert_{L^2} dt'.
\end{align*}

As in \eqref{eq:normK1}, we can write
\begin{equation*} 
	\Vert \Delta_{q'}^h B_{\phi} \Vert_{L^2_h(L^{\infty}_v)} \lesssim  2^{q'} \int_0^y\Vert \Delta_{q'}^h A_{\phi}(t,x,y') \Vert_{L_h^2}dy' \lesssim 2^{q'} \Vert  \Delta_{q'}^h A_{\phi}(t) \Vert_{L^2}.
\end{equation*}

Since $ 0<s\leq 1 $, we have 
\begin{align*}
	 K_{2,q} &\leq  \sum_{|q-q'|\leq 4} \int_0^t e^{2\Rcal t'}\Vert S^h_{q'-1} \pa_y A_\phi \Vert_{L^\infty_h(L^2_v)} \Vert \Delta_{q'}^h B_\phi \Vert_{L^2_h(L^\infty_v)} \Vert \Delta_q^h C_\phi \Vert_{L^2} dt'\\
	&\leq  \sum_{|q-q'|\leq 4} \int_0^t 2^{q'} e^{\Rcal t'}\Vert S^h_{q'-1} \pa_y A_\phi \Vert_{L^\infty_h(L^2_v)} \Vert  \Delta_{q'}^h A_{\phi}(t) \Vert_{L^2}e^{\Rcal t'}\Vert \Delta_q^h C_\phi \Vert_{L^2} dt'\\
	&\leq  \sum_{|q-q'|\leq 4} \int_0^t 2^{q'} e^{\Rcal t'} \Vert \pa_y A_\phi \Vert_{\mathcal{B}^\f12}\Vert  \Delta_{q'}^h A_{\phi}(t) \Vert_{L^2}e^{\Rcal t'}\Vert \Delta_q^h C_\phi \Vert_{L^2} dt'\\
	&\leq \sum_{|q-q'|\leq 4} \int_0^t 2^{q'} \Vert \pa_y A_\phi \Vert_{\mathcal{B}^\f12}\Vert e^{\Rcal t'} \Delta_{q'}^h A_{\phi}(t) \Vert_{L^2}\Vert e^{\Rcal t'}\Delta_q^h C_\phi \Vert_{L^2} dt'\\
	&\lesssim \sum_{|q-q'| \leq 4} 2^{q'} \left(\int_0^t \Vert \pa_y A_{\phi}(t') \Vert_{\mathcal{B}^{\f12}}  \Vert e^{\Rcal t'}\Delta_{q'}^h A_{\phi} \Vert_{L^2}^2 dt'\right)^{\f12} \times \left(\int_0^t \Vert  \pa_y A_{\phi}(t') \Vert_{\mathcal{B}^{\f12}} \Vert e^{\Rcal t'} \Delta_q^h C_{\phi} \Vert_{L^2}^2 dt'\right)^{\f12}.
\end{align*}

we note that $\dot{\theta}(t) \simeq \Vert  \pa_y A_{\phi}(t') \Vert_{\mathcal{B}^{\f12}}$, using the definition \eqref{def:CLweight}, we have 
\begin{align*}
	\left(\int_0^t \Vert \pa_y A_{\phi}(t') \Vert_{\mathcal{B}^{\f12}} e^{2\Rcal t'} \Vert \Delta_q^h C_{\phi} \Vert_{L^2}^2 dt'\right)^{\frac{1}{2}} &\lesssim \left(\int_0^t \dot{\theta}(t') e^{2\Rcal t'} \Vert \Delta_q^h C_{\phi} \Vert_{L^2}^2 dt'\right)^{\frac{1}{2}} \\ &\lesssim 2^{-q(s+\frac{1}{2})} d_q(C_\phi) \Vert e^{\Rcal t} C_{\phi} \Vert_{\tilde{L}^{2}_{t,\dot{\theta}}(\mathcal{B}^{s+\frac{1}{2}})}.
\end{align*}
Then,
\begin{align} \label{eq:K_2,q}
	K_{2,q} \lesssim 2^{-2qs} d_q^2 \Vert e^{\Rcal t} A_{\phi} \Vert_{\tilde{L}^{2}_{t,\dot{\theta}}(\mathcal{B}^{s+\frac{1}{2}})} \Vert e^{\Rcal t} C_{\phi} \Vert_{\tilde{L}^{2}_{t,\dot{\theta}}(\mathcal{B}^{s+\frac{1}{2}})},
\end{align}

where
\begin{align*}
	d_q^2 = d_q(C_\phi) \left(\sum_{|q-q'|\leq 4}d_{q'} 2^{(q -q')(s-\frac{1}{2})}\right)
\end{align*}
is a summable sequence of positive constants. Summing with respect to $q \in \ZZ$, and using Fubini's theorem, we get
\begin{align} \label{eq:K2q}
	\sum_{q\in\ZZ} 2^{2qs} K_{2,q} \lesssim C\Vert e^{\Rcal t} A_{\phi} \Vert_{\tilde{L}^2_{t,\dot{\theta}}(\mathcal{B}^{s+\frac{1}{2}})} \Vert e^{\Rcal t} C_{\phi} \Vert_{\tilde{L}^2_{t,\dot{\theta}}(\mathcal{B}^{s+\frac{1}{2}})},
\end{align}
where we recall that $\dot{\theta} (t) \simeq \Vert \pa_y A_{\phi} \Vert_{\mathcal{B}^{\frac{1}{2}}}$.

To end this proof, it remains to estimate $K_{3,q}$(is the rest term). Using the support properties given in [\cite{B1981}, Proposition 2.10], the definition of $R^h(B,\pa_y A)$ and Bernstein lemma \ref{le:Bernstein}, we can write
\begin{align*}
	K_{3,q} &= \int_0^t \abs{\psca{ e^{\Rcal t'} \Delta_q^h (R^h(B,\pa_yA))_{\phi}, e^{\Rcal t'} \Delta_q^h C_{\phi}}_{L^2}} dt'\\ &\leq 2^{\frac{q}{2}} \sum_{q'\geq q-3}\int_0^t e^{2\Rcal t'} \Vert \DD^h_{q'} B_{\phi}\Vert_{L^{2}_h(L^{\infty}_v)} \Vert \Delta_{q'}^h \pa_y A_{\phi}  \Vert_{L^{2}} \Vert \Delta_{q}^h  C_{\phi} \Vert_{L^{2}}dt'.
\end{align*}
Similar calculations as in \eqref{eq:normK1} imply
\begin{align*}
	\Vert \Delta_q^h B_{\phi}(t) \Vert_{L^{\infty}_v(L^2_h)} \leq \int_{0}^{1} \Vert \Delta_q^h \pa_x A_{\phi}(t,.,y') \Vert_{L_{h}^2} dy' \lesssim  2^{q} \int_{0}^{1} \Vert \Delta_q^h  A_{\phi}(t,.,y') \Vert_{L_h^2} dy' \lesssim  2^{q} \Vert \Delta_q^h  A_{\phi}(t) \Vert_{L^2}
\end{align*}

Since $ 0<s\leq 1 $, we have
\begin{align*}
	 K_{3,q} &\leq 2^{\frac{q}{2}} \sum_{q'\geq q-3}\int_0^t   2^{q} \Vert e^{\Rcal t'}\Delta_{q'}^h  A_{\phi}(t') \Vert_{L^2}  \Vert \Delta_{q'}^h  \pa_y A_{\phi}  \Vert_{L^{2}} \Vert \Delta_{q}^h e^{\Rcal t'} C_{\phi} \Vert_{L^{2}}dt' \\
	&\leq  2^{\frac{q}{2}} \sum_{q'\geq q-3}\int_0^t 2^{\frac{q'}{2}} \Vert e^{\Rcal t'}\Delta_{q'}^h  A_{\phi}(t') \Vert_{L^2}  \Vert  \pa_y A_{\phi}  \Vert_{\mathcal{B}^\f12} \Vert \Delta_{q}^h e^{\Rcal t'} C_{\phi} \Vert_{L^{2}}dt' \\
&\lesssim 2^{\frac{q}{2}} \sum_{q'\geq q-3} 2^{\frac{q'}{2}} \left(\int_0^t \Vert \pa_y A_{\phi}(t') \Vert_{\mathcal{B}^{\f12}}  \Vert e^{\Rcal t'}\Delta_{q'}^h A_{\phi} \Vert_{L^2}^2 dt'\right)^{\f12} \times \left(\int_0^t \Vert  \pa_y A_{\phi}(t') \Vert_{\mathcal{B}^{\f12}} \Vert e^{\Rcal t'} \Delta_q^h C_{\phi} \Vert_{L^2}^2 dt'\right)^{\f12}.
\end{align*}

we note that $\dot{\theta}(t) \simeq \Vert  \pa_y A_{\phi}(t') \Vert_{\mathcal{B}^{\f12}}$, using the definition \eqref{def:CLweight}, we have 
\begin{align*}
	\left(\int_0^t \Vert \pa_y A_{\phi}(t') \Vert_{\mathcal{B}^{\f12}} e^{2\Rcal t'} \Vert \Delta_q^h C_{\phi} \Vert_{L^2}^2 dt'\right)^{\frac{1}{2}} &\lesssim \left(\int_0^t \dot{\theta}(t') e^{2\Rcal t'} \Vert \Delta_q^h C_{\phi} \Vert_{L^2}^2 dt'\right)^{\frac{1}{2}} \\ &\lesssim 2^{-q(s+\frac{1}{2})} d_q(C_\phi) \Vert e^{\Rcal t} C_{\phi} \Vert_{\tilde{L}^{2}_{t,\dot{\theta}}(\mathcal{B}^{s+\frac{1}{2}})}.
\end{align*}
Then,
\begin{align} \label{eq:K_3,q}
	K_{3,q} \lesssim 2^{-2qs} d_q^2 \Vert e^{\Rcal t} A_{\phi} \Vert_{\tilde{L}^{2}_{t,\dot{\theta}}(\mathcal{B}^{s+\frac{1}{2}})} \Vert e^{\Rcal t} C_{\phi} \Vert_{\tilde{L}^{2}_{t,\dot{\theta}}(\mathcal{B}^{s+\frac{1}{2}})},
\end{align}

where
\begin{align*}
	d_q^2 = d_q(C_\phi) \left( \sum_{q'\geq q-3}d_{q'}(A_\phi) \, 2^{(q-q')s} \right)
\end{align*}
is a summable sequence of positive constants. Summing with respect to $q \in \ZZ$, and using Fubini's theorem, we get
\begin{align} \label{eq:K3q}
	\sum_{q\in\ZZ} 2^{2qs} K_{3,q} \lesssim C\Vert e^{\Rcal t} B_{\phi} \Vert_{\tilde{L}^2_{t,\dot{\theta}}(\mathcal{B}^{s+\frac{1}{2}})} \Vert e^{\Rcal t} B_{\phi} \Vert_{\tilde{L}^2_{t,\dot{\theta}}(\mathcal{B}^{s+\frac{1}{2}})},
\end{align}
Lemma \ref{lem C+B+A} is then proved by summing Estimates \eqref{eq:K1q}, \eqref{eq:K2q} and \eqref{eq:K3q}. \hfill 
\end{proof}

\begin{lemma} \label{lem:cbb}
	For any $s\in ]0,1]$ and $t\leq T^{\ast}$, there exist $C \geq 1$ such that,
	\begin{align}
	\label{eq:bub}
	\sum_{q\in\ZZ} 2^{2qs} \int_0^t \abs{\psca{e^{\mathcal{R}t'} \Delta_q^h (b\p_x u)_{\phi}, e^{\mathcal{R}t'} \Delta_q^h b_{\phi}}_{L^2}} dt' \leq C \Vert e^{\mathcal{R}t} u_{\phi} \Vert_{\tilde{L}^2_{t,\dot{\theta}(t)}(\mathcal{B}^{s+\frac{1}{2}})} \Vert e^{\mathcal{R}t} b_{\phi} \Vert_{\tilde{L}^2_{t,\dot{\theta}(t)}(\mathcal{B}^{s+\frac{1}{2}})}.
	\end{align}
	and
	\begin{align}
	\label{eq:cub}
	\sum_{q\in\ZZ} 2^{2qs} \int_0^t \abs{\psca{ e^{\Rcal t'} \Delta_q^h (c\p_y u)_{\phi}, e^{\Rcal t'}\Delta_q^h b_{\phi}}_{L^2} dt'}  \leq C \norm{e^{\Rcal t} b_\phi}_{\tilde{L}^2_{t,\dot{\theta}(t)}(\mathcal{B}^{s+\frac{1}{2}})}^2.
	\end{align}
\end{lemma}

\begin{proof}
At first, we will prove Estimate \eqref{eq:bub} of Lemma \ref{lem:cbb}. We define the time-dependent function $L(t)$ 
$$ L(t) = \int_0^t \abs{\psca{e^{\mathcal{R}t'} \Delta_q^h (b\p_x u)_{\phi}, e^{\mathcal{R}t'} \Delta_q^h b_{\phi}}_{L^2}} dt'.$$

Bony's decomposition for the horizontal variable into the para-products $b\pa_x u$ implies
\begin{align}
	\label{eq:B123q} \int_0^t \abs{\psca{e^{\mathcal{R}t'} \Delta_q^h (b\p_x u)_{\phi}, e^{\mathcal{R}t'} \Delta_q^h b_{\phi}}_{L^2}} dt' \leq L_{1,q} + L_{2,q} + L_{3,q},
\end{align}
with
\begin{align*} 
	L_{1,q} &= \int_0^t \abs{\psca{ e^{\mathcal{R}t'} \Delta_q^h ( T^h_{b}\pa_x u)_{\phi}, e^{\mathcal{R}t'} \Delta_q^h b_{\phi}}_{L^2}} dt'\\
	L_{2,q} &= \int_0^t \abs{\psca{ e^{\mathcal{R}t'} \Delta_q^h ( T^h_{\pa_xu} b )_{\phi}, e^{\mathcal{R}t'} \Delta_q^h b_{\phi}}_{L^2}} dt'\\
	L_{3,q} &= \int_0^t \abs{\psca{ e^{\mathcal{R}t'} \Delta_q^h (R^h(b,\pa_xu))_{\phi}, e^{\mathcal{R}t'} \Delta_q^h b_{\phi}}_{L^2}} dt'.
\end{align*}
We start by getting the estimate of the first term $L_{1,q}$, Using the support properties given in [\cite{B1981}, Proposition 2.10] and the definition of $T^{h}_{b} \pa_x u$, we infer 
\begin{equation} \label{eq:normbub}
	L_{1,q} \leq \sum_{|q-q'| \leq 4} \int_0^t e^{2\Rcal t'} \Vert S^h_{q'-1} b_{\phi} (t') \Vert_{L^{\infty}}\Vert \Delta_{q'}^h \partial_x u_{\phi} (t') \Vert_{L^2} \Vert \Delta_q^h b_{\phi} (t') \Vert_{L^2} dt'.
\end{equation}

By the poincarré's inequality on the interval $\lbrace 0<y<1 \rbrace$, we have the inclusion $\dot{H}^1_y \hookrightarrow L^\infty_y$ and 
\begin{align} \label{eq:normb12}
	\Vert \Delta^h_q b_{\phi}(t') \Vert_{L^{\infty}} \lesssim 2^{\frac{q}{2}} \Vert \Delta_q^h b_{\phi}(t') \Vert_{L^2_h(L^{\infty}_v)} \lesssim 2^{\frac{q}{2}} \Vert \Delta_q^h \pa_y b_{\phi}(t') \Vert_{L^2} \lesssim d_q(b_\phi) \Vert  \pa_y b_{\phi}(t') \Vert_{\mathcal{B}^{\f12}},
\end{align}
where $\lbrace d_q(b_{\phi}) \rbrace$ is a square-summable sequence with $\sum d_q(u_\phi)^2 = 1$. Then,
\begin{align*}
	\Vert S^h_{q'-1} b_{\phi} (t') \Vert_{L^{\infty}} \lesssim  \Vert  \pa_y b_{\phi}(t') \Vert_{\mathcal{B}^{\f12}},
\end{align*}
then, we replace this result in our estimate \eqref{eq:normbub}, and combining with H\"older inequality, imply that 

\begin{align*}
	L_{1,q} &\lesssim \sum_{|q-q'| \leq 4}  \int_0^t \Vert \pa_y b_{\phi}(t') \Vert_{\mathcal{B}^{\f12}} e^{\Rcal t'} \Vert \Delta_{q'}^h \pa_x u_{\phi}(t') \Vert_{L^2} e^{\Rcal t'} \Vert \Delta_q^h b_{\phi}(t') \Vert_{L^2} dt'\\
	&\lesssim \sum_{|q-q'| \leq 4} 2^{q'} \int_0^t \Vert \pa_y b_{\phi}(t') \Vert_{\mathcal{B}^{\f12}}^\f12 e^{\Rcal t'} \Vert \Delta_{q'}^h u_{\phi}(t') \Vert_{L^2} \Vert \pa_y b_{\phi}(t') \Vert_{\mathcal{B}^{\f12}}^\f12e^{\Rcal t'} \Vert \Delta_q^h b_{\phi}(t') \Vert_{L^2} dt'\\
	&\lesssim \sum_{|q-q'| \leq 4} 2^{q'} \left(\int_0^t \Vert \pa_y b_{\phi}(t') \Vert_{\mathcal{B}^{\f12}} e^{2\Rcal t'} \Vert \Delta_{q'}^h u_{\phi} \Vert_{L^2}^2 dt'\right)^{\f12} \left(\int_0^t \Vert  \pa_y b_{\phi}(t') \Vert_{\mathcal{B}^{\f12}} e^{2\Rcal t'} \Vert \Delta_q^h b_{\phi} \Vert_{L^2}^2 dt'\right)^{\f12}.
\end{align*}
we note that $\dot{\theta}(t) \simeq \Vert  \pa_y b_{\phi}(t') \Vert_{\mathcal{B}^{\f12}}$, using the definition \eqref{def:CLweight}, we have 

\begin{align*}
	\left(\int_0^t \Vert \pa_y b_{\phi}(t') \Vert_{\mathcal{B}^{\f12}} e^{2\Rcal t'} \Vert \Delta_{q'}^h u_{\phi} \Vert_{L^2}^2 dt'\right)^{\frac{1}{2}} &\lesssim \left(\int_0^t \dot{\theta}(t') e^{2\Rcal t'} \Vert \Delta_{q'}^h u_{\phi} \Vert_{L^2}^2 dt'\right)^{\frac{1}{2}} \\ &\lesssim 2^{-q'(s+\frac{1}{2})} d_{q'}(u_\phi) \Vert e^{\Rcal t} u_{\phi} \Vert_{\tilde{L}^{2}_{t,\dot{\theta}}(\mathcal{B}^{s+\frac{1}{2}})}.
\end{align*}
Then,
\begin{align} \label{eq:L_1,q}
	L_{1,q} \lesssim 2^{-2qs} d_q^2 \Vert e^{\Rcal t} b_{\phi} \Vert_{\tilde{L}^{2}_{t,\dot{\theta}}(\mathcal{B}^{s+\frac{1}{2}})} \Vert e^{\Rcal t} u_{\phi} \Vert_{\tilde{L}^{2}_{t,\dot{\theta}}(\mathcal{B}^{s+\frac{1}{2}})},
\end{align}
where 
$$ d_q^2 = d_q(b_\phi) \left( \sum_{|q-q'| \leq 4} d_{q'}(u_\phi) 2^{(q-q')(s-\f12)} \right) $$
if we multiply \eqref{eq:L_1,q} by $2^{2qs}$ and summing with respect to $q\in \mathbb{Z}$, we get 
\begin{align} \label{eq:L1q}
	\sum_{q\in\ZZ} 2^{2qs} L_{1,q} \lesssim C\Vert e^{\Rcal t} b_{\phi} \Vert_{\tilde{L}^{2}_{t,\dot{\theta}}(\mathcal{B}^{s+\frac{1}{2}})} \Vert e^{\Rcal t} u_{\phi} \Vert_{\tilde{L}^{2}_{t,\dot{\theta}}(\mathcal{B}^{s+\frac{1}{2}})}.
\end{align}

\medskip

Using the support properties given in [\cite{B1981}, Proposition 2.10] and the definition of $T^h_{b}\pa_xu$, we can estimate $L_{2,q}$ in a similar way as we did for $L_{1,q}$. As in \eqref{eq:normb12}, we can write
\begin{equation*} 
	\Vert \Delta_{q'}^h b_{\phi} \Vert_{L^2_h(L^{\infty}_v)} \lesssim \Vert \Delta_{q'}^h \pa_y b_{\phi} \Vert_{L^2} \lesssim 2^{-\frac{q'}{2}} d_{q'}(b_\phi) \Vert  \pa_y b_{\phi} \Vert_{\mathcal{B}^{\f12}}.
\end{equation*}
Then,
\begin{align*}
	I_q = \abs{\psca{\Delta_q^h ( T^h_{\pa_x u}b)_{\phi}, \Delta_q^h b_{\phi}}_{L^2}} &\leq \sum_{|q-q'|\leq 4} \Vert S_{q'-1}^h \pa_x u_{\phi} \Vert_{L^{\infty}_h(L^2_v)} \Vert \Delta_{q'}^h b_{\phi} \Vert_{L^{2}_h(L^{\infty}_v)} \Vert \Delta_{q}^h b_{\phi} \Vert_{L^{2}}\\
	&\leq  \sum_{|q-q'|\leq 4} 2^{-\frac{q'}{2}} d_{q'}(b_\phi) \Vert S_{q'-1}^h \pa_x u_{\phi} \Vert_{L^{\infty}_h(L^2_v)} \Vert \pa_y b_{\phi} \Vert_{\mathcal{B}^{\frac{1}{2}}} \Vert \Delta_{q}^h b_{\phi} \Vert_{L^{2}}.
\end{align*}
Since $ 0<s\leq 1 $, we have 
\begin{align*}
	e^{2\Rcal t} I_q &\leq  \sum_{|q-q'|\leq 4} e^{2\Rcal t} 2^{\frac{-q'}{2}} d_{q'}(b_\phi)  \Vert S_{q'-1}^h \pa_x u_{\phi} \Vert_{L^{\infty}_h(L^2_v)} \Vert \pa_y b_{\phi} \Vert_{\mathcal{B}^{\frac{1}{2}}} \Vert \Delta_{q}^h b_{\phi} \Vert_{L^{2}}\\
	&\leq \sum_{|q-q'|\leq 4} e^{2\Rcal t} 2^{\frac{-q'}{2}} d_{q'}(b_\phi) \sum_{l \lesssim q'-2}2^{\frac{3l}{2}}  \Vert \Delta_l^h u_{\phi} \Vert_{L^{2}} \Vert \pa_y b_{\phi} \Vert_{\mathcal{B}^{\frac{1}{2}}} \Vert \Delta_{q}^h b_{\phi} \Vert_{L^{2}}\\
	&\leq \sum_{|q-q'|\leq 4} 2^{\frac{-q'}{2}} d_{q'}(b_\phi) \sum_{l \lesssim q'-2}2^{l(1-s)} d_l(u_\phi) \Vert e^{\Rcal t} u_{\phi} \Vert_{\mathcal{B}^{s+\frac{1}{2}}} \Vert \pa_y b_{\phi} \Vert_{\mathcal{B}^{\frac{1}{2}}}  \Vert \Delta_{q}^h e^{\Rcal t} b_{\phi} \Vert_{L^{2}}\\
	&\leq \sum_{|q-q'|\leq 4} 2^{\frac{-q'}{2}} d_{q'}(b_\phi) 2^{q'(1-s)} \Vert e^{\Rcal t} u_{\phi} \Vert_{\mathcal{B}^{s+\frac{1}{2}}} \Vert \pa_y b_{\phi} \Vert_{\mathcal{B}^{\frac{1}{2}}}  \Vert \Delta_{q}^h e^{\Rcal t}  b_{\phi} \Vert_{L^{2}} \\
	&\leq  \sum_{|q-q'|\leq 4} 2^{\frac{-q'}{2}} d_{q'}(b_\phi) 2^{q'(1-s)} 2^{q(s +\frac{1}{2})} 2^{-q(s+\frac{1}{2})} \Vert e^{\Rcal t} u_{\phi} \Vert_{\mathcal{B}^{s+\frac{1}{2}}} \Vert \pa_y b_{\phi} \Vert_{\mathcal{B}^{\frac{1}{2}}}  \Vert \Delta_{q}^h e^{\Rcal t} b_{\phi} \Vert_{L^{2}}\\
	&\leq d_q^2 2^{-2qs} \Vert e^{\Rcal t} u_{\phi} \Vert_{\mathcal{B}^{s+\frac{1}{2}}} \Vert \pa_y b_{\phi} \Vert_{\mathcal{B}^{\frac{1}{2}}}\Vert e^{\Rcal t} b_{\phi} \Vert_{\mathcal{B}^{s+\frac{1}{2}}}.
\end{align*}
where
\begin{align*}
	d_q^2 = d_q(b_\phi) \left(\sum_{|q-q'|\leq 4}d_{q'} 2^{(q -q')(s-\frac{1}{2})}\right)
\end{align*}
is a summable sequence of positive constants. Summing with respect to $q \in \ZZ$, integrating over $[0,t]$ and using Fubini's theorem, we get
\begin{align} \label{eq:L2q}
	\sum_{q\in\ZZ} 2^{2qs} L_{2,q} = \int_0^t \pare{\sum_{q\in\ZZ} 2^{2qs} e^{2\Rcal t'} I_q} dt' \lesssim \Vert e^{\Rcal t} b_{\phi} \Vert_{\tilde{L}^2_{t,\dot{\theta}}(\mathcal{B}^{s+\frac{1}{2}})}  \Vert e^{\Rcal t} u_{\phi} \Vert_{\tilde{L}^2_{t,\dot{\theta}}(\mathcal{B}^{s+\frac{1}{2}})},
\end{align}
where we recall that $\dot{\theta} (t) = \Vert \pa_y b_{\phi} \Vert_{\mathcal{B}^{\frac{1}{2}}}$.

To end this proof, it remains to estimate $L_{3,q}$. Using the support properties given in [\cite{B1981}, Proposition 2.10], the definition of $R^h(b,\pa_x u)$ and Bernstein lemma \ref{le:Bernstein}, we can write
\begin{align*}
	J_q = \abs{\psca{\Delta_q^h (R^h(b,\pa_x u))_{\phi}, \Delta_q^h b_{\phi}}_{L^2}} &\leq 2^{\frac{q}{2}} \sum_{q'\geq k-3} \Vert \DD^h_{q'} b_{\phi}\Vert_{L^{2}_h(L^{\infty}_v)} \Vert \Delta_{q'}^h \pa_x u_{\phi}  \Vert_{L^{2}} \Vert \Delta_{q}^h  b_{\phi} \Vert_{L^{2}} \\
	&\leq 2^{\frac{q}{2}} \sum_{q'\geq k-3}  2^{q'(1-\frac{1}{2})}  \Vert \pa_y b_{\phi}\Vert_{\mathcal{B}^{\frac{1}{2}}} \Vert \Delta_{q'}^h u_{\phi}  \Vert_{L^{2}} \Vert \Delta_{q}^h b_{\phi} \Vert_{L^{2}}  \\
	&\leq 2^{\frac{q}{2}} \sum_{q'\geq k-3}  2^{\frac{q'}{2}}  \Vert \pa_y b_{\phi}\Vert_{\mathcal{B}^{\frac{1}{2}}} \Vert \Delta_{q'}^h u_{\phi}  \Vert_{L^{2}} \Vert \Delta_{q}^h b_{\phi} \Vert_{L^{2}}.
\end{align*}
Since $ 0<s\leq 1 $, we have
\begin{align*}
	e^{2\Rcal t} J_q &\leq 2^{\frac{q}{2}} \sum_{q'\geq k-3}  2^{\frac{q'}{2}}  \Vert \pa_y b_{\phi}\Vert_{\mathcal{B}^{\frac{1}{2}}} \Vert \Delta_{q'}^h  e^{\Rcal t} u_{\phi}  \Vert_{L^{2}} \Vert \Delta_{q}^h e^{\Rcal t} b_{\phi} \Vert_{L^{2}} \\
	&\leq  2^{\frac{q}{2}} \sum_{q'\geq k-3} 2^{\frac{q'}{2}}  d_{q'}(u_\phi) 2^{-q'(s+\frac{1}{2})} \Vert e^{\Rcal t} u_{\phi} \Vert_{\mathcal{B}^{s+\frac{1}{2}}}  \Vert \pa_y b_{\phi} \Vert_{\mathcal{B}^{\frac{1}{2}}}   d_{q}(b_\phi) 2^{-q(s+\frac{1}{2})} \Vert e^{\Rcal t} b_{\phi} \Vert_{\mathcal{B}^{s+\frac{1}{2}}}  \\
	&\leq d_q(b_\phi) 2^{-2qs} \Vert e^{\Rcal t} u_{\phi} \Vert_{\mathcal{B}^{s+\frac{1}{2}}} \Vert \pa_y b_{\phi} \Vert_{\mathcal{B}^{\frac{1}{2}}}\Vert e^{\Rcal t} b_{\phi} \Vert_{\mathcal{B}^{s+\frac{1}{2}}} \left(  \sum_{q'\geq k-3} d_{q'}(u_\phi) 2^{(q-q')s} \right) \\
	&\leq d_q^2  2^{-2qs} \Vert e^{\Rcal t} u_{\phi} \Vert_{\mathcal{B}^{s+\frac{1}{2}}} \Vert \pa_y b_{\phi} \Vert_{\mathcal{B}^{\frac{1}{2}}}\Vert e^{\Rcal t} b_{\phi} \Vert_{\mathcal{B}^{s+\frac{1}{2}}},
\end{align*}
where
\begin{align*}
d_q^2 = d_q(b_\phi) \left(  \sum_{q'\geq k-3} d_{q'}(u_\phi) 2^{(q-q')s} \right)
\end{align*}
is a summable sequence of positive constants. Summing with respect to $q \in \ZZ$, integrating over $[0,t]$ and using Fubini's theorem, we finally obtain
\begin{align} \label{eq:L3q}
	\sum_{q\in\ZZ} 2^{2qs} L_{3,q} = \int_0^t \pare{\sum_{q\in\ZZ} 2^{2qs} e^{2\Rcal t'} J_q} dt' \lesssim \Vert e^{\Rcal t} u_{\phi} \Vert_{\tilde{L}^2_{t,\dot{\theta}}(\mathcal{B}^{s+\frac{1}{2}})}\Vert e^{\Rcal t} b_{\phi} \Vert_{\tilde{L}^2_{t,\dot{\theta}}(\mathcal{B}^{s+\frac{1}{2}})}.
\end{align}
 then by summing the estimates \eqref{eq:L1q}, \eqref{eq:L2q} and \eqref{eq:L3q}, we achieved the proof of the \eqref{eq:bub} . \hfill 
 
 We will now prove Estimate \eqref{eq:cub}. Using Bony's decomposition for the horizontal variable, we have
\begin{align} \label{eq:C123q}
	\int_0^t \abs{\psca{e^{\mathcal{R}t'} \Delta_q^h (c\p_y u)_{\phi}, e^{\mathcal{R}t'} \Delta_q^h b_{\phi}}_{L^2}} dt' \leq C_{1,q} + C_{2,q} + C_{3,q},
\end{align}
where
\begin{align*}
	C_{1,q} &= \int_0^t \abs{\psca{e^{\mathcal{R}t'} \Delta_q^h ( T^h_{c}\pa_y u)_{\phi}, e^{\mathcal{R}t'} \Delta_q^h b_{\phi}}_{L^2}} dt'\\
	C_{2,q} &= \int_0^t \abs{\psca{e^{\mathcal{R}t'} \Delta_q^h ( T^h_{\pa_yu} c )_{\phi}, e^{\mathcal{R}t'} \Delta_q^h b_{\phi}}_{L^2}} dt'\\ 
	C_{3,q} &= \int_0^t \abs{\psca{e^{\mathcal{R}t'} \Delta_q^h (R^h(c,\pa_yu))_{\phi}, e^{\mathcal{R}t'} \Delta_q^h b_{\phi}}_{L^2}} dt'.
\end{align*}

Identity \eqref{vv} and Bernstein lemma imply
\begin{align}
	\label{eq:cphi}
	\Vert \Delta_q^h c_{\phi}(t) \Vert_{L^{\infty}} \leq \int_{0}^{1} \Vert \Delta_q^h \pa_x b_{\phi}(t,.,y') \Vert_{L_{h}^{\infty}} dy' \lesssim 2^{\frac{3q}{2}} \int_{0}^{1} \Vert \Delta_q^h  b_{\phi}(t,.,y') \Vert_{L_h^2} dy' \lesssim 2^{\frac{3q}{2}} \Vert \Delta_q^h  b_{\phi}(t) \Vert_{L^2},
\end{align}

then we get for $s \leq 1$
\begin{align*}
	e^{2\mathcal{R}t} M_q &= \abs{\psca{ e^{\mathcal{R}t} \Delta_q^h ( T^h_{c}\pa_y u)_{\phi}, e^{\mathcal{R}t} \Delta_q^h b_{\phi}}_{L^2}}\\ & \lesssim \sum_{|q'-q|\leq 4} e^{2\mathcal{R}t} \Vert S_{q'-1}^h c_{\phi} \Vert_{L^{\infty}} \Vert \Delta_{q'}^h \pa_y u_{\phi} \Vert_{L^{2}} \Vert \Delta_q^h b_{\phi} \Vert_{L^{2}}\\ 
	& \lesssim \sum_{|q'-q|\leq 4} e^{2\mathcal{R}t} \Vert S_{q'-1}^h c_{\phi} \Vert_{L^{\infty}} d_{q'}(u_\phi) 2^{-\frac{q'}{2}} \Vert \pa_y u_{\phi} \Vert_{\mathcal{B}^\f12} \Vert \Delta_q^h b_{\phi} \Vert_{L^{2}} \\ 
	&\lesssim \sum_{|q'-q|\leq 4} d_{q'}(u_\phi) 2^{-\frac{q'}{2}} \sum_{l\leq q'-2} 2^{\frac{3l}{2}} \Vert \Delta_l^h e^{\mathcal{R}t} b_{\phi} \Vert_{L^2}\Vert \pa_y u_{\phi} \Vert_{\mathcal{B}^\f12} \Vert \Delta_q^h e^{\mathcal{R}t}b_{\phi} \Vert_{L^{2}} \\ 
	&\lesssim \sum_{|q'-q|\leq 4} d_{q'}(u_\phi) 2^{-\frac{q'}{2}} \sum_{l\leq q'-2} d_l 2^{\frac{3l}{2}} 2^{-l(s+\f12)}\Vert e^{\mathcal{R}t} b_{\phi} \Vert_{\mathcal{B}^{s+\f12}} \Vert \pa_y u_{\phi} \Vert_{\mathcal{B}^\f12} \Vert \Delta_q^h e^{\mathcal{R}t}b_{\phi} \Vert_{L^{2}} \\ 
	&\lesssim \sum_{|q'-q|\leq 4} d_{q'}(u_\phi) 2^{-\frac{q'}{2}} 2^{q'(1-s)}\Vert e^{\mathcal{R}t} b_{\phi} \Vert_{\mathcal{B}^{s+\f12}} \Vert \pa_y u_{\phi} \Vert_{\mathcal{B}^\f12} \Vert \Delta_q^h e^{\mathcal{R}t}b_{\phi} \Vert_{L^{2}} \\
	&\lesssim \sum_{|q'-q|\leq 4} d_{q'}(u_\phi) 2^{-\frac{q'}{2}} 2^{q'(1-s)}\Vert e^{\mathcal{R}t} b_{\phi} \Vert_{\mathcal{B}^{s+\f12}} \Vert \pa_y u_{\phi} \Vert_{\mathcal{B}^\f12} d_q(b_\phi)2^{-q(s+\f12)} \Vert e^{\mathcal{R}t}b_{\phi} \Vert_{\mathcal{B}^{s+\f12}}\\
	&\lesssim d_q^2 2^{-2qs}\Vert \pa_y u_{\phi} \Vert_{\mathcal{B}^\f12} \Vert e^{\mathcal{R}t}b_{\phi} \Vert_{\mathcal{B}^{s+\f12}}^2
\end{align*}
where
\begin{align*}
	d_q^2 = d_q(b_\phi) \left(\sum_{|q'-q|\leq 4} d_{q'}(u_\phi) 2^{(q-q')(s-1)} \right),
\end{align*}
is a summable sequence of positive constants. Taking the sum with respect to $q \in \ZZ$, integrating it over $[0,t]$ and using Fubini's theorem, we arrive to
\begin{align}
	\label{eq:C1q}
	\sum_{q\in\ZZ} 2^{2qs} C_{1,q} = \int_0^t \pare{\sum_{q\in\ZZ} 2^{2qs} e^{2\Rcal t'} M_q} dt' \lesssim C \Vert e^{\Rcal t} b_{\phi} \Vert_{\tilde{L}^2_{t,\dot{\theta}}(\mathcal{B}^{s+\frac{1}{2}})}^2.
\end{align}

\medskip

Now we move to get the estimate of the second terme, by using the support properties given in [\cite{B1981}, Proposition 2.10] and the definition of $T^h_{\pa_y u} c$, we can estimate $C_{2,q}$ in a similar way as we did for $C_{1,q}$. 

\begin{align*}
C_{2,q} (t) &\leq \int_0^t \abs{\psca{ e^{\Rcal t'} \Delta_q^h ( T^h_{\pa_yu} c )_{\phi}, e^{\Rcal t'} \Delta_q^h b_{\phi}}_{L^2}} dt' \\
&\leq \sum_{|q-q'|\leq 4} \int_0^t e^{2\Rcal t'}\Vert S^h_{q'-1} \pa_y u_\phi \Vert_{L^\infty_h(L^2_v)} \Vert \Delta_{q'}^h c_\phi \Vert_{L^2_h(L^\infty_v)} \Vert \Delta_q^h b_\phi \Vert_{L^2} dt'.
\end{align*}

As in \eqref{eq:cphi}, we can write
\begin{equation*} 
	\Vert \Delta_{q'}^h c_{\phi} \Vert_{L^2_h(L^{\infty}_v)} \lesssim  2^{q'} \int_0^y\Vert \Delta_{q'}^h b_{\phi}(t,x,y') \Vert_{L_h^2}dy' \lesssim 2^{q'} \Vert  \Delta_{q'}^h b_{\phi}(t) \Vert_{L^2}.
\end{equation*}

Since $ 0<s\leq 1 $, we have 
\begin{align*}
	 C_{2,q} &\leq  \sum_{|q-q'|\leq 4} \int_0^t e^{2\Rcal t'}\Vert S^h_{q'-1} \pa_y u_\phi \Vert_{L^\infty_h(L^2_v)} \Vert \Delta_{q'}^h c_\phi \Vert_{L^2_h(L^\infty_v)} \Vert \Delta_q^h b_\phi \Vert_{L^2} dt'\\
	&\leq  \sum_{|q-q'|\leq 4} \int_0^t 2^{q'} e^{\Rcal t'}\Vert S^h_{q'-1} \pa_y u_\phi \Vert_{L^\infty_h(L^2_v)} \Vert  \Delta_{q'}^h b_{\phi}(t) \Vert_{L^2}e^{\Rcal t'}\Vert \Delta_q^h b_\phi \Vert_{L^2} dt'\\
	&\leq  \sum_{|q-q'|\leq 4} \int_0^t 2^{q'} e^{\Rcal t'} \Vert \pa_y u_\phi \Vert_{\mathcal{B}^\f12}\Vert  \Delta_{q'}^h b_{\phi}(t) \Vert_{L^2}e^{\Rcal t'}\Vert \Delta_q^h b_\phi \Vert_{L^2} dt'\\
	&\leq \sum_{|q-q'|\leq 4} \int_0^t 2^{q'} \Vert \pa_y u_\phi \Vert_{\mathcal{B}^\f12}\Vert e^{\Rcal t'} \Delta_{q'}^h b_{\phi}(t) \Vert_{L^2}\Vert e^{\Rcal t'}\Delta_q^h b_\phi \Vert_{L^2} dt'\\
	&\lesssim \sum_{|q-q'| \leq 4} 2^{q'} \left(\int_0^t \Vert \pa_y u_{\phi}(t') \Vert_{\mathcal{B}^{\f12}}  \Vert e^{\Rcal t'}\Delta_{q'}^h b_{\phi} \Vert_{L^2}^2 dt'\right)^{\f12} \times \left(\int_0^t \Vert  \pa_y u_{\phi}(t') \Vert_{\mathcal{B}^{\f12}} \Vert e^{\Rcal t'} \Delta_q^h b_{\phi} \Vert_{L^2}^2 dt'\right)^{\f12}.
\end{align*}

we note that $\dot{\theta}(t) \simeq \Vert  \pa_y u_{\phi}(t') \Vert_{\mathcal{B}^{\f12}}$, using the definition \eqref{def:CLweight}, we have 
\begin{align*}
	\left(\int_0^t \Vert \pa_y u_{\phi}(t') \Vert_{\mathcal{B}^{\f12}} e^{2\Rcal t'} \Vert \Delta_q^h b_{\phi} \Vert_{L^2}^2 dt'\right)^{\frac{1}{2}} &\lesssim \left(\int_0^t \dot{\theta}(t') e^{2\Rcal t'} \Vert \Delta_q^h b_{\phi} \Vert_{L^2}^2 dt'\right)^{\frac{1}{2}} \\ &\lesssim 2^{-q(s+\frac{1}{2})} d_q(b_\phi) \Vert e^{\Rcal t} b_{\phi} \Vert_{\tilde{L}^{2}_{t,\dot{\theta}}(\mathcal{B}^{s+\frac{1}{2}})}.
\end{align*}
Then,
\begin{align} \label{eq:C_2,q}
	C_{2,q} \lesssim 2^{-2qs} d_q^2 \Vert e^{\Rcal t} b_{\phi} \Vert_{\tilde{L}^{2}_{t,\dot{\theta}}(\mathcal{B}^{s+\frac{1}{2}})}^2 ,
\end{align}

where
\begin{align*}
	d_q^2 = d_q(b_\phi) \left(\sum_{|q-q'|\leq 4}d_{q'} 2^{(q -q')(s-\frac{1}{2})}\right)
\end{align*}
is a summable sequence of positive constants. Summing with respect to $q \in \ZZ$, and using Fubini's theorem, we get
\begin{align} \label{eq:C2q}
	\sum_{q\in\ZZ} 2^{2qs} C_{2,q} \lesssim C\Vert e^{\Rcal t} b_{\phi} \Vert_{\tilde{L}^2_{t,\dot{\theta}}(\mathcal{B}^{s+\frac{1}{2}})}^2,
\end{align}
where we recall that $\dot{\theta} (t) \simeq \Vert \pa_y u_{\phi} \Vert_{\mathcal{B}^{\frac{1}{2}}}$.

To end this proof, it remains to estimate $C_{3,q}$(is the rest term). Using the support properties given in [\cite{B1981}, Proposition 2.10], the definition of $R^h(B,\pa_y A)$ and Bernstein lemma \ref{le:Bernstein}, we can write
\begin{align*}
	C_{3,q} &= \int_0^t \abs{\psca{ e^{\Rcal t'} \Delta_q^h (R^h(c,\pa_yu))_{\phi}, e^{\Rcal t'} \Delta_q^h b_{\phi}}_{L^2}} dt'\\ &\leq 2^{\frac{q}{2}} \sum_{q'\geq q-3}\int_0^t e^{2\Rcal t'} \Vert \DD^h_{q'} c_{\phi}\Vert_{L^{2}_h(L^{\infty}_v)} \Vert \Delta_{q'}^h \pa_y u_{\phi}  \Vert_{L^{2}} \Vert \Delta_{q}^h  b_{\phi} \Vert_{L^{2}}dt'.
\end{align*}
Similar calculations as in \eqref{eq:normK1} imply
\begin{align*}
	\Vert \Delta_q^h c_{\phi}(t) \Vert_{L^{\infty}_v(L^2_h)} \leq \int_{0}^{1} \Vert \Delta_q^h \pa_x b_{\phi}(t,.,y') \Vert_{L_{h}^2} dy' \lesssim  2^{q} \int_{0}^{1} \Vert \Delta_q^h  b_{\phi}(t,.,y') \Vert_{L_h^2} dy' \lesssim  2^{q} \Vert \Delta_q^h  b_{\phi}(t) \Vert_{L^2}
\end{align*}

Since $ 0<s\leq 1 $, we have
\begin{align*}
	 C_{3,q} &\leq 2^{\frac{q}{2}} \sum_{q'\geq q-3}\int_0^t   2^{q} \Vert e^{\Rcal t'}\Delta_{q'}^h  b_{\phi}(t') \Vert_{L^2}  \Vert \Delta_{q'}^h  \pa_y u_{\phi}  \Vert_{L^{2}} \Vert \Delta_{q}^h e^{\Rcal t'} b_{\phi} \Vert_{L^{2}}dt' \\
	&\leq  2^{\frac{q}{2}} \sum_{q'\geq q-3}\int_0^t 2^{\frac{q'}{2}} \Vert e^{\Rcal t'}\Delta_{q'}^h  b_{\phi}(t') \Vert_{L^2}  \Vert  \pa_y u_{\phi}  \Vert_{\mathcal{B}^\f12} \Vert \Delta_{q}^h e^{\Rcal t'} b_{\phi} \Vert_{L^{2}}dt' \\
&\lesssim 2^{\frac{q}{2}} \sum_{q'\geq q-3} 2^{\frac{q'}{2}} \left(\int_0^t \Vert \pa_y u_{\phi}(t') \Vert_{\mathcal{B}^{\f12}}  \Vert e^{\Rcal t'}\Delta_{q'}^h b_{\phi} \Vert_{L^2}^2 dt'\right)^{\f12} \times \left(\int_0^t \Vert  \pa_y u_{\phi}(t') \Vert_{\mathcal{B}^{\f12}} \Vert e^{\Rcal t'} \Delta_q^h b_{\phi} \Vert_{L^2}^2 dt'\right)^{\f12}.
\end{align*}

we note that $\dot{\theta}(t) \simeq \Vert  \pa_y u_{\phi}(t') \Vert_{\mathcal{B}^{\f12}}$, using the definition \eqref{def:CLweight}, we have 
\begin{align*}
	\left(\int_0^t \Vert \pa_y u_{\phi}(t') \Vert_{\mathcal{B}^{\f12}} e^{2\Rcal t'} \Vert \Delta_q^h b_{\phi} \Vert_{L^2}^2 dt'\right)^{\frac{1}{2}} &\lesssim \left(\int_0^t \dot{\theta}(t') e^{2\Rcal t'} \Vert \Delta_q^h b_{\phi} \Vert_{L^2}^2 dt'\right)^{\frac{1}{2}} \\ &\lesssim 2^{-q(s+\frac{1}{2})} d_q(b_\phi) \Vert e^{\Rcal t} b_{\phi} \Vert_{\tilde{L}^{2}_{t,\dot{\theta}}(\mathcal{B}^{s+\frac{1}{2}})}.
\end{align*}
Then,
\begin{align} \label{eq:C_3,q}
	C_{3,q} \lesssim 2^{-2qs} d_q^2 \Vert e^{\Rcal t} b_{\phi} \Vert_{\tilde{L}^{2}_{t,\dot{\theta}}(\mathcal{B}^{s+\frac{1}{2}})}^2,
\end{align}

where
\begin{align*}
	d_q^2 = d_q(b_\phi) \left( \sum_{q'\geq q-3}d_{q'}(b_\phi) \, 2^{(q-q')s} \right)
\end{align*}
is a summable sequence of positive constants. Summing with respect to $q \in \ZZ$, and using Fubini's theorem, we get
\begin{align} \label{eq:C3q}
	\sum_{q\in\ZZ} 2^{2qs} C_{3,q} \lesssim C \Vert e^{\Rcal t} b_{\phi} \Vert_{\tilde{L}^2_{t,\dot{\theta}}(\mathcal{B}^{s+\frac{1}{2}})}^2.
\end{align}

%
%
%
%
%
%
%
%
%

The proof of Lemma \ref{lem:cbb} is then completed by summing Estimates \eqref{eq:C1q}, \eqref{eq:C2q} and \eqref{eq:C3q} . \hfill
\end{proof}

\section{Global well posedness of the 2D MHD system in a thin strip}

The goal of this section is to prove Theorem \ref{th:2} and to establish the global well-posedness of the system $(\ref{eq:hydroPE})$ with small analytic data. As in the section 3 for any locally bounded function $\varphi$ on $\mathbb{R}_+ \times \mathbb{R}$ and any $f\in L^2(\mathcal{S})$, we define the analyticity in the horizontal variable $x$ by means of the following auxiliary function
\begin{align}\label{analy}
	f_{\varphi}^{\eps}(t,x,y) = \mathcal{F}^{-1}_{\xi\rightarrow x}(\eps^{\varphi(t,\xi)}\widehat{f}^{\eps}(t,\xi,y)).
\end{align}
The width of the analyticity band $\varphi$ is defined by
\begin{equation*}
	\varphi(t,\xi) = (a- \lambda \tau(t))|\xi|,
\end{equation*}
where $\lambda > 0$ with be precised later and $\tau(t)$ will be chosen in such a way that $\varphi(t,\xi) > 0$, for any $(t,\xi) \in \RR_+ \times \RR$ and $\dot{\tau}(t) = \tau'(t) = -\lambda \dot{\varphi}(t) \geq 0$. In our paper, we will choose
\begin{eqnarray}\label{band}
	\dot{\tau}(t) = \Vert \pa_y u_{\varphi}^{\eps}(t) \Vert_{\mathcal{B}^{\frac{1}{2}}} + \eps \Vert \pa_y v_{\varphi}^{\eps}(t) \Vert_{\mathcal{B}^{\frac{1}{2}}} \text{ \ \ \ \ with \ \ \ \ } \tau (0) = 0. 
\end{eqnarray}

In what follows, for the sake of the simplicity, we will neglect the script $\eps$ and write $(u_{\Theta},v_{\Theta},T_{\Theta})$ instead of $(u_{\varphi}^\epsilon,v_{\varphi}^\epsilon,b_{\varphi}^\epsilon,c^\eps_\varphi)$. Direct calculations from \eqref{eq:hydroPE} and \eqref{analy} show that $(u_{\varphi},v_{\varphi},b_{\varphi},c_\varphi)$ satisfies the system:

\begin{equation}\label{S4:eq1}
\quad\left\{\begin{array}{l}
\displaystyle \partial_t u_{\varphi}+ \lambda \dot{\tau}(t) |D_x|u_\varphi + (u\pa_xu)_\varphi + (v\pa_yu)_\varphi - \eps^2 \pa_x^2 u_\varphi - \pa_y^2 u_\varphi + \pa_x p_\varphi = (b\pa_xb)_\varphi + (c\pa_y b)_{\varphi},\ \ \\
\displaystyle \eps^2\left(\partial_t v_{\varphi}+ \lambda \dot{\tau}(t) |D_x|v_\varphi + (u\pa_xv)_\varphi + (v\pa_yv)_\varphi - \eps^2 \pa_x^2 v_\varphi - \pa_y^2 v_\varphi \right) + \pa_y p_\varphi = \eps^2 \left( (b\pa_xc)_\varphi + (c\pa_y c)_{\varphi} \right) ,\\
\displaystyle \partial_t b_{\varphi}+ \lambda \dot{\tau}(t) |D_x|b_\varphi + (u\pa_xb)_\varphi + (v\pa_yb)_\varphi - \eps^2 \pa_x^2 b_\varphi - \pa_y^2 b_\varphi = (b\pa_xu)_\varphi + (c\pa_y u)_{\varphi}\\
\displaystyle \eps^2 \left( \partial_t c_{\varphi}+ \lambda \dot{\tau}(t) |D_x|c_\varphi + (u\pa_xc)_\varphi + (v\pa_yc)_\varphi - \eps^2 \pa_x^2 c_\varphi - \pa_y^2 c_\varphi \right) = \eps^2 \left( (b\pa_xv)_\varphi + (c\pa_y v)_{\varphi} \right)\\
\displaystyle \partial_x u_{\varphi}+\partial_yv_{\varphi}=0  \text{ \ and \ } \pa_x b_\varphi + \pa_y c_\varphi = 0,\\
\displaystyle \left(u_{\varphi}, v_{\varphi}, b_{\varphi}, c_\varphi \right)|_{y=0}=\left(u_{\varphi}, v_{\varphi}, b_{\varphi}, c_\varphi \right)|_{y=1} = 0,\\
\displaystyle \left(u_{\varphi}, v_{\varphi}, b_{\varphi}, c_\varphi \right)|_{t=0}=\left(u_0, v_0,b_0,c_0 \right).
\end{array}\right.
\end{equation}

Where $|D_x|$ denote the Fourier multiplier of the symbol $|\xi|$. In what follows, we recall that we use "C" to denote a generic positive constant which can change from line to line.

Applying the dyadic operator $\DD^h_q$ to the system \eqref{S4:eq1}, then taking the $L^2(\Scal)$ ($\Scal = \mathbb{R}\times ]0,1[$) scalar product of the first, second, third and the fourth  equations of the obtained system with $\Delta_q^h u_{\phi}$, $\Delta_q^h v_{\phi}$ $\Delta_q^h b_{\phi}$ and $\Delta_q^h c_{\phi}$ respectively, we get

\begin{multline*}
	  \psca{ \Delta_q^h \pa_t  (u_{\varphi},\eps v_\varphi), \Delta_q^h (u_{\varphi},\eps v_\varphi) }_{L^2} + \lambda \dot{\tau}(t) \psca{|D_x| \Delta_q^h (u_{\varphi},\eps v_\varphi),\Delta_q^h (u_{\varphi},\eps v_\varphi)}_{L^2} - \psca{\Delta_q^h \pa_y^2 (u_{\varphi},\eps v_\varphi),\Delta_q^h (u_{\varphi},\eps v_\varphi) }_{L^2} \\ - \eps^2 \psca{\Delta_q^h \pa_x^2 (u_{\varphi},\eps v_\varphi),\Delta_q^h (u_{\varphi},\eps v_\varphi) }_{L^2}+ \psca{\Delta_q^h \nabla p_\varphi,\Delta_q^h (u_{\varphi}, v_\varphi)}_{L^2} 
	= -\psca{\Delta_q^h (u\p_x u + v\p_y u)_{\varphi}, \Delta_q^h u_{\varphi})}_{L^2} \\ -\epsilon ^2 \psca{\Delta_q^h (u\p_x v + v\p_y v)_{\varphi}, \Delta_q^h v_{\varphi})}_{L^2}  + \psca{\Delta_q^h (b\p_x b + c\p_yb)_{\varphi},\Delta_q^h u_\varphi}_{L^2} + \eps^2 \psca{\Delta_q^h (b\p_x c + c\p_yc)_{\varphi},\Delta_q^h v_\varphi}_{L^2},  \\
\end{multline*}
and
\begin{multline*}
 \psca{ \Delta_q^h \pa_t  (b_{\varphi},\eps c_\varphi), \Delta_q^h (b_{\varphi},\eps c_\varphi) }_{L^2} + \lambda \dot{\tau}(t) \psca{|D_x| \Delta_q^h (b_{\varphi},\eps c_\varphi),\Delta_q^h (b_{\varphi},\eps c_\varphi)}_{L^2} - \psca{ \Delta_q^h \pa_y^2 (b_{\varphi},\eps c_\varphi),\Delta_q^h (b_{\varphi},\eps c_\varphi) }_{L^2}  \\ - \eps^2 \psca{ \Delta_q^h \pa_x^2 (b_{\varphi},\eps c_\varphi),\Delta_q^h (b_{\varphi},\eps c_\varphi) }_{L^2}
	= -\psca{\Delta_q^h (u\p_x b+v\p_y b)_{\varphi}, \Delta_q^h b_{\varphi}}_{L^2} - \eps^2 \psca{\Delta_q^h (u\p_x c+v\p_y c)_{\varphi}, \Delta_q^h c_{\varphi}}_{L^2} \\ + \psca{\Delta_q^h(b\p_x u + c\p_yu)_{\varphi}, \Delta_q^h b_\varphi}_{L^2} + \eps^2 \psca{\Delta_q^h(b\p_x v + c\p_yv)_{\varphi}, \Delta_q^h c_\varphi}_{L^2}. \hspace{1cm}
\end{multline*}

Thanks to the Dirichlet boundary condition and due to the free divergence of $U$ (its mean that $\divv U = \pa_xu +\pa_y v =0$ ), we get by using the integration by part that 
\begin{align*}
    \psca{\Delta_q^h \na p_\varphi, \Delta_q^h(u_{\varphi}, v_\varphi) }_{L^2} &= - \psca{\Delta_q^h p_\varphi, \Delta_q^h \divv (u_{\varphi}, v_\varphi) }_{L^2}  \\
    & = \psca{\Delta_q^h p_\varphi, \Delta_q^h (\pa_y v_\varphi+ \pa_x u_\varphi)}_{L^2} \\
    &= 0. \ \ \  (\text{ because } \pa_y v_\varphi+ \pa_x u_\varphi =0) 
\end{align*}
we recall that we have by integrating by part that 
\begin{align*}
    \psca{\Delta_q^h \pa_y^2 (u_{\varphi},\eps v_\varphi),\Delta_q^h (u_{\varphi},\eps v_\varphi) }_{L^2} & = - \psca{\Delta_q^h \pa_y (u_{\varphi},\eps v_\varphi),\Delta_q^h \pa_y (u_{\varphi},\eps v_\varphi) }_{L^2} = - \norm{\Delta_q^h \pa_y  (u_{\varphi},\eps v_\varphi)}_{L^2}^2 \\
    \eps^2 \psca{\Delta_q^h \pa_x^2 (u_{\varphi},\eps v_\varphi),\Delta_q^h (u_{\varphi},\eps v_\varphi) }_{L^2} & = - \eps^2 \psca{\Delta_q^h \pa_x (u_{\varphi},\eps v_\varphi),\Delta_q^h \pa_x (u_{\varphi},\eps v_\varphi) }_{L^2} = - \eps^2 \norm{\Delta_q^h \pa_x  (u_{\varphi},\eps v_\varphi)}_{L^2}^2 \\
     \psca{\Delta_q^h \pa_y^2 (b_{\varphi},\eps c_\varphi),\Delta_q^h (b_{\varphi},\eps c_\varphi) }_{L^2} & = - \psca{\Delta_q^h \pa_y (b_{\varphi},\eps c_\varphi),\Delta_q^h \pa_y (b_{\varphi},\eps c_\varphi) }_{L^2} = - \norm{\Delta_q^h \pa_y  (b_{\varphi},\eps c_\varphi)}_{L^2}^2 \\
    \eps^2 \psca{\Delta_q^h \pa_x^2 (b_{\varphi},\eps c_\varphi),\Delta_q^h (b_{\varphi},\eps c_\varphi) }_{L^2} & = - \eps^2 \psca{\Delta_q^h \pa_x (b_{\varphi},\eps c_\varphi),\Delta_q^h \pa_x (b_{\varphi},\eps c_\varphi) }_{L^2} = - \eps^2 \norm{\Delta_q^h \pa_x  (b_{\varphi},\eps c_\varphi)}_{L^2}^2,
\end{align*}

we replace in the obtained estimate we get  
\begin{multline} \label{S5eq5}
	\frac{1}{2} \frac{d}{dt} \Vert \Delta_q^h (u_{\varphi},\eps v_\varphi)(t) \Vert_{L^2}^2 + \lambda \dot{\tau}(t) \norm{|D_x|^{\frac12} \Delta_q^h (u_{\varphi},\eps v_\varphi)}_{L^2}^2 + \Vert \Delta_q^h \pa_y (u_{\varphi},\eps v_\varphi)(t) \Vert_{L^2}^2 \\
	\eps^2 \Vert \Delta_q^h \pa_x (u_{\varphi},\eps v_\varphi)(t) \Vert_{L^2}^2 
	= -\psca{\Delta_q^h (u\p_x u + v\p_y u)_{\varphi}, \Delta_q^h u_{\varphi})}_{L^2} -\epsilon ^2 \psca{\Delta_q^h (u\p_x v + v\p_y v)_{\varphi}, \Delta_q^h v_{\varphi})}_{L^2}\\  + \psca{\Delta_q^h (b\p_x b + c\p_yb)_{\varphi},\Delta_q^h u_\varphi}_{L^2} + \eps^2 \psca{\Delta_q^h (b\p_x c + c\p_yc)_{\varphi},\Delta_q^h v_\varphi}_{L^2},\\
\end{multline}
and
\begin{multline} \label{S5eq6}
	\frac{1}{2} \frac{d}{dt} \Vert \Delta_q^h (b_{\varphi},\eps c_\varphi)(t) \Vert_{L^2}^2 + \lambda \dot{\tau}(t) \norm{|D_x|^{\frac12} \Delta_q^h (b_{\varphi},\eps c_\varphi)}_{L^2}^2 + \Vert \Delta_q^h \pa_y (b_{\varphi},\eps c_\varphi)(t) \Vert_{L^2}^2 +\\
	\eps^2 \Vert \Delta_q^h \pa_x (b_{\varphi},\eps c_\varphi)(t) \Vert_{L^2}^2
	= -\psca{\Delta_q^h (u\p_x b+v\p_y b)_{\varphi}, \Delta_q^h b_{\varphi}}_{L^2} - \eps^2 \psca{\Delta_q^h (u\p_x c+v\p_y c)_{\varphi}, \Delta_q^h c_{\varphi}}_{L^2} \\ + \psca{\Delta_q^h(b\p_x u + c\p_yu)_{\varphi}, \Delta_q^h b_\varphi}_{L^2} + \eps^2 \psca{\Delta_q^h(b\p_x v + c\p_yv)_{\varphi}, \Delta_q^h c_\varphi}_{L^2}. \hspace{1cm}
\end{multline}
Multiplying \eqref{S4eq5} and \eqref{S4eq6} with $e^{2\Rcal t}$ and then integrating with respect to the time variable, we have the result estimates 
\begin{multline} \label{S5eq5bis}
	\norm{e^{\Rcal t} \Delta_q^h (u_{\varphi},\eps v_\varphi)(t)}_{L^\infty_t (L^2)}^2 + \lambda \int_0^t \dot{\tau}(t') \norm{e^{\Rcal t'} |D_x|^{\frac12} \Delta_q^h (u_{\varphi},\eps v_\varphi)}_{L^2}^2 dt' + \norm{e^{\Rcal t} \Delta_q^h \pa_y (u_{\varphi},\eps v_\varphi)(t)}_{L^2_t(L^2)}^2 \\
	+ \eps^2 \norm{e^{\Rcal t} \Delta_q^h \pa_x (u_{\varphi},\eps v_\varphi)(t)}_{L^2_t(L^2)}^2 \leq \norm{\Delta_q^h (u_{\varphi},\eps v_\varphi)(0)}_{L^2}^2 + F_1 + F_2 + F_3 + F_4, \hspace{2cm}
\end{multline}
and
\begin{multline} \label{S5eq6bis}
	\norm{e^{\Rcal t} \Delta_q^h (b_{\varphi},\eps c_\varphi)(t)}_{L^\infty_t (L^2)}^2 + \lambda \int_0^t \dot{\tau}(t') \norm{e^{\Rcal t} |D_x|^{\frac12} \Delta_q^h (b_{\varphi},\eps c_\varphi)}_{L^2}^2 dt' + \norm{e^{\Rcal t} \Delta_q^h \pa_y^ (b_{\varphi},\eps c_\varphi)(t)}_{L^2_t(L^2)}^2\\
	+ \eps^2 \norm{e^{\Rcal t} \Delta_q^h \pa_x (b_{\varphi},\eps c_\varphi)(t)}_{L^2_t(L^2)}^2 \leq \norm{\Delta_q^h (b_{\varphi},\eps c_\varphi)(0)}_{L^2}^2 + F_5 +F_6. \hspace{2.5cm}
\end{multline}

\begin{lemma} \label{lem:vvv}
	For any $s\in ]0,1]$ and $t\leq T^{\ast}$, and $\varphi$ be defined as in \eqref{analy}, with $$\dot{\tau}(t) =  \Vert \pa_y A_{\varphi}^{\eps}(t) \Vert_{\mathcal{B}^{\frac{1}{2}}} + \eps \Vert \pa_y B_{\varphi}^{\eps}(t) \Vert_{\mathcal{B}^{\frac{1}{2}}}.$$ Then, there exists $C \geq 1$ such that, for any $t > 0$, $\varphi(t,\xi) > 0$ and for any $A \in \tilde{L}^2_{t,\dot{\tau}(t)}(\mathcal{B}^{s+\frac{1}{2}})$ that satisfied $B(t,x,y) = - \int_0^t \pa_x A(t,x,s) ds$, we have
	\begin{align*}
		\eps^2 \sum_{q\in\ZZ} 2^{2qs} \int_0^t \abs{ \psca{e^{\mathcal{R}t'}\Delta_q^h (B\p_y B)_{\varphi}, e^{\mathcal{R}t'}\Delta_q^h B_{\varphi}}_{L^2}}dt' \leq C \Vert e^{\mathcal{R}t} (A_{\varphi},\eps B_{\varphi})\Vert_{\tilde{L}^2_{t,\dot{\tau}(t)}(\mathcal{B}^{s+\frac{1}{2}})}^2.
	\end{align*}
\end{lemma}

Next, by using the lemmas \eqref{lem:ApaxA}, \eqref{lem C+B+A}, \eqref{lem:vvv} and \eqref{lem:cbb}

\begin{align*}
	\abs{F_1} &= \abs{\int_0^t \psca{e^{\mathcal{R}t}\Delta_q^h (u\p_x u)_{\varphi}, e^{\mathcal{R}t}\Delta_q^h u_{\varphi}}_{L^2} + \psca{e^{\mathcal{R}t}\Delta_q^h (v\p_y u)_{\varphi}, e^{\mathcal{R}t}\Delta_q^h u_{\varphi}}_{L^2}  dt' } \\
	&\leq C d_q^2 2^{-2qs} \Vert e^{\mathcal{R}t} u_{\varphi} \Vert_{\tilde{L}^2_{t,\dot{\tau}(t)}(\mathcal{B}^{s+\frac{1}{2}})}^2,
\end{align*}
\begin{align*}
	\abs{F_2} &= \eps^2 \abs{\int_0^t  \psca{e^{\mathcal{R}t} \Delta_q^h (u\p_x v)_{\varphi},e^{\mathcal{R}t} \Delta_q^h v_{\varphi}}_{L^2} + \psca{e^{\mathcal{R}t}\Delta_q^h (v\p_y v)_{\varphi}, e^{\mathcal{R}t}\Delta_q^h v_{\varphi}}_{L^2} dt'} \\ &\leq Cd_q^2 2^{-2qs} \Vert e^{\mathcal{R}t} (u_{\varphi},\eps v_{\varphi})\Vert_{\tilde{L}^2_{t,\dot{\tau}(t)}(\mathcal{B}^{s+\frac{1}{2}})}^2,
\end{align*}

\begin{align*}
	\abs{F_3} &= \abs{\int_0^t \psca{e^{\mathcal{R}t}\Delta_q^h (b\p_x b)_{\varphi}, e^{\mathcal{R}t}\Delta_q^h u_{\varphi}}_{L^2} + \psca{e^{\mathcal{R}t}\Delta_q^h (c\p_y b)_{\varphi}, e^{\mathcal{R}t}\Delta_q^h u_{\varphi}}_{L^2}  dt' } \\
	&\leq C d_q^2 2^{-2qs} \Vert e^{\mathcal{R}t} (u_{\varphi},b_\varphi) \Vert_{\tilde{L}^2_{t,\dot{\tau}(t)}(\mathcal{B}^{s+\frac{1}{2}})}^2,
\end{align*}
\begin{align*}
	\abs{F_4} &= \eps^2 \abs{\int_0^t  \psca{e^{\mathcal{R}t} \Delta_q^h (b\p_x c)_{\varphi},e^{\mathcal{R}t} \Delta_q^h v_{\varphi}}_{L^2} + \psca{e^{\mathcal{R}t}\Delta_q^h (c\p_y c)_{\varphi}, e^{\mathcal{R}t}\Delta_q^h v_{\varphi}}_{L^2} dt'} \\ &\leq Cd_q^2 2^{-2qs} \Vert e^{\mathcal{R}t} (b_{\varphi},\eps v_{\varphi})\Vert_{\tilde{L}^2_{t,\dot{\tau}(t)}(\mathcal{B}^{s+\frac{1}{2}})}^2
\end{align*}
and
\begin{align*}
	\abs{F_5} &= \abs{\int_0^t \psca{\Delta_q^h (u\p_x b+v\p_y b)_{\varphi}, \Delta_q^h b_{\varphi}}_{L^2} + \psca{\Delta_q^h(b\p_x u + c\p_yu)_{\varphi}, \Delta_q^h b_\varphi}_{L^2} dt'}\\ &\leq C d_q^2 2^{-2qs} \Vert e^{\Rcal t} (u_\varphi,b_{\varphi} \Vert_{\tilde{L}^2_{t,\dot{\tau}(t)}(\mathcal{B}^{s+\frac{1}{2}})}^2,
\end{align*}
\begin{align*}
	\abs{F_6} &= \eps^2 \abs{\int_0^t \psca{\Delta_q^h (u\p_x c+v\p_y c)_{\varphi}, \Delta_q^h c_{\varphi}}_{L^2} + \psca{\Delta_q^h(b\p_x v + c\p_yv)_{\varphi}, \Delta_q^h c_\varphi}_{L^2} dt'} \\ &\leq C d_q^2 2^{-2qs} \norm{e^{\Rcal t} (u_\varphi,\eps c_\varphi)}_{\tilde{L}^2_{t,\dot{\tau}(t)}(\mathcal{B}^{s+\f12})}^2.
	\end{align*}	
	
Multiplying \eqref{S5eq5bis} and \eqref{S5eq6bis} by $2^{2qs}$ and summing up with respect to $q\in \mathbb{Z}$, we get
	\begin{multline}\label{eq5:uvq}
	\norm{e^{\Rcal t} (u_{\varphi},\eps v_\varphi)(t)}_{L^\infty_t (\mathcal{B}^s)}^2 + \lambda \norm{e^{\Rcal t'} (u_{\varphi},\eps v_\varphi)}_{\tilde{L}^2_{t,\dot{\tau}(t)}(\mathcal{B}^{s+\f12})}^2  + \norm{e^{\Rcal t} \pa_y (u_{\varphi},\eps v_\varphi)(t)}_{L^2_t(\mathbf{B}^s)}^2 \\
	+ \eps^2 \norm{e^{\Rcal t} \pa_x (u_{\varphi},\eps v_\varphi)(t)}_{L^2_t(\mathcal{B}^s)}^2 \leq \norm{ (u_{\varphi},\eps v_\varphi)(0)}_{\mathcal{B}^s}^2  + C \Vert e^{\mathcal{R}t} (u_{\varphi},b_\varphi,\eps v_\varphi) \Vert_{\tilde{L}^2_{t,\dot{\tau}(t)}(\mathcal{B}^{s+\frac{1}{2}})}^2, \hspace{2cm}
	\end{multline}
	and
		\begin{multline}\label{eq5:bcq}
	\norm{e^{\Rcal t} (b_{\varphi},\eps c_\varphi)(t)}_{L^\infty_t (\mathcal{B}^s)}^2 + \lambda \norm{e^{\Rcal t'} (b_{\varphi},\eps c_\varphi)}_{\tilde{L}^2_{t,\dot{\tau}(t)}(\mathcal{B}^{s+\f12})}^2  + \norm{e^{\Rcal t} \pa_y (b_{\varphi},\eps c_\varphi)(t)}_{L^2_t(\mathbf{B}^s)}^2 \\
	+ \eps^2 \norm{e^{\Rcal t} \pa_x (b_{\varphi},\eps c_\varphi)(t)}_{L^2_t(\mathcal{B}^s)}^2 \leq \norm{ (b_{\varphi},\eps c_\varphi)(0)}_{\mathcal{B}^s}^2  + C \Vert e^{\mathcal{R}t} (u_{\varphi},b_\varphi,\eps c_\varphi) \Vert_{\tilde{L}^2_{t,\dot{\tau}(t)}(\mathcal{B}^{s+\frac{1}{2}})}^2. \hspace{2cm}
	\end{multline}
	
Thus, choosing our constant $C$ such that 
\begin{equation}
	\label{eq:C2} C \geq \max\set{4,\frac{1}{2\Rcal}},
\end{equation}
	 and then taking the sum of the two estimate \eqref{eq5:uvq} and \eqref{eq5:bcq}, we obtain
	 
	 \begin{multline*}
	\norm{e^{\Rcal t} \pare{u_{\varphi},b_{\varphi},\eps(v_\varphi,c_\varphi)}}_{\tilde{L}^\infty_t(\Bcal^s)}^2 + \lambda \norm{e^{\Rcal t} \pare{u_{\phi},b_{\phi},\eps(v_\varphi,c_\varphi)}}_{\tilde{L}^2_{t,\dot{\theta}(t)} (\Bcal^{s + \frac12})}^2 + \norm{e^{\Rcal t} \pa_y \pare{u_{\phi},b_{\phi},\eps(v_\varphi,c_\varphi)}}_{\tilde{L}^2_t(\Bcal^s)}^2 \\ +\eps^2 \norm{e^{\Rcal t} \pa_x \pare{u_{\phi},b_{\phi},\eps(v_\varphi,c_\varphi)}}_{\tilde{L}^2_t(\Bcal^s)}^2
	\leq 2C \norm{e^{a\abs{D_x}} (u_0,b_0,\eps v_0,\eps c_0)}_{\mathcal{B}^{s}}^2 + 2C^2 \norm{e^{\mathcal{R}t} \pare{u_{\phi},b_{\phi},\eps(v_\varphi,c_\varphi)}}_{\tilde{L}^2_{t,\dot{\theta}(t)}(\mathcal{B}^{s+\frac{1}{2}})}^2.
\end{multline*}

We set 
\begin{equation} \label{eq:Tstar3} 
	T^\star \stackrel{\tiny def}{=} \sup\set{t>0\ : \ \norm{u_{\varphi}}_{\mathcal{B}^{\frac{1}{2}}} \leq \frac{1}{2C^2} \ \mbox{ and } \ \tau(t) \leq  \frac{a}{\lambda}},
\end{equation}
and we choose the initial data such that 
\begin{align*}
	C\left( \Vert e^{a|D_x|} (u_0,\eps v_0) \Vert_{\mathcal{B}^{\frac12}} + \Vert e^{a|D_x|} (b_0,\eps c_0) \Vert_{\mathcal{B}^{\frac12}} \right) < \min\set{\frac{1}{2C^2},\frac{a}{2\lambda}}.
\end{align*}
The fact that $\tau(0) = 0$ implies already that $T^\star > 0$. If $\lambda = 2 C^2$, for any $0 < t < T^\star$, we have
\begin{multline}\label{eq:u+v+b+c2}
	\Vert e^{\mathcal{R}t} (u_{\varphi},\eps v_{\varphi}) \Vert_{\tilde{L}^{\infty}_t(\mathcal{B}^s)}^2+ \Vert e^{\mathcal{R}t} (b_{\varphi},\eps c_{\varphi}) \Vert_{\tilde{L}^{\infty}_t(\mathcal{B}^s)}^2 +  \norm{e^{\Rcal t} \pa_y (b_{\varphi},\eps c_{\varphi})}_{\tilde{L}^2_t(\Bcal^s)}^2 + \Vert e^{\mathcal{R}t} \pa_y (u_{\varphi},\eps v_{\varphi})\Vert_{\tilde{L}^{2}_t(\mathcal{B}^s)}^2 \\ + \eps^2 \Vert e^{\mathcal{R}t}  \pa_x (u_{\varphi},\eps v_{\varphi}) \Vert_{\tilde{L}^{2}_t(\mathcal{B}^s)}^2 + \eps^2 \Vert e^{\mathcal{R}t}  \pa_x (b_{\varphi},\eps c_{\varphi}) \Vert_{\tilde{L}^{2}_t(\mathcal{B}^s)}^2 \leq 2C \left( \norm{e^{a\abs{D_x}} (u_0,\eps v_0)}_{\mathcal{B}^{s}}^2 + \norm{e^{a\abs{D_x}} (u_0,\eps v_0)}_{\mathcal{B}^{s}}^2 \right).
\end{multline}
From \eqref{eq:u+v+b+c2} and \eqref{eq:C2}, we get that, for any $0 < t < T^\star$, 
\begin{align*}
	\norm{u_{\varphi}}_{\mathcal{B}^{\frac{1}{2}}} &\leq 	\Vert e^{\mathcal{R}t} (u_{\varphi},\eps v_{\varphi}) \Vert_{\tilde{L}^{\infty}_t(\mathcal{B}^s)} \leq  C\left( \Vert e^{a|D_x|} (u_0,\eps v_0) \Vert_{\mathcal{B}^{s}} +\Vert e^{a|D_x|} (b_0,\eps c_0) \Vert_{\mathcal{B}^{s}} \right) \\ &\leq  C\left( \Vert e^{a|D_x|} (u_0,\eps v_0) \Vert_{\mathcal{B}^{\frac12}} + \Vert e^{a|D_x|} (b_0,\eps c_0) \Vert_{\mathcal{B}^{s}} \right) < \frac{1}{2C^2}.
\end{align*}
Now, we recall that we already defined $\dot{\tau}(t) = \Vert \pa_y u_{\varphi}^{\eps}(t) \Vert_{\mathcal{B}^{\frac{1}{2}}} + \eps \Vert \pa_y v_{\varphi}^{\eps}(t) \Vert_{\mathcal{B}^{\frac{1}{2}}}$ with $\tau(0) = 0$. Then, for any $0 < t < T^\star$, Inequality \eqref{eq:u+v+b+c2} yields
\begin{align*}
	\tau (t) &= \int_0^t \left( \Vert \pa_y u_{\varphi}^{\eps}(t) \Vert_{\mathcal{B}^{\frac{1}{2}}} + \eps \Vert \pa_y v_{\varphi}^{\eps}(t) \Vert_{\mathcal{B}^{\frac{1}{2}}} \right) dt' \\
	& \leq \int_0^t  e^{-\mathcal{R}t'} \left(\Vert e^{\mathcal{R}t'} \pa_y u_{\varphi}^{\eps}(t) \Vert_{\mathcal{B}^{\frac{1}{2}}} + \eps \Vert e^{\mathcal{R}t'} \pa_y v_{\varphi}^{\eps}(t) \Vert_{\mathcal{B}^{\frac{1}{2}}}\right) dt' \\
	&\leq \left( \int_0^t  e^{-2\mathcal{R}t'} dt' \right)^{\f12} \left( \int_0^t (\Vert e^{\mathcal{R}t'}\pa_y u_{\varphi}^{\eps}(t) \Vert_{\mathcal{B}^{\frac{1}{2}}} + \eps \Vert e^{\mathcal{R}t'}\pa_y v_{\varphi}^{\eps}(t) \Vert_{\mathcal{B}^{\frac{1}{2}}})^2 dt' \right)^{\f12} \\
	&\leq C \Vert e^{\mathcal{R}t}(\eps\pa_y v_{\varphi}^{\eps},\pa_y u^{\eps}_{\varphi}) \Vert_{\tilde{L}^2_t(\mathcal{B}^{\f12})} \\
	&\leq C \left( \Vert e^{a|D_x|} (u_0,\eps v_0) \Vert_{\mathcal{B}^\f12} + \Vert e^{a|D_x|}  (b_{0},\eps c_0) \Vert_{\mathcal{B}^\f12} \right) < \frac{a}{2\lambda}.
\end{align*}
A continuity argument implies that $T^\star = +\infty$ and we have \eqref{eq:u+v+b+c2} is valid for any $t \in \RR_+$.

\section{The convergence to the system limit MHD}

In this section, we will justify the limit from the scaled an-isotropic MHD system in a 2D striped domain. As in the first section, our main idea is getting the control of the difference between the two solution $(U^\eps,B^\eps)$ and $(U,B)$ of the systems \eqref{eq:hydroPE} and \eqref{eq:hydrolimit}  (respectively), in analytic space with some small initial data. In this end, we introduce  

\begin{align} \label{eq:psiphiq}
	\left\{
	\begin{aligned}
		(\Psi^{1,\eps},\Psi^{2,\eps},q^\eps) &= (u^\eps - u,v^\eps - v,p^\eps-p ), \\
		(\Phi^{1,\eps},\Phi^{2,\eps}) &= (b^\eps -b,c^\eps - c) .
	\end{aligned}
		\right.
\end{align}

Then, systems \eqref{eq:hydroPE} and \eqref{eq:hydrolimit} imply that $(\Psi^{1,\eps},\Psi^{2,\eps},q_{\eps},\Phi^{1,\eps},\Phi^{2,\eps})$ verifies 
\begin{equation}\label{S1eq10}
\quad\left\{\begin{array}{l}
\displaystyle \partial_t \Psi^{1,\eps} - \eps^2\partial_x^2 \Psi^{1,\eps} -\partial_y^2 \Psi^{1,\eps} + \partial_x q^\eps= R^{1,\eps}\ \ \mbox{in} \ \mathcal{S}\times ]0,\infty[,\\
\displaystyle \eps^2\left(\partial_t \Psi^{2,\eps} - \eps^2\partial_x^2 \Psi^{2,\eps}-\partial_y^2 \Psi^{2,\eps} \right)+\partial_y q^\eps = R^{2,\eps},\\
\displaystyle \partial_t \Phi^{1,\eps} - \eps^2\partial_x^2 \Phi^{1,\eps} -\partial_y^2 \Phi^{1,\eps} = R^{3,\eps},\\
\displaystyle \partial_t \Phi^{2,\eps} - \eps^2\partial_x^2 \Phi^{2,\eps}-\partial_y^2 \Phi^{2,\eps} = R^{4,\eps},\\
\displaystyle \partial_x \Psi^{1,\eps}+\partial_y  \Psi^{2,\eps}=0 \text{ \ } \partial_x \Phi^{1,\eps}+\partial_y  \Phi^{2,\eps}=0\\
\displaystyle \left( \Psi^{1,\eps}, \Psi^{2,\eps}, \Phi^{1,\eps}, \Phi^{2,\eps} \right)|_{t=0}=\left(u_0^\eps-u_0, v_0^{\eps} - v_0,b_0^{\eps}- b_0, c^\eps_0 - c \right),\\
\displaystyle \left( \Psi^{1,\eps}, \Psi^{2,\eps}, \Phi^{1,\eps}, \Phi^{2,\eps} \right)|_{y=0}=\left( \Psi^{1,\eps}, \Psi^{2,\eps}, \Phi^{1,\eps}, \Phi^{2,\eps} \right)|_{y=1} = 0,
\end{array}\right.
\end{equation}
where $v_0$ is a function of $u_0$ and $c_0$ is a function of $b_0$, using \eqref{vv} and the remaining terms $R^{i,\eps}$, with $ i =1,2,3,4$, are determined by the rest  
\begin{align} \label{eq:Ri}
	\left\{
	\begin{aligned}
		R^{1,\eps} &= \eps^2 \pa_x^2 u -(u^{\eps} \pa_x u^{\eps} - u\pa_xu )- (v^\eps \pa_y u^\eps - v\pa_y u) + (b^\eps \pa_x b^\eps - b\pa_x b)+ (c^\eps\pa_y b^\eps - c\pa_y b), \\
		R^{2,\eps} &= - \eps^2\left(\pa_t v -\eps^2 \pa_x^2 v -\pa_y^2 v + u^{\eps} \pa_x v^{\eps} + v^\eps \pa_y v^\eps + b^\eps \pa_x c^\eps + c^\eps \pa_y c^\eps \right), \\
		R^{3,\eps} &= \eps^2 \pa_x^2 b -(u^{\eps} \pa_x b^{\eps} - u\pa_xb )- (v^\eps \pa_y b^\eps - v\pa_y b) + (b^\eps \pa_x u^\eps - b\pa_x u)+ (c^\eps\pa_y u^\eps - c\pa_y u), \\
		R^{4,\eps} &= -\eps^2 \left(\pa_t c -\eps^2 \pa_x^2 c - \pa_y^2c +u^{\eps} \pa_x c^{\eps} + v^\eps \pa_y c^\eps + b^\eps \pa_x u^\eps + c^\eps \pa_y v^\eps \right).
	\end{aligned}
	\right.
\end{align}

As $ (\Psi^{1,\eps},\Psi^{2,\eps},q_{\eps},\Phi^{1,\eps},\Phi^{2,\eps})$ satisfies the boundary condition and also the free divergence, therefore these two conditions allows us to write

\begin{align}\label{Psi}
    \Psi^{2,\eps}(t,x,y) &= \int_0^y \pa_y \Psi^{2,\eps}(t,x,s) ds = - \int_0^y \pa_x \Psi^{1,\eps}(t,x,s)ds 
\end{align}
\begin{align}\label{Phi}
    \Phi^{2,\eps}(t,x,y) = \int_0^y \pa_y \Phi^{2,\eps}(t,x,s) ds = - \int_0^y \pa_x \Phi^{1,\eps}(t,x,s)ds.
\end{align}

If we replace $y$ by $1$ in \eqref{Psi} and \eqref{Phi}, we deduce from the incompressibility condition $\partial_x \Psi^{1,\eps}+\partial_y  \Psi^{2,\eps}=0 \text{ \ } \partial_x \Phi^{1,\eps}+\partial_y  \Phi^{2,\eps}=0$ that

\begin{align*}
	\pa_x\int_0^1\Psi^{1,\eps}(t,x,y)\, dy = -\int_0^1 \dd_y \Psi^{2,\eps}(t,x,y)\, dy = \Psi^{2,\eps}(t,x,1) - \Psi^{2,\eps}(t,x,0) = 0 \\
	\pa_x\int_0^1\Phi^{1,\eps}(t,x,y)\, dy = -\int_0^1 \dd_y \Phi^{2,\eps}(t,x,y)\, dy = \Phi^{2,\eps}(t,x,1) - \Phi^{2,\eps}(t,x,0) = 0.
\end{align*}

Now for suitable function $f$, we define
\begin{eqnarray}\label{analytiqueff}
f_{\Theta}(t,x,y) = \mathcal{F}^{-1}_{\xi \rightarrow x}\left(e^{\Theta(t,\xi)} \widehat{f}(t,\xi,y)\right) \text{ \ \ \ where \ \ \ } \Theta(t,\xi) =\left(a-\mu \eta(t)\right)|\xi|,
\end{eqnarray}
where $\mu \geq \lambda$ will be determined later, and $\eta(t)$ is given by 
$$ \eta(t) = \int_0^t \left( \Vert (\pa_y u^{\eps}_{\varphi}, \eps \pa_x u^{\eps}_{\varphi})(t') \Vert_{\mathcal{B}^{\frac{1}{2}}} + \Vert \pa_y u_{\phi} (t') \Vert_{\mathcal{B}^{\frac{1}{2}}} \right) dt'. $$
We can observe that, if we take $c_0$ and $c_1$ small enough in Theorems \ref{th:hydrolimit} and \ref{th:primitive} then $\Theta(t) \geq 0$ and 
$$ 0 \leq \Theta(t,\xi) \leq min\left(\phi(t,\xi),\varphi(t,\xi)\right) .$$
In what follows, for simplicity, we drop the script $\eps$ and we will write $(\Psi^{1}_{\Theta},\Psi^{2}_\Theta,q_{\Theta},\Phi^{1}_\Theta,\Phi^{2}_\Theta)$ instead of $(\Psi^{1,\eps}_{\Theta},\Psi^{2,\eps}_\Theta,q_{\Theta}^\eps,\Phi^{1,\eps}_\Theta,\Phi^{2,\eps}_\Theta)$. Direct calculations show that $(\Psi^{1}_{\Theta},\Psi^{2}_\Theta,q_{\Theta},\Phi^{1}_\Theta,\Phi^{2}_\Theta)$ satisfies
\begin{equation}\label{22}
\quad\left\{\begin{array}{l}
\displaystyle \partial_t \Psi^{1}_{\Theta} + \mu |D_x| \dot{\eta}(t) \Psi^{1}_{\Theta}  - \eps^2\partial_x^2 \Psi^{1}_{\Theta}-\partial_y^2 \Psi^{1}_{\Theta} + \pa_x q_\Theta = R^{1}_\Theta \ \ \mbox{in} \ \mathcal{S}\times ]0,\infty[,\\
\displaystyle \eps^2\left(\partial_t \Psi^{2}_\Theta +\mu |D_x| \dot{\eta}(t) \Psi^{2}_\Theta - \eps^2\partial_x^2 \Psi^{2}_\Theta-\partial_y^2 \Psi^{2}_\Theta\right) + \pa_y q_\Theta  = R^{2}_\Theta,\\
\displaystyle \partial_t \Phi^{1}_{\Theta} + \mu |D_x| \dot{\eta}(t) \Phi^{1}_{\Theta}  - \eps^2\partial_x^2 \Phi^{1}_{\Theta}-\partial_y^2 \Phi^{1}_{\Theta} = R^{3}_\Theta \ \ \mbox{in} \ \mathcal{S}\times ]0,\infty[,\\
\displaystyle \eps^2\left(\partial_t \Phi^{2}_\Theta +\mu |D_x| \dot{\eta}(t) \Phi^{2}_\Theta - \eps^2\partial_x^2 \Phi^{2}_\Theta-\partial_y^2 \Phi^{2}_\Theta\right) = R^{4}_\Theta,\\
\displaystyle \partial_x \Psi^{1}_\Theta+\partial_y  \Psi^{2}_\Theta=0 \text{ \ } \partial_x \Phi^{1}_\Theta+\partial_y  \Phi^{2}_\Theta=0\\
\displaystyle \left( \Psi^{1}_\Theta, \Psi^{2}_\Theta, \Phi^{1}_\Theta, \Phi^{2}_\Theta \right)|_{t=0}=\left(u_0^\eps-u_0, v_0^{\eps} - v_0,b_0^{\eps}- b_0, c^\eps_0 - c \right),\\
\displaystyle \left( \Psi^{1}_\Theta, \Psi^{2}_\Theta, \Phi^{1}_\Theta, \Phi^{2}_\Theta \right)|_{y=0}=\left( \Psi^{1}_\Theta, \Psi^{2}_\Theta, \Phi^{1}_\Theta, \Phi^{2}_\Theta \right)|_{y=1} = 0.
\end{array}\right.
\end{equation}
As in the previous sections, we will use ``$C$'' to denote a generic positive constant which can change from line to line. So, thanks to the theorems \ref{th:hydrolimit} and \ref{th:primitive}, and the proposition \ref{prop} we deduce that 
\begin{equation} \label{5M}
	\Vert (u^\eps_\varphi,b^\eps_\varphi) \Vert_{\tilde{L}^\infty(\mathbb{R}^+;\mathcal{B}^\f12)} + \Vert (u_\phi,b_\phi) \Vert_{\tilde{L}^\infty(\mathbb{R}^+;\mathcal{B}^\f12 \cap \mathcal{B}^\f52)} +\Vert \pa_y (u_\phi,b_\phi) \Vert_{\tilde{L}^2(\mathbb{R}^+;\mathcal{B}^\f12 \cap \mathcal{B}^\f52)} + \Vert (\pa_t (u,b))_\phi \Vert_{\tilde{L}^2(\mathbb{R}^+;\mathcal{B}^\f32)} \leq M,
\end{equation}
where $u^\eps_\Theta$ and $u_\phi$ are respectively determined by \eqref{S4:eq1} and \eqref{S1eq8} and $M \geq 1$ is a constant independent to $\eps$. Now we return to get the proof of the last theorem, for that we start by applying the dyadic operator in the system \eqref{22} and then taking the $L^2(\mathcal{S})$ (such that $\Scal = \mathbb{R}\times ]0,1[$) scalar product of all the equations then we obtain that 

\begin{align*}
&\psca{ \Delta_q^h \pa_t  (\Psi_{\Theta}^1,\eps \Psi_\Theta^2), \Delta_q^h (\Psi_{\Theta}^1,\eps \Psi_\Theta^2) }_{L^2} + \mu \dot{\eta}(t) \psca{|D_x| \Delta_q^h (\Psi_{\Theta}^1,\eps \Psi_\Theta^2),\Delta_q^h (\Psi_{\Theta}^1,\eps \Psi_\Theta^2)}_{L^2} \\
	  &+ \psca{\Delta_q^h \nabla q_\Theta,\Delta_q^h (\Psi_{\Theta}^1,\eps \Psi_\Theta^2)}_{L^2} - \psca{\Delta_q^h \pa_y^2 (\Psi_{\Theta}^1,\eps \Psi_\Theta^2),\Delta_q^h (\Psi_{\Theta}^1,\eps \Psi_\Theta^2) }_{L^2} - \eps^2 \psca{\Delta_q^h \pa_x^2 (\Psi_{\Theta}^1,\eps \Psi_\Theta^2),\Delta_q^h (\Psi_{\Theta}^1,\eps \Psi_\Theta^2) }_{L^2} \\ &= \psca{ \Delta_q^h R_\Theta^1, \Delta_q^h \Psi_{\Theta}^1 }_{L^2} + \psca{ \Delta_q^h R_\Theta^2, \Delta_q^h \Psi_{\Theta}^2 }_{L^2},
\end{align*}
and
\begin{multline*}
 \psca{ \Delta_q^h \pa_t  (\Phi_{\Theta}^1,\eps \Phi_\Theta^2), \Delta_q^h (\Phi_{\Theta}^1,\eps \Phi_\Theta^2) }_{L^2} + \mu \dot{\eta}(t) \psca{|D_x| \Delta_q^h (\Phi_{\Theta}^1,\eps \Phi_\Theta^2),\Delta_q^h (\Phi_{\Theta}^1,\eps \Phi_\Theta^2)}_{L^2} \\- \psca{ \Delta_q^h \pa_y^2 (\Phi_{\Theta}^1,\eps \Phi_\Theta^2),\Delta_q^h (\Phi_{\Theta}^1,\eps \Phi_\Theta^2) }_{L^2}  - \eps^2 \psca{ \Delta_q^h \pa_x^2 (\Phi_{\Theta}^1,\eps \Phi_\Theta^2),\Delta_q^h (\Phi_{\Theta}^1,\eps \Phi_\Theta^2) }_{L^2}
	\\= \psca{ \Delta_q^h R_\Theta^3, \Delta_q^h \Phi_{\Theta}^1 }_{L^2} + \psca{ \Delta_q^h R_\Theta^4, \Delta_q^h \Phi_{\Theta}^2 }_{L^2}. \hspace{1cm}
\end{multline*}

Thanks to the Dirichlet boundary condition and due to the free divergence of $\Psi$ (its mean that $\divv \Psi = \pa_x \Psi^1 +\pa_y \Psi^2 =0$ ), we get by using the integration by part that 
\begin{align*}
    \psca{\Delta_q^h \na q_\Theta, \Delta_q^h(\Psi_{\Theta}^1,\eps \Psi_\Theta^2) }_{L^2} &= - \psca{\Delta_q^h q_\Theta, \Delta_q^h \divv (\Psi_{\Theta}^1,\eps \Psi_\Theta^2) }_{L^2}  \\
    & = \psca{\Delta_q^h q_\Theta, \Delta_q^h (\pa_y \Psi_\Theta^2+ \pa_x \Psi_\Theta^1)}_{L^2} \\
    &= 0. \ \ \  (\text{ because } \pa_y \Psi_\Theta^2+ \pa_x \Psi_\Theta^1 =0) 
\end{align*}

we recall that we have by integrating by part that 
\begin{align*}
    \psca{\Delta_q^h \pa_y^2 (\Psi_{\Theta}^1,\eps \Psi_\Theta^2),\Delta_q^h (\Psi_{\Theta}^1,\eps \Psi_\Theta^2) }_{L^2} & = - \norm{\Delta_q^h \pa_y  (\Psi_{\Theta}^1,\eps \Psi_\Theta^2)}_{L^2}^2 \\
    \eps^2 \psca{\Delta_q^h \pa_x^2 (\Psi_{\Theta}^1,\eps \Psi_\Theta^2),\Delta_q^h (\Psi_{\Theta}^1,\eps \Psi_\Theta^2) }_{L^2} & = - \eps^2 \norm{\Delta_q^h \pa_x  (\Psi_{\Theta}^1,\eps \Psi_\Theta^2)}_{L^2}^2 \\
     \psca{\Delta_q^h \pa_y^2 (\Phi_{\Theta}^1,\eps \Phi_\Theta^2),\Delta_q^h (\Phi_{\Theta}^1,\eps \Phi_\Theta^2) }_{L^2} & = - \norm{\Delta_q^h \pa_y  (\Phi_{\Theta}^1,\eps \Phi_\Theta^2)}_{L^2}^2 \\
    \eps^2 \psca{\Delta_q^h \pa_x^2 (\Phi_{\Theta}^1,\eps \Phi_\Theta^2),\Delta_q^h (\Phi_{\Theta}^1,\eps \Phi_\Theta^2) }_{L^2} &  = - \eps^2 \norm{\Delta_q^h \pa_x  (\Phi_{\Theta}^1,\eps \Phi_\Theta^2)}_{L^2}^2,
\end{align*}

we replace in the obtained estimate and then integrating with respect to the time variable, we have the result estimates 
\begin{multline} \label{S5eq5bis}
	\norm{e^{\Rcal t} \Delta_q^h (\Psi_{\Theta}^1,\eps \Psi_\Theta^2)(t)}_{L^\infty_t (L^2)}^2 + \mu \int_0^t \dot{\eta}(t') \norm{e^{\Rcal t'} |D_x|^{\frac12} \Delta_q^h (\Psi_{\Theta}^1,\eps \Psi_\Theta^2)}_{L^2}^2 dt' + \norm{e^{\Rcal t} \Delta_q^h \pa_y (\Psi_{\Theta}^1,\eps \Psi_\Theta^2)(t)}_{L^2_t(L^2)}^2 \\
	+ \eps^2 \norm{e^{\Rcal t} \Delta_q^h \pa_x (\Psi_{\Theta}^1,\eps \Psi_\Theta^2)(t)}_{L^2_t(L^2)}^2 \leq \norm{\Delta_q^h (\Psi_{\Theta}^1,\eps \Psi_\Theta^2)(0)}_{L^2}^2 + G_1 + G_2, \hspace{2cm}
\end{multline}
and
\begin{multline} \label{S5eq6bis}
	\norm{e^{\Rcal t} \Delta_q^h (\Phi_{\Theta}^1,\eps \Phi_\Theta^2)(t)}_{L^\infty_t (L^2)}^2 + \mu \int_0^t \dot{\eta}(t') \norm{e^{\Rcal t} |D_x|^{\frac12} \Delta_q^h (\Phi_{\Theta}^1,\eps \Phi_\Theta^2)}_{L^2}^2 dt' + \norm{e^{\Rcal t} \Delta_q^h \pa_y^ (\Phi_{\Theta}^1,\eps \Phi_\Theta^2)(t)}_{L^2_t(L^2)}^2\\
	+ \eps^2 \norm{e^{\Rcal t} \Delta_q^h \pa_x (\Phi_{\Theta}^1,\eps \Phi_\Theta^2)(t)}_{L^2_t(L^2)}^2 \leq \norm{\Delta_q^h (\Phi_{\Theta}^1,\eps \Phi_\Theta^2)(0)}_{L^2}^2 + G_3 +G_4, \hspace{2.5cm}
\end{multline}
where 
\begin{align*}
\norm{\Delta_q^h (\Psi_{\Theta}^1,\eps \Psi_\Theta^2)(0)}_{L^2}^2 & = \norm{\Delta_q^h e^{a|D_x|}(u_0^\eps -u_0,\eps(v_0^\eps - v_0))}_{L^2}^2 \\
\norm{\Delta_q^h (\Phi_{\Theta}^1,\eps \Phi_\Theta^2)(0)}_{L^2}^2 & = \norm{\Delta_q^h e^{a|D_x|}(b_0^\eps -b_0,\eps(c_0^\eps - c_0))}_{L^2}^2.
\end{align*}
Next, we claim that $G_i$, $i=1,..4$ satisfied

\begin{multline} \label{eq:G1}
 G_1^q =  \int_0^t \abs{\psca{\Delta_q^h R^1_{\Theta}, \Delta_q^h \Psi^1_{\Theta}}_{L^2}} dt' \lesssim  2^{-q}d_q^2 \eps \Vert \pa_y u_{\Theta} \Vert_{\tilde{L}^2_{t}(\mathcal{B}^{\frac{3}{2}})} \Vert \eps \Psi^1_{\Theta} \Vert_{\tilde{L}^2_{t}(\mathcal{B}^{\frac{3}{2}})} \\ + 2^{-q}d_q^2\left( \Vert u_{\Theta} \Vert_{\tilde{L}^\infty_{t}(\mathcal{B}^{\frac{3}{2}})}^{\f12} \Vert \pa_y \Psi^1_{\Theta} \Vert_{\tilde{L}^2_{t}\mathcal{B}^{\frac{1}{2}})} +\Vert b_{\Theta} \Vert_{\tilde{L}^\infty_{t}(\mathcal{B}^{\frac{3}{2}})}^{\f12} \Vert \pa_y \Phi^1_{\Theta} \Vert_{\tilde{L}^2_{t}\mathcal{B}^{\frac{1}{2}})} \right)  \Vert \Psi^1_{\Theta} \Vert_{\tilde{L}^2_{t,\dot{\eta}(t)}(\mathcal{B}^1)}\\ + 2^{-q}d_q^2\Vert \Psi^1_{\Theta} \Vert_{\tilde{L}^2_{t,\dot{\eta}(t)}(\mathcal{B}^1)}^2 + 2^{-q}d_q^2\Vert \Psi^1_{\Theta} \Vert_{\tilde{L}^2_{t,\dot{\eta}(t)}(\mathcal{B}^1)} \Vert \Phi^1_{\Theta} \Vert_{\tilde{L}^2_{t,\dot{\eta}(t)}(\mathcal{B}^1)} .
\end{multline}
\begin{align} \label{eq:G2}
 G_2^q &=  \int_0^t \abs{\psca{\Delta_q^h R^2_{\Theta}, \Delta_q^h \Psi^2_{\Theta}}_{L^2}} dt'\\
&\lesssim 2^{-q} d_q^2 \Big\lbrace \Vert (\Psi^1_\Theta,\eps \Psi^2_\Theta) \Vert_{\tilde{L}^2_{t,\dot{\eta}(t)}(\mathcal{B}^{1})}  + \eps^2 \Vert \Psi^2_\Theta \Vert_{\tilde{L}^2_{t,\dot{\eta}(t)}(\mathcal{B}^1)} \Big( \Vert \Psi^2_\Theta \Vert_{\tilde{L}^2_{t,\dot{\eta}(t)}(\mathcal{B}^1)} +\Vert \Phi^2_\Theta \Vert_{\tilde{L}^2_{t,\dot{\eta}(t)}(\mathcal{B}^1)} \notag \\ 
&+ \Vert \Phi^1_\Theta \Vert_{\tilde{L}^2_{t,\dot{\eta}(t)}(\mathcal{B}^1)} + \eps^2\Vert (\pa_y \Psi^2_\Theta,\eps \pa_x\Psi^2_\Theta)\Vert_{\tilde{L}^2_{t}(\mathcal{B}^{\f12})}\big( \Vert (\pa_t u)_\Theta \Vert_{\tilde{L}^2_{t}(\mathcal{B}^{\f32})} + \Vert \pa_y u_\Theta \Vert_{\tilde{L}^2_{t}(\mathcal{B}^{\f32})} +  \Vert \pa_y u_\Theta \Vert_{\tilde{L}^2_{t}(\mathcal{B}^{\f52})}\big)   \notag \\
&+\Vert u^\eps_\Theta \Vert_{\tilde{L}^\infty_{t}(\mathcal{B}^{\f12})}^{\f12} \Vert \pa_y u_\Theta \Vert_{\tilde{L}^2_{t}(\mathcal{B}^2)} + \Vert u_\Theta \Vert_{\tilde{L}^\infty_{t}(\mathcal{B}^{\f32})}^\f12(\Vert \pa_y \Psi^2_\Theta \Vert_{\tilde{L}^2_{t}(\mathcal{B}^{\f12})} + \Vert \pa_y u_\Theta \Vert_{\tilde{L}^2_{t}(\mathcal{B}^{\f32})})\notag\\ & +\Vert b_\Theta \Vert_{\tilde{L}^\infty_{t}(\mathcal{B}^{\f32})}^\f12(\Vert \pa_y \Phi^2_\Theta \Vert_{\tilde{L}^2_{t}(\mathcal{B}^{\f12})} + \Vert \pa_y b_\Theta \Vert_{\tilde{L}^2_{t}(\mathcal{B}^{\f32})})+ \Vert b^\eps_\Theta \Vert_{\tilde{L}^\infty_{t}(\mathcal{B}^{\f12})}^{\f12} \Vert \pa_y b_\Theta \Vert_{\tilde{L}^2_{t}(\mathcal{B}^2)}\Big)\Big\rbrace . \notag
\end{align}

\begin{multline} \label{eq:G3}
 G_3^q =  \int_0^t \abs{\psca{\Delta_q^h R^3_{\Theta}, \Delta_q^h \Phi^1_{\Theta}}_{L^2}} dt' \lesssim  C2^{-q}d_q^2 \eps \Vert \pa_y b_{\Theta} \Vert_{\tilde{L}^2_{t}(\mathcal{B}^{\frac{3}{2}})} \Vert \eps \pa_x\Phi^1_{\Theta} \Vert_{\tilde{L}^2_{t}(\mathcal{B}^{\frac{1}{2}})} \\ + C2^{-q}d_q^2\left( \Vert u_{\Theta} \Vert_{L^\infty_{t}(\mathcal{B}^{\frac{3}{2}})}^{\f12} \Vert \pa_y \Phi^1_{\Theta} \Vert_{\tilde{L}^2_{t}\mathcal{B}^{\frac{1}{2}})} +\Vert b_{\Theta} \Vert_{L^\infty_{t}(\mathcal{B}^{\frac{3}{2}})}^{\f12} \Vert \pa_y \Psi^1_{\Theta} \Vert_{\tilde{L}^2_{t}\mathcal{B}^{\frac{1}{2}})} \right)  \Vert \Phi^1_{\Theta} \Vert_{\tilde{L}^2_{t,\dot{\eta}(t)}(\mathcal{B}^1)}\\ + C2^{-q}d_q^2\Vert \Phi^1_{\Theta} \Vert_{\tilde{L}^2_{t,\dot{\eta}(t)}(\mathcal{B}^1)}^2 + C2^{-q}d_q^2\Vert \Psi^1_{\Theta} \Vert_{\tilde{L}^2_{t,\dot{\eta}(t)}(\mathcal{B}^1)} \Vert \Phi^1_{\Theta} \Vert_{\tilde{L}^2_{t,\dot{\eta}(t)}(\mathcal{B}^1)} .
\end{multline}
\begin{align} \label{eq:G4}
 G_4^q &=  \int_0^t \abs{\psca{\Delta_q^h R^4_{\Theta}, \Delta_q^h \Phi^2_{\Theta}}_{L^2}} dt'\\
&\lesssim 2^{-q} d_q^2 \Big\lbrace \Vert (\Psi^1_\Theta,\eps \Phi^2_\Theta) \Vert_{\tilde{L}^2_{t,\dot{\eta}(t)}(\mathcal{B}^{1})}^2  + \eps^2 \Vert \Phi^2_\Theta \Vert_{\tilde{L}^2_{t,\dot{\eta}(t)}(\mathcal{B}^1)} \Big( \Vert \Phi^2_\Theta \Vert_{\tilde{L}^2_{t,\dot{\eta}(t)}(\mathcal{B}^1)} +\Vert \Psi^2_\Theta \Vert_{\tilde{L}^2_{t,\dot{\eta}(t)}(\mathcal{B}^1)}\Big) \notag \\ 
& + \eps^2\Vert (\pa_y \Phi^2_\Theta,\eps \pa_x\Phi^2_\Theta)\Vert_{\tilde{L}^2_{t}(\mathcal{B}^{\f12})}\big( \Vert (\pa_t b)_\Theta \Vert_{\tilde{L}^2_{t}(\mathcal{B}^{\f32})} + \Vert \pa_y b_\Theta \Vert_{\tilde{L}^2_{t}(\mathcal{B}^{\f32})} +  \Vert \pa_y b_\Theta \Vert_{\tilde{L}^2_{t}(\mathcal{B}^{\f52})}\big)   \notag \\
&+\eps^2 \Vert \Phi^2_\Theta \Vert_{\tilde{L}^2_{t,\dot{\eta}(t)}(\mathcal{B}^1)} \Big(\Vert b^\eps_\Theta \Vert_{\tilde{L}^\infty_{t}(\mathcal{B}^{\f12})}^{\f12} \Vert \pa_y u_\Theta \Vert_{\tilde{L}^2_{t}(\mathcal{B}^2)} + \Vert u_\Theta \Vert_{\tilde{L}^\infty_{t}(\mathcal{B}^{\f32})}^\f12(\Vert \pa_y \Phi^2_\Theta \Vert_{\tilde{L}^2_{t}(\mathcal{B}^{\f12})} + \Vert \pa_y b_\Theta \Vert_{\tilde{L}^2_{t}(\mathcal{B}^{\f32})})\notag\\ & +\Vert b_\Theta \Vert_{\tilde{L}^\infty_{t}(\mathcal{B}^{\f32})}^\f12(\Vert \pa_y \Psi^2_\Theta \Vert_{\tilde{L}^2_{t}(\mathcal{B}^{\f12})} + \Vert \pa_y u_\Theta \Vert_{\tilde{L}^2_{t}(\mathcal{B}^{\f32})})+ \Vert u^\eps_\Theta \Vert_{\tilde{L}^\infty_{t}(\mathcal{B}^{\f12})}^{\f12} \Vert \pa_y b_\Theta \Vert_{\tilde{L}^2_{t}(\mathcal{B}^2)}\Big)\Big\rbrace . \notag
\end{align}
By virtue of \eqref{5M}, \eqref{eq:G1}, \eqref{eq:G2}, \eqref{eq:G3} and \eqref{eq:G4}, we infer 
\begin{align} 
\sum_{i=1}^2 G_i^q &= \sum_{i=1}^4 \int_0^t \abs{\psca{\Delta_q^h R^i_{\Theta}, \Delta_q^h \Psi^i_{\Theta}}_{L^2}} dt' \notag \\ &\lesssim  \bigr( M\eps \Vert (\eps \pa_x(\Psi^1_\Theta,\eps \Psi^2_\Theta),\eps \pa_y \Psi^2_\Theta \Vert_{\tilde{L}^2_t(\mathcal{B}^\f12)} \\
&+M^\f12 \Vert \pa_y (\Psi^1_\Theta,\eps \Psi^2_\Theta) \Vert_{\tilde{L}^2_t(\mathcal{B}^\f12)} \Vert (\Psi^1_\Theta,\eps \Psi^2_\Theta) \Vert_{\tilde{L}^2_{t,\dot{\eta}(t)}(\mathcal{B}^1)} \notag\\
&+ M^\f32 \eps \Vert \eps \Psi^2_\Theta \Vert_{\tilde{L}^2_{t,\dot{\eta}(t)}(\mathcal{B}^1)} + \Vert (\Psi^1_\Theta,\eps \Psi^2_\Theta) \Vert_{\tilde{L}^2_{t,\dot{\eta}(t)}(\mathcal{B}^1)}^2 \notag \\ 
& + \Vert (\eps\Psi^2_\Theta,\eps \Phi^2_\Theta) \Vert_{\tilde{L}^2_{t,\dot{\eta}(t)}(\mathcal{B}^1)}^2 + \Vert (\Psi^1_\Theta,\Phi^1_\Theta) \Vert_{\tilde{L}^2_{t,\dot{\eta}(t)}(\mathcal{B}^1)}^2 \notag \\
& + M^\f12 \Vert \pa_y \Phi_\Theta^1 \Vert_{\tilde{L}_t(\mathcal{B}^\f12)}\Vert\Psi^1_\Theta) \Vert_{\tilde{L}^2_{t,\dot{\eta}(t)}(\mathcal{B}^1)} + M^\f12 \eps^\f12 \Vert \pa_y \Phi_\Theta^2 \Vert_{\tilde{L}_t(\mathcal{B}^\f12)}\Vert \eps \Psi^2_\Theta) \Vert_{\tilde{L}^2_{t,\dot{\eta}(t)}(\mathcal{B}^1)} \bigr). \notag
\end{align}

Multiplying the above inequality \eqref{S5eq5bis} and \eqref{S5eq6bis}  by $2^{q}$, and summing the obtained inequalities with respect to $q \in \Z$, we come to 
\begin{align} 
\Vert (\Psi^1_\Theta,\eps &\Psi^2_\Theta) \Vert_{\tilde{L}^\infty_t(\mathcal{B}^\f12)}+ \norm{(\Phi_{\Theta}^1,\eps \Phi_\Theta^2)(t)}_{L^\infty_t (\mathcal{B}^\f12)} +\mu^\f12 \Vert (\Psi^1_\Theta,\eps \Psi^2_\Theta) \Vert_{\tilde{L}^2_{t,\dot{\eta}(t)}(\mathcal{B}^1)} +\mu^\f12 \Vert (\Phi^1_\Theta,\eps \Phi^2_\Theta) \Vert_{\tilde{L}^2_{t,\dot{\eta}(t)}(\mathcal{B}^1)} \notag \\ &+ \Vert \pa_y (\Psi^1_\Theta,\eps \Psi^2_\Theta) \Vert_{\tilde{L}^2_t(\mathcal{B}^\f12)} +\eps \Vert (\Psi^1_\Theta,\eps \Psi^2_\Theta) \Vert_{\tilde{L}^2_t(\mathcal{B}^\f32)}+ \Vert \pa_y (\Phi^1_\Theta,\eps \Phi^2_\Theta) \Vert_{\tilde{L}^2_t(\mathcal{B}^\f12)} +\eps \Vert (\Phi^1_\Theta,\eps \Phi^2_\Theta) \Vert_{\tilde{L}^2_t(\mathcal{B}^\f32)} \notag \\
&\leq C\Vert e^{a|D_x|}(u_0^\eps-u_0, \eps(v_0^\eps  -v_0)) \Vert_{\mathcal{B}^\f12} + C\Vert e^{a|D_x|}(b_0^\eps-b_0, \eps(c_0^\eps  -c_0)) \Vert_{\mathcal{B}^\f12} \notag \\ &\bigr( M\eps \Vert (\eps \pa_x(\Psi^1_\Theta,\eps \Psi^2_\Theta),\eps \pa_y \Psi^2_\Theta \Vert_{\tilde{L}^2_t(\mathcal{B}^\f12)} \\
&+M^\f12 \Vert \pa_y (\Psi^1_\Theta,\eps \Psi^2_\Theta) \Vert_{\tilde{L}^2_t(\mathcal{B}^\f12)} \Vert (\Psi^1_\Theta,\eps \Psi^2_\Theta) \Vert_{\tilde{L}^2_{t,\dot{\eta}(t)}(\mathcal{B}^1)} \notag\\
&+ M^\f32 \eps \Vert \eps \Psi^2_\Theta \Vert_{\tilde{L}^2_{t,\dot{\eta}(t)}(\mathcal{B}^1)} + \Vert (\Psi^1_\Theta,\eps \Psi^2_\Theta) \Vert_{\tilde{L}^2_{t,\dot{\eta}(t)}(\mathcal{B}^1)}^2 \notag \\ 
& + \Vert (\eps\Psi^2_\Theta,\eps \Phi^2_\Theta) \Vert_{\tilde{L}^2_{t,\dot{\eta}(t)}(\mathcal{B}^1)}^2 + \Vert (\Psi^1_\Theta,\Phi^1_\Theta) \Vert_{\tilde{L}^2_{t,\dot{\eta}(t)}(\mathcal{B}^1)}^2 \notag \\
& + M^\f12 \Vert \pa_y \Phi_\Theta^1 \Vert_{\tilde{L}_t(\mathcal{B}^\f12)}\Vert\Psi^1_\Theta) \Vert_{\tilde{L}^2_{t,\dot{\eta}(t)}(\mathcal{B}^1)} + M^\f12 \eps^\f12 \Vert \pa_y \Phi_\Theta^2 \Vert_{\tilde{L}_t(\mathcal{B}^\f12)}\Vert \eps \Psi^2_\Theta) \Vert_{\tilde{L}^2_{t,\dot{\eta}(t)}(\mathcal{B}^1)} \bigr). \notag
\end{align}
Young's inequality leads to 
\begin{align} 
\Vert (\Psi^1_\Theta,\eps &\Psi^2_\Theta) \Vert_{\tilde{L}^\infty_t(\mathcal{B}^\f12)}+ \norm{(\Phi_{\Theta}^1,\eps \Phi_\Theta^2)(t)}_{L^\infty_t (\mathcal{B}^\f12)} +\mu^\f12 \Vert (\Psi^1_\Theta,\eps \Psi^2_\Theta) \Vert_{\tilde{L}^2_{t,\dot{\eta}(t)}(\mathcal{B}^1)} +\mu^\f12 \Vert (\Phi^1_\Theta,\eps \Phi^2_\Theta) \Vert_{\tilde{L}^2_{t,\dot{\eta}(t)}(\mathcal{B}^1)} \notag \\ &+ \Vert \pa_y (\Psi^1_\Theta,\eps \Psi^2_\Theta) \Vert_{\tilde{L}^2_t(\mathcal{B}^\f12)} +\eps \Vert (\Psi^1_\Theta,\eps \Psi^2_\Theta) \Vert_{\tilde{L}^2_t(\mathcal{B}^\f32)}+ \Vert \pa_y (\Phi^1_\Theta,\eps \Phi^2_\Theta) \Vert_{\tilde{L}^2_t(\mathcal{B}^\f12)} +\eps \Vert (\Phi^1_\Theta,\eps \Phi^2_\Theta) \Vert_{\tilde{L}^2_t(\mathcal{B}^\f32)}  \\
&\leq C\Vert e^{a|D_x|}(u_0^\eps-u_0, \eps(v_0^\eps  -v_0)) \Vert_{\mathcal{B}^\f12} + C\Vert e^{a|D_x|}(b_0^\eps-b_0, \eps(c_0^\eps  -c_0)) \Vert_{\mathcal{B}^\f12} \notag \\
&+ CM\left( \eps + \Vert (\Psi^1_\Theta,\eps \Psi^2_\Theta) \Vert_{\tilde{L}^2_{t,\dot{\eta}(t)}(\mathcal{B}^1)}+ \Vert (\Phi^1_\Theta,\eps \Phi^2_\Theta) \Vert_{\tilde{L}^2_{t,\dot{\eta}(t)}(\mathcal{B}^1)} \right). \notag
\end{align}
Then by taking $\mu = CM$, we can complete the proof Theorem \ref{th:3}. \hfill $\square$.
\begin{lemma}
we use the same assertion used in the theorem \ref{th:3}, and the result of the remark \ref{remk} we obtain the convergence of $\Phi^{2,\eps} = c^\eps - c$ to $0$ when $\eps \to 0.$
\end{lemma}
\begin{proof} of the estimate \eqref{eq:G1}. According to \eqref{eq:Ri}, we write 
\begin{align*}
R^1_\Theta &= \left(\eps^2 \pa_x^2 u -(u^{\eps} \pa_x u^{\eps} - u\pa_xu )- (v^\eps \pa_y u^\eps - v\pa_y u) + (b^\eps \pa_x b^\eps - b\pa_x b)+ (c^\eps\pa_y b^\eps - c\pa_y b)\right)_\Theta \\
& = \eps^2 \pa_x^2 u_\Theta -\left(\big( u^{\eps} \pa_x (u^{\eps}-u) + (u^\eps -u)\pa_x u\big) - \big(v^\eps\pa_y(u^\eps - u) + (v^\eps - v)\pa_y u \big)\right)_\Theta \\
&+ \left(\big( b^{\eps} \pa_x (b^{\eps}-b) + (b^\eps -b)\pa_x b\big) - \big(c^\eps\pa_y(b^\eps - b) + (c^\eps - c)\pa_y b \big)\right)_\Theta \\
R^1_\Theta & = Q^1_\Theta + \left(\big( b^{\eps} \pa_x \Phi^1 + \Phi^1\pa_x b\big) - \big(c^\eps\pa_y \Phi^1 + \Phi^2\pa_y b \big)\right)_\Theta
\end{align*}
We have already do the proof of the estimate $\int_0^t \abs{\psca{\Delta_q^h Q^1_{\Theta}, \Delta_q^h \Psi^1_{\Theta}}_{L^2}} dt'$ ( see the proof of the estimate $(4.76)$ in \cite{AN2020}), then 
\begin{align}\label{eq:Q^1}
\sum_{q\in \mathbb{Z}} 2^q \int_0^t \abs{\psca{\Delta_q^h Q^1_{\Theta}, \Delta_q^h \Psi^1_{\Theta}}_{L^2}} dt' \lesssim  \eps \Vert \pa_y u_{\varphi} \Vert_{\tilde{L}^2_{t}(\mathcal{B}^{\frac{3}{2}})} \Vert \eps w^1_{\varphi} \Vert_{\tilde{L}^2_{t}(\mathcal{B}^{\frac{3}{2}})} \notag \\ + \Vert u_{\varphi} \Vert_{\tilde{L}^\infty_{t}(\mathcal{B}^{\frac{3}{2}})}^{\f12} \Vert \pa_y w^1_{\varphi} \Vert_{\tilde{L}^2_{t}\mathcal{B}^{\frac{1}{2}})}  \Vert w^1_{\varphi} \Vert_{\tilde{L}^2_{t,\dot{\eta}(t)}(\mathcal{B}^1)} + \Vert w^1_{\varphi} \Vert_{\tilde{L}^2_{t,\dot{\eta}(t)}(\mathcal{B}^1)}^2
\end{align} 
So, we still have to get the estimate of 
$$ \int_0^t \abs{\psca{\Delta_q^h(b^{\eps} \pa_x \Phi^1 + \Phi^1\pa_x b)_\Theta, \Delta_q^h \Psi^1_\Theta }} dt' \text{ and } \int_0^t \abs{\psca{\Delta_q^h(c^{\eps} \pa_y \Phi^1 + \Phi^2\pa_y b)_\Theta, \Delta_q^h \Psi^1_\Theta }} dt'. $$
Then, we start by giving the estimate of the term $ \int_0^t \abs{\psca{\Delta_q^h(b^{\eps} \pa_x \Phi^1 + \Phi^1\pa_x b)_\Theta, \Delta_q^h \Psi^1_\Theta }} dt'.$ 

It follow from the lemma \ref{lem C+B+A} that 
\begin{align}\label{bepsPhiPsi}
\int_0^t \abs{\psca{\Delta_q^h(b^{\eps} \pa_x \Phi^1)_\Theta, \Delta_q^h \Psi^1_\Theta }} dt' \leq Cd_q^2 2^{-q} \Vert \Phi_\Theta \Vert_{\tilde{L}^2_{t,\dot{\eta}(t)}(\mathcal{B}^1)} \Vert \Psi_\Theta \Vert_{\tilde{L}^2_{t,\dot{\eta}(t)}(\mathcal{B}^1)}.
\end{align}
We note $$ I^q_1 = \int_0^t \abs{\psca{\Phi^1\pa_x b)_\Theta, \Delta_q^h \Psi^1_\Theta }} dt',$$ by applying Bony's decomposition \eqref{eq:Bony} for the horizontal variable to $\Phi^1\pa_x b$ we obtain 
$$ \Phi^1\pa_x b = T^h_{\Phi^1}\pa_xb + T^h_{\pa_xb}\Phi^1 + R^h(\pa_xb,\Phi^1).$$
and then, we have the following bound
\begin{align*}
	I_{1}^q = \int_0^t \abs{\psca{ \Delta_q^h  (T^h_{\Phi^1}\pa_xb + T^h_{\pa_xb}\Phi^1 + R^h(\pa_xb,\Phi^1))_{\Theta},\Delta_q^h \Psi^1_\Theta}_{L^2}} dt' \leq I_{1,1}^q + I_{1,2}^q +I_{1,3}^q
\end{align*}
with
\begin{align*} 
	I_{1,2}^q &= \int_0^t  \abs{\psca{\Delta_q^h  (T^h_{\pa_xb}\Phi^1)_{\Theta},\Delta_q^h \Psi^1_\Theta}_{L^2}} dt'\\
	I_{1,1}^q &= \int_0^t  \abs{\psca{\Delta_q^h  (T^h_{\Phi^1}\pa_xb)_{\Theta},\Delta_q^h \Psi^1_\Theta}_{L^2}} dt'\\
	I_{1,3}^q &= \int_0^t  \abs{\psca{\Delta_q^h  ( R^h(\pa_xb,\Phi^1)_{\Theta},\Delta_q^h \Psi^1_\Theta}_{L^2}} dt'.
\end{align*}
Using the support properties given in [\cite{B1981}, Proposition 2.10] and the definition of $T^h_{\Phi^1}\pa_xb$, we have 
\begin{align*}
	I_{1,1}^q &\leq   \sum_{|q-q'|\leq4} \int_0^t \Vert S_{q'-1}^h \Phi^1_{\Theta} \Vert_{L^\infty_h(L^2_v)} \Vert \Delta_{q'}^h \pa_x b_{\Theta} \Vert_{L^2_h(L^\infty_v)} \Vert \Delta_q^h \Psi^1_{\Theta} \Vert_{L^2} \\
	&\leq \sum_{|q-q'|\leq4} \int_0^t 2^{\frac{q'}{2}} \Vert S_{q'-1}^h \Phi^1_{\Theta} \Vert_{L^2} \Vert \Delta_{q'}^h \pa_x b_{\Theta} \Vert_{L^2_h(L^\infty_v)} \Vert \Delta_q^h \Psi^1_{\Theta} \Vert_{L^2}.
\end{align*}
Since 
$$ \Vert \Delta_{q'}^h \pa_x b_{\Theta} \Vert_{L^2_h(L^\infty_v)} \lesssim d_{q'}(b_\Theta) \Vert b_{\Theta} \Vert_{\mathcal{B}^{\frac{3}{2}}}^{\frac{1}{2}} \Vert \pa_y b_{\Theta} \Vert_{\mathcal{B}^{\frac{1}{2}}}^{\frac{1}{2}},  $$
then,
\begin{align*}
	I_{1,1}^q = \int_0^t \abs{\psca{\Delta_q^h  (T^h_{\Phi^1}\pa_x b)_{\Theta},\Delta_q^h \Psi^1_\Theta}} & \lesssim \sum_{|q-q'|\leq4} \int_0^t 2^{\frac{q'}{2}} \Vert S_{q'-1}^h \Phi^1_{\Theta}\Vert_{L^2} d_{q'}(b_\Theta) \Vert b_{\Theta} \Vert_{\mathcal{B}^{\frac{3}{2}}}^{\frac{1}{2}} \Vert \pa_y b_{\Theta} \Vert_{\mathcal{B}^{\frac{1}{2}}}^{\frac{1}{2}} \Vert  \Delta_q^h \Psi^1_{\Theta} \Vert_{L^2} \\
	&\lesssim \sum_{|q-q'|\leq4}\int_0^t  d_{q'}(b_\Theta) 2^{\frac{q'}{2}} 2^{\frac{-q'}{2}} \Vert \pa_y \Phi^1_\Theta \Vert_{\mathcal{B}^{\frac{1}{2}}} \Vert b_{\Theta} \Vert_{\mathcal{B}^{\frac{3}{2}}}^{\frac{1}{2}} \Vert \pa_y b_{\Theta} \Vert_{\mathcal{B}^{\frac{1}{2}}}^{\frac{1}{2}} 2^{-q} d_q(\Psi^1_\Theta) \Vert \Psi^1_{\Theta} \Vert_{\mathcal{B}^1} \\
	&\lesssim d_q^2 2^{-q} \Vert b_{\Theta} \Vert_{\tilde{L}^\infty_t(\mathcal{B}^{\frac{3}{2}})}^{\frac{1}{2}} \Vert \pa_y \Phi^1_\Theta \Vert_{\tilde{L}^2_t(\mathcal{B}^{\frac{1}{2}})} \left( \int_0^t \Vert \pa_y b_{\Theta} \Vert_{\mathcal{B}^{\frac{1}{2}}} \Vert \Psi^1_{\Theta} \Vert_{\mathcal{B}^1}^2 dt' \right)^\f12 \\
	&\lesssim C d_q^2 2^{-q} \Vert b_{\Theta} \Vert_{\tilde{L}^\infty_t(\mathcal{B}^{\frac{3}{2}})}^{\frac{1}{2}} \Vert \pa_y \Phi^1_\Theta \Vert_{\tilde{L}^2_t(\mathcal{B}^{\frac{1}{2}})}\Vert \Psi^1_{\Theta} \Vert_{\tilde{L}^2_{t,\dot{\eta}(t)}(\mathcal{B}^1)}
\end{align*}
where $\set{d_q^2}$ is a summable sequence, which implies 
\begin{align} \label{eq:I1,1}
	 I_{1,1}^q = \int_0^t \abs{\psca{\Delta_q^h  (T^h_{\Phi^1}\pa_x b)_{\Theta},\Delta_q^h \Psi^1_\Theta}}dt' \lesssim C d_q^2 2^{-q} \Vert b_{\Theta} \Vert_{\tilde{L}^\infty_t(\mathcal{B}^{\frac{3}{2}})}^{\frac{1}{2}} \Vert \pa_y \Phi^1_\Theta \Vert_{\tilde{L}^2_t(\mathcal{B}^{\frac{1}{2}})}\Vert \Psi^1_{\Theta} \Vert_{\tilde{L}^2_{t,\dot{\eta}(t)}(\mathcal{B}^1)}.
\end{align}
While observing that
$$ \Vert \Delta_q^h \pa_x b_{\Theta} \Vert_{L^{\infty}} \leq \sum_{l\leq q-2} 2^{\frac{3l}{2}} \Vert \Delta_l^h b_{\Theta} \Vert_{L^2}^{\frac{1}{2}} \Vert \Delta_l^h \pa_y b_\Theta \Vert_{L^2}^{\frac{1}{2}} \lesssim 2^q \Vert \pa_y b_\Theta \Vert_{\mathcal{B}^{\frac{1}{2}}},$$
so we can deduce 
\begin{align*}
	I_{1,2}^q &= \int_0^t \abs{\psca{\Delta_q^h  (T^h_{\pa_xb}\Phi^1)_{\Theta},\Delta_q^h \Psi^1_\Theta}}  \leq \sum_{|q-q'|\leq 4} \int_0^t  \Vert S_{q'-1}^h \pa_x b_{\Theta} \Vert_{L^{\infty}} \Vert \Delta_{q'}^h \Phi^1_\Theta \Vert_{L^2} \Vert \Delta_{q}^h \Psi^1_\Theta \Vert_{L^2} \\
	&\lesssim \sum_{|q-q'|\leq 4} \int_0^t  2^{q'} \Vert \pa_y b_\Theta \Vert_{\mathcal{B}^{\frac{1}{2}}} \Vert \Delta_{q'}^h \Phi^1_\Theta \Vert_{L^2} \Vert \Delta_{q}^h \Psi^1_\Theta \Vert_{^2} \\
	&\lesssim  \sum_{|q-q'|\leq 4} 2^{q'} \left( \int_0^t \Vert \pa_y b_\Theta \Vert_{\mathcal{B}^{\frac{1}{2}}} \Vert \Delta_{q'}^h \Phi^1_\Theta \Vert_{L^2}^2 dt' \right)^\f12  \left( \int_0^t \Vert \pa_y b_\Theta \Vert_{\mathcal{B}^{\frac{1}{2}}} \Vert \Delta_{q}^h \Psi^1_\Theta \Vert_{L^2}^2 dt' \right)^\f12 
\end{align*}
Using the definition of $\dot{\eta}(t)$ and Definition \ref{def:CLweight} we have
$$ \left( \int_0^t \Vert \pa_y b_\Theta \Vert_{\mathcal{B}^{\frac{1}{2}}} \Vert \Delta_{q}^h \Psi^1_\Theta \Vert_{L^2}^2 dt' \right)^\f12 \lesssim 2^{-q} d_q(\Psi^1_\Theta) \Vert \Psi^1_{\Theta} \Vert_{\tilde{L}^2_{t,\dot{\eta}(t)}(\mathcal{B}^1)}. $$
Then,
\begin{align} \label{eq:I1,2} I_{1,2}^q \lesssim  C2^{-q} d_q^2 \Vert \Psi^1_{\Theta} \Vert_{\tilde{L}^2_{t,\dot{\eta}(t)}(\mathcal{B}^1)}\Vert \Phi^1_{\Theta} \Vert_{\tilde{L}^2_{t,\dot{\eta}(t)}(\mathcal{B}^1)}, \end{align}
where
$$ d_q^2 = d_q(\Psi^1_\Theta) \left(\sum_{|q-q'|\leq 4} d_{q'}(\Phi^1_\Theta) \right) $$

In a similar way, we have 
\begin{align*}
	I_{1,3}^q &= \int_0^t \abs{\psca{\Delta_q^h  (R^h(\Phi^1,\pa_xb))_{\Theta},\Delta_q^h \Psi^1_\Theta}}dt' \lesssim 2^{\frac{q}{2}} \sum_{q'\geq q-3} \int_0^t \Vert \Delta_{q'}^h \Phi^1_\Theta \Vert_{L^2} \Vert \tilde{\Delta}_{q'}^h \pa_x b_\Theta \Vert_{L^2_h(L^\infty_v)} \Vert \Delta_{q}^h \Psi^1_\Theta \Vert_{L^2} dt' \\
	&\lesssim 2^{\frac{q}{2}} \sum_{q'\geq q-3} \int_0^t 2^{\frac{q'}{2}} \Vert \Delta_{q'}^h \Phi^1_\Theta \Vert_{L^2} \Vert \pa_y b_\Theta \Vert_{\mathcal{B}^{\f12}} \Vert \Delta_{q}^h \Psi^1_\Theta \Vert_{L^2} \\
	&\lesssim 2^{\frac{q}{2}} \sum_{q'\geq q-3} 2^{\frac{q'}{2}} \left(\int_0^t  \Vert \pa_y b_\Theta \Vert_{\mathcal{B}^{\f12}} \Vert \Delta_{q'}^h \Phi^1_\Theta \Vert_{L^2}^2 dt'  \right)^\f12 \left(\int_0^t  \Vert \pa_y b_\Theta \Vert_{\mathcal{B}^{\f12}} \Vert \Delta_{q}^h \Psi^1_\Theta \Vert_{L^2}^2 dt'  \right)^\f12.
\end{align*}
Using the definition of $\dot{\eta}(t)$ and Definition \ref{def:CLweight} we have
$$ \left( \int_0^t \Vert \pa_y b_\Theta \Vert_{\mathcal{B}^{\frac{1}{2}}} \Vert \Delta_{q}^h \Psi^1_\Theta \Vert_{L^2}^2 dt' \right)^\f12 \lesssim 2^{-q} d_q(\Psi^1_\Theta) \Vert \Psi^1_{\Theta} \Vert_{\tilde{L}^2_{t,\dot{\eta}(t)}(\mathcal{B}^1)}. $$
Then,
\begin{align} \label{eq:I1,3} I_{1,3}^q \lesssim  C2^{-q} d_q^2 \Vert \Psi^1_{\Theta} \Vert_{\tilde{L}^2_{t,\dot{\eta}(t)}(\mathcal{B}^1)}\Vert \Phi^1_{\Theta} \Vert_{\tilde{L}^2_{t,\dot{\eta}(t)}(\mathcal{B}^1)}, \end{align}
where
$$ d_q^2 = d_q(\Psi^1_\Theta) \left(\sum_{|q-q'|\leq 4} d_{q'}(\Phi^1_\Theta) \right) $$

Summing the estimates \eqref{eq:I1,1}, \eqref{eq:I1,2} and \eqref{eq:I1,3} we obtain 
\begin{align} \label{eq:I1q}
	 I_{1}^q  \lesssim Cd_q^2 2^{-q}\left( \Vert b_{\Theta} \Vert_{\tilde{L}^\infty_t(\mathcal{B}^{\frac{3}{2}})}^{\frac{1}{2}} \Vert \pa_y \Phi^1_\Theta \Vert_{\tilde{L}^2_t(\mathcal{B}^{\frac{1}{2}})} + \Vert \Phi^1_{\Theta} \Vert_{\tilde{L}^2_{t,\dot{\eta}(t)}(\mathcal{B}^1)}\right)\Vert \Psi^1_{\Theta} \Vert_{\tilde{L}^2_{t,\dot{\eta}(t)}(\mathcal{B}^1)}.
\end{align}

Now we move to get the estimate of the term $\int_0^t \abs{\psca{\Delta_q^h(c^{\eps} \pa_y \Phi^1 + \Phi^2\pa_y b)_\Theta, \Delta_q^h \Psi^1_\Theta }} dt' = J_1^q + J_2^q$, we start by $J_1^q$, we recall that $ c^\eps = c+\Phi^2$, then 
\begin{align*}
J_2^q &= \int_0^t \abs{\psca{\Delta_q^h(c^{\eps} \pa_y \Phi^1 )_\Theta, \Delta_q^h \Psi^1_\Theta }} dt' \\ &= \int_0^t \abs{\psca{\Delta_q^h(\Phi^2 \pa_y \Phi^1 )_\Theta, \Delta_q^h \Psi^1_\Theta }} dt' + \int_0^t \abs{\psca{\Delta_q^h(c \pa_y \Phi^1 )_\Theta, \Delta_q^h \Psi^1_\Theta }} dt' \\
&=J_{1,1}^q + J_{1,2}^q 
\end{align*}
 It follow from the lemma \ref{lem:ApaxA} that 
 \begin{align}\label{eq:J11}
 J_{1,1}^q  = \int_0^t \abs{\psca{\Delta_q^h(\Phi^2 \pa_y \Phi^1 )_\Theta, \Delta_q^h \Psi^1_\Theta }} dt' \leq Cd_q^2 2^{-q} \Vert \Phi_\Theta \Vert_{\tilde{L}_{t,\dot{\eta}(t)}^2(\mathcal{B}^1)} \Vert \Psi_\Theta \Vert_{\tilde{L}_{t,\dot{\eta}(t)}^2(\mathcal{B}^1)}.
 \end{align}
 For the second term $J_{1,2}^q$, we apply the Bony's decomposition \eqref{eq:Bony} for the horizontal variable to $c\pa_y\Phi^1$, we obtain 
 $$ c\pa_y\Phi^1 = T^h_c\pa_y \Phi^1 + T^h_{\pa_y \Phi^1}c + R^h(c,\pa_y\Phi^1),$$
 so, we have the following bound 
 $$J_{1,2}^q = \int_0^t \abs{\psca{\Delta_q^h(T^h_c\pa_y \Phi^1 + T^h_{\pa_y \Phi^1}c + R^h(c,\pa_y\Phi^1))_\Theta, \Delta_q^h \Psi^1_\Theta }} dt' \leq J_{1,21}^q + J_{1,22}^q +J_{1,23}^q $$
 where 
 \begin{align*}
 J_{1,21}^q &= \int_0^t \abs{\psca{\Delta_q^h( T^h_c\pa_y \Phi^1 )_\Theta, \Delta_q^h \Psi^1_\Theta }} dt'\\
 J_{1,22}^q &= \int_0^t \abs{\psca{\Delta_q^h(T^h_{\pa_y \Phi^1}c)_\Theta, \Delta_q^h \Psi^1_\Theta }} dt'\\
 J_{1,23}^q &= \int_0^t \abs{\psca{\Delta_q^h(R^h(c,\pa_y\Phi^1) )_\Theta, \Delta_q^h \Psi^1_\Theta }} dt'\\
 \end{align*}
 Using the support properties given in [\cite{B1981}, Proposition 2.10] and the definition of $T^h_c\pa_y \Phi^1$, we have 
\begin{align*}
	J_{1,21}^q \leq   \sum_{|q-q'|\leq4} \int_0^t \Vert S_{q'-1}^h c_{\Theta} \Vert_{L^\infty} \Vert \Delta_{q'}^h \pa_y \Phi^1_{\Theta} \Vert_{L^2} \Vert \Delta_q^h \Psi^1_{\Theta} \Vert_{L^2} dt'.
\end{align*}
Due to $\pa_x b +\pa_y c =0$ and Poincar\'e inequality, we can write $ c(t,x,y) = -\int_0^y \pa_x b(t,x,s) ds$, then we deduce from the lemma \ref{lem:Bernstein} that
Since 
 \begin{align}
 \Vert \Delta_{q'}^h c_{\Theta} \Vert_{L^\infty} &\leq \int_0^1 \Vert \Delta_{q'}^h \pa_x b_{\Theta}(t,x,s) \Vert_{L^\infty_h} ds \notag \\
 &\leq 2^{\frac{3q}{2}} \int_0^1 \Vert \Delta_{q'}^h b_{\Theta}(t,x,s) \Vert_{L^2_h} ds \leq 2^{\frac{3q}{2}} \Vert \Delta_{q'}^h b_{\Theta}(t,x,s) \Vert_{L^2}
 \end{align}
then,
$$ \Vert S_{q'-1}^h c_{\Theta} \Vert_{L^\infty} \leq 2^\frac{q'}{2} \Vert b_\Theta(t') \Vert_{\mathcal{B}^\f32}^\f12 \Vert \pa_y b_\Theta(t') \Vert_{\mathcal{B}^\f12}^\f12, $$ 
from which we infer
\begin{align*}
	J_{1,21}^q &\lesssim \sum_{|q-q'|\leq4} \int_0^t \Vert S_{q'-1}^h c_{\Theta} \Vert_{L^\infty} \Vert \Delta_{q'}^h \pa_y \Phi^1_{\Theta} \Vert_{L^2} \Vert \Delta_q^h \Psi^1_{\Theta} \Vert_{L^2} dt' \\ &\lesssim \sum_{|q-q'|\leq4} \int_0^t 2^\frac{q'}{2} \Vert b_\Theta(t') \Vert_{\mathcal{B}^\f32}^\f12 \Vert \pa_y b_\Theta(t') \Vert_{\mathcal{B}^\f12}^\f12 \Vert \Delta_{q'}^h \pa_y \Phi^1_{\Theta} \Vert_{L^2} \Vert \Delta_q^h \Psi^1_{\Theta} \Vert_{L^2} dt'\\
& \lesssim	\sum_{|q-q'|\leq4} \int_0^t 2^\frac{q'}{2} \Vert b_\Theta(t') \Vert_{\mathcal{B}^\f32}^\f12  \Vert \Delta_{q'}^h \pa_y \Phi^1_{\Theta} \Vert_{L^2}\Vert \pa_y b_\Theta(t') \Vert_{\mathcal{B}^\f12}^\f12  \Vert \Delta_q^h \Psi^1_{\Theta} \Vert_{L^2} dt'\\
&\lesssim \sum_{|q-q'|\leq4} 2^\frac{q'}{2}\Vert b_\Theta(t') \Vert_{L^\infty_t(\mathcal{B}^\f32)}^\f12 \Vert \Delta_{q'}^h \pa_y \Phi^1_{\Theta} \Vert_{L^2_t(L^2)} \left( \int_0^t \Vert \pa_y b_\Theta(t') \Vert_{\mathcal{B}^\f12} \Vert \Delta_q^h \Psi^1_{\Theta} \Vert_{L^2}^2 dt' \right)^\f12 \\
&\lesssim Cd_q^2 2^{-q}\Vert b_\Theta(t') \Vert_{L^\infty_t(\mathcal{B}^\f32)}^\f12 \Vert  \pa_y \Phi^1_{\Theta} \Vert_{\tilde{L}^2_t(\mathcal{B}^\f12)} \Vert \Psi^1_\Theta \Vert_{\tilde{L}^2_{t,\dot{\eta}(t)}(\mathcal{B}^1)}
\end{align*}
where $\set{d_q^2}$ is a summable sequence, which implies 
\begin{align} \label{eq:J1,11}
	 J_{1,21}^q = Cd_q^2 2^{-q}\Vert b_\Theta \Vert_{L^\infty_t(\mathcal{B}^\f32)}^\f12 \Vert  \pa_y \Phi^1_{\Theta} \Vert_{\tilde{L}^2_t(\mathcal{B}^\f12)} \Vert \Psi^1_\Theta \Vert_{\tilde{L}^2_{t,\dot{\eta}(t)}(\mathcal{B}^1)}.
\end{align}

While observing that
$$ \Vert \Delta_q^h c_{\Theta} \Vert_{L^2_h(L^\infty_v)} \leq d_q \Vert b_\Theta(t') \Vert_{\mathcal{B}^\f32}^\f12 \Vert \pa_y b_\Theta(t') \Vert_{\mathcal{B}^\f12}^\f12,$$
so we can deduce 
\begin{align*}
	J_{1,22}^q &= \int_0^t \abs{\psca{\Delta_q^h  (T^h_{\pa_y \Phi^1}c)_{\Theta},\Delta_q^h \Psi^1_\Theta}} dt'\\ & \leq \sum_{|q-q'|\leq 4} \int_0^t  \Vert S_{q'-1}^h \pa_y \Phi^1_{\Theta} \Vert_{L^{\infty}_h(L^2_v)} \Vert \Delta_{q'}^h c_\Theta \Vert_{L^2_h(L^\infty_v)} \Vert \Delta_{q}^h \Psi^1_\Theta \Vert_{L^2} dt' \\
	&\lesssim \sum_{|q-q'|\leq 4} \int_0^t \Vert S_{q'-1}^h \pa_y \Phi^1_{\Theta} \Vert_{L^{\infty}_h(L^2_v)} d_{q'}(b_\Theta) \Vert b_\Theta(t') \Vert_{\mathcal{B}^\f32}^\f12 \Vert \pa_y b_\Theta(t') \Vert_{\mathcal{B}^\f12}^\f12  \Vert \Delta_{q}^h \Psi^1_\Theta \Vert_{L^2} dt'\\
	&\lesssim \sum_{|q-q'|\leq 4}d_{q'}(b_\Theta)\Vert b_\Theta(t') \Vert_{L^\infty_t(\mathcal{B}^\f32)}^\f12 \Vert S_{q'-1}^h \pa_y \Phi^1_{\Theta} \Vert_{L^2_t(L^{\infty}_h(L^2_v))} \left( \int_0^t  \Vert \pa_y b_\Theta(t') \Vert_{\mathcal{B}^\f12} \Vert \Delta_{q}^h \Psi^1_\Theta \Vert_{L^2}^2 dt' \right)^\f12\\
	&\lesssim Cd_q^2 2^{-q} \Vert b_\Theta(t') \Vert_{L^\infty_t(\mathcal{B}^\f32)}^\f12 \Vert \pa_y \Phi^1_{\Theta} \Vert_{\tilde{L}^2_t(\mathcal{B}^\f12)} \Vert \Psi^1_\Theta \Vert_{\tilde{L}^2_{t,\dot{\eta}(t)}(\mathcal{B}^1)}.
\end{align*}
Using the definition of $\dot{\eta}(t)$ and Definition \ref{def:CLweight} we have
$$ \left( \int_0^t \Vert \pa_y b_\Theta \Vert_{\mathcal{B}^{\frac{1}{2}}} \Vert \Delta_{q}^h \Psi^1_\Theta \Vert_{L^2}^2 dt' \right)^\f12 \lesssim 2^{-q} d_q(\Psi^1_\Theta) \Vert \Psi^1_{\Theta} \Vert_{\tilde{L}^2_{t,\dot{\eta}(t)}(\mathcal{B}^1)}. $$
Then,
\begin{align} \label{eq:J1,12} J_{1,22}^q \lesssim  C2^{-q} d_q^2\Vert b_\Theta \Vert_{L^\infty_t(\mathcal{B}^\f32)}^\f12 \Vert \pa_y \Phi^1_{\Theta} \Vert_{\tilde{L}^2_t(\mathcal{B}^\f12)} \Vert \Psi^1_\Theta \Vert_{\tilde{L}^2_{t,\dot{\eta}(t)}(\mathcal{B}^1)}, \end{align}
where
$$ d_q^2 = d_q(\Psi^1_\Theta) \left(\sum_{|q-q'|\leq 4} d_{q'}(b_\Theta) \right) $$

In a similar way, we have 
\begin{align*}
	J_{1,23}^q &= \int_0^t \abs{\psca{\Delta_q^h  (R^h(c,\pa_y \Phi^1))_{\Theta},\Delta_q^h \Psi^1_\Theta}}dt'\\ & \lesssim 2^{\frac{q}{2}} \sum_{q'\geq q-3} \int_0^t \Vert \Delta_{q'}^h \pa_y \Phi^1_\Theta \Vert_{L^2} \Vert \tilde{\Delta}_{q'}^h c_\Theta \Vert_{L^2_h(L^\infty_v)} \Vert \Delta_{q}^h \Psi^1_\Theta \Vert_{L^2} dt' \\
	&\lesssim 2^{\frac{q}{2}} \sum_{q'\geq q-3} \int_0^t \Vert b_\Theta(t') \Vert_{\mathcal{B}^\f32}^\f12 \Vert \pa_y b_\Theta(t') \Vert_{\mathcal{B}^\f12}^\f12  \Vert \Delta_{q'}^h \pa_y \Phi^1_\Theta \Vert_{L^2} \Vert \Delta_{q}^h \Psi^1_\Theta \Vert_{L^2} dt'\\
	&\lesssim 2^{\frac{q}{2}} \sum_{q'\geq q-3} \Vert b_\Theta(t') \Vert_{L^\infty_t(\mathcal{B}^\f32)}^\f12 \Vert \Delta_{q'}^h \pa_y \Phi^1_\Theta \Vert_{L^2_t(L^2)} \left(\int_0^t  \Vert \pa_y b_\Theta \Vert_{\mathcal{B}^{\f12}} \Vert \Delta_{q}^h \Psi^1_\Theta \Vert_{L^2}^2 dt'  \right)^\f12.
\end{align*}
Using the definition of $\dot{\eta}(t)$ and Definition \ref{def:CLweight} we have
$$ \left( \int_0^t \Vert \pa_y b_\Theta \Vert_{\mathcal{B}^{\frac{1}{2}}} \Vert \Delta_{q}^h \Psi^1_\Theta \Vert_{L^2}^2 dt' \right)^\f12 \lesssim 2^{-q} d_q(\Psi^1_\Theta) \Vert \Psi^1_{\Theta} \Vert_{\tilde{L}^2_{t,\dot{\eta}(t)}(\mathcal{B}^1)}. $$
As a consequence, we arrive at
\begin{align} \label{eq:J1,13} J_{1,23}^q \lesssim  Cd_q^2 2^{-q} \Vert b_\Theta \Vert_{L^\infty_t(\mathcal{B}^\f32)}^\f12 \Vert \pa_y \Phi^1_{\Theta} \Vert_{\tilde{L}^2_t(\mathcal{B}^\f12)} \Vert \Psi^1_\Theta \Vert_{\tilde{L}^2_{t,\dot{\eta}(t)}(\mathcal{B}^1)}. \end{align}

Summing the estimates \eqref{eq:J1,11}, \eqref{eq:J1,12} and \eqref{eq:J1,13} we obtain 
\begin{align} \label{eq:J12q}
	 J_{1,2}^q  \lesssim Cd_q^2 2^{-q} \Vert b_\Theta \Vert_{L^\infty_t(\mathcal{B}^\f32)}^\f12 \Vert \pa_y \Phi^1_{\Theta} \Vert_{\tilde{L}^2_t(\mathcal{B}^\f12)} \Vert \Psi^1_\Theta \Vert_{\tilde{L}^2_{t,\dot{\eta}(t)}(\mathcal{B}^1)}.
\end{align}

Now we will get the estimate of the last term $J_2^q = \int_0^t \abs{\psca{\Delta_q^h  (\Phi^2\pa_y b)_{\Theta},\Delta_q^h \Psi^1_\Theta}}dt' $. For this end, we apply the Bony's decomposition \eqref{eq:Bony} for the horizontal variable to $\Phi^2\pa_y b$, we obtain 
 $$ \Phi^2\pa_y b = T^h_{\Phi^2}\pa_y b + T^h_{\pa_y b}\Phi^2 + R^h(\pa_yb,\Phi^2),$$
 so, we have the following bound 
 $$J_{2}^q = \int_0^t \abs{\psca{\Delta_q^h(T^h_{\Phi^2}\pa_y b + T^h_{\pa_y b}\Phi^2 + R^h(\pa_yb,\Phi^2))_\Theta, \Delta_q^h \Psi^1_\Theta }} dt' \leq J_{2,1}^q + J_{2,2}^q +J_{2,3}^q $$
 where 
 \begin{align*}
 J_{2,1}^q &= \int_0^t \abs{\psca{\Delta_q^h( T^h_{\Phi^2}\pa_y b )_\Theta, \Delta_q^h \Psi^1_\Theta }} dt'\\
 J_{2,2}^q &= \int_0^t \abs{\psca{\Delta_q^h(T^h_{\pa_y b}\Phi^2)_\Theta, \Delta_q^h \Psi^1_\Theta }} dt'\\
 J_{2,3}^q &= \int_0^t \abs{\psca{\Delta_q^h(R^h(\pa_yb,\Phi^2) )_\Theta, \Delta_q^h \Psi^1_\Theta }} dt'\\
 \end{align*}
 Using the support properties given in [\cite{B1981}, Proposition 2.10] and the definition of $ T^h_{\Phi^2}\pa_y b$, we have 
\begin{align*}
	J_{2,1}^q &\leq   \sum_{|q-q'|\leq4} \int_0^t \Vert S_{q'-1}^h \Phi^2_{\Theta} \Vert_{L^\infty} \Vert \Delta_{q'}^h \pa_y b_{\Theta} \Vert_{L^2} \Vert \Delta_q^h \Psi^1_{\Theta} \Vert_{L^2} dt'\\
	&\lesssim \sum_{|q-q'|\leq4} \int_0^t d_{q'}2^{-\frac{q'}{2}}\Vert S_{q'-1}^h \Phi^2_{\Theta} \Vert_{L^\infty} \Vert \pa_y b_{\Theta} \Vert_{\mathcal{B}^\f12} \Vert \Delta_q^h \Psi^1_{\Theta} \Vert_{L^2} dt'
\end{align*}
Due to $\pa_x \Phi^1 +\pa_y \Phi^2 =0$ and Poincar\'e inequality, we can write $ \Phi^2(t,x,y) = -\int_0^y \pa_x \Phi^1(t,x,s) ds$, then we deduce from the lemma \ref{lem:Bernstein} that
Since 
 \begin{align}
 \Vert \Delta_{q'}^h \Phi^2_{\Theta} \Vert_{L^\infty} &\leq \int_0^1 \Vert \Delta_{q'}^h \pa_x \Phi^1_{\Theta}(t,x,s) \Vert_{L^\infty_h} ds \notag \\
 &\leq 2^{\frac{3q}{2}} \int_0^1 \Vert \Delta_{q'}^h  \Phi^1_{\Theta}(t,x,s) \Vert_{L^2_h} ds \leq 2^{\frac{3q}{2}} \Vert \Delta_{q'}^h \Phi^1_{\Theta}(t,x,s) \Vert_{L^2} \notag
 \end{align}
from which, we infer
\begin{align*}
    &\left(\int_0^t \Vert S_{q'-1}^h \Phi^2_{\Theta} \Vert_{L^\infty}^2 \Vert \pa_y b_{\Theta} \Vert_{\mathcal{B}^\f12}dt' \right)^\f12 \\
    &\sum_{l\leq q'-2} 2^{\frac{3l}{2}} \left(\int_0^t \Vert \Delta_l^h \Phi^1_{\Theta} \Vert_{L^2}^2 \Vert \pa_y b_{\Theta} \Vert_{\mathcal{B}^\f12}dt' \right)^\f12 \\
    &\lesssim \sum_{l\leq q'-2}d_l 2^{\frac{l}{2}} \Vert \Phi^1(t) \Vert_{\tilde{L}^2_{t,\dot{\eta}(t)}(\mathcal{B}^1)} \\
    &\lesssim 2^{\frac{q'}{2}} \Vert \Phi^1(t) \Vert_{\tilde{L}^2_{t,\dot{\eta}(t)}(\mathcal{B}^1)}.
\end{align*}
Then we replace in $J_{2,1}^q$, we obtain
\begin{align} \label{eq:J2,1}
	 J_{2,1}^q = Cd_q^2 2^{-q} \Vert \Phi^1_{\Theta} \Vert_{\tilde{L}^2_{t,\dot{\eta}(t)}(\mathcal{B}^1)} \Vert \Psi^1_\Theta \Vert_{\tilde{L}^2_{t,\dot{\eta}(t)}(\mathcal{B}^1)}.
\end{align}

Now let get the estimate of the second term $J_{2,2}$
\begin{align*}
	J_{2,2}^q &= \int_0^t \abs{\psca{\Delta_q^h  (T^h_{\pa_y b}\Phi^2)_{\Theta},\Delta_q^h \Psi^1_\Theta}} dt'\\ &\leq \sum_{|q-q'|\leq 4} \int_0^t  \Vert S_{q'-1}^h \pa_y b_{\Theta} \Vert_{L^{\infty}_h(L^2_v)} \Vert \Delta_{q'}^h \Phi^2_\Theta \Vert_{L^2_h(L^\infty_v)} \Vert \Delta_{q}^h \Psi^1_\Theta \Vert_{L^2} dt' \\
	&\lesssim \sum_{|q-q'|\leq 4} 2^{q'} \int_0^t  \Vert \pa_y b_{\Theta} \Vert_{\mathcal{B}^\f12} \Vert \Delta_{q'}^h \Phi^1_\Theta \Vert_{L^2} \Vert \Delta_{q}^h \Psi^1_\Theta \Vert_{L^2} dt' \\
	&\lesssim \sum_{|q-q'|\leq 4} 2^{q'} \left( \int_0^t \Vert \pa_y b_{\Theta} \Vert_{\mathcal{B}^\f12} \Vert \Delta_{q'}^h \Phi^1_\Theta \Vert_{L^2}^2 dt'\right)^\f12\left( \int_0^t\Vert \pa_y b_{\Theta} \Vert_{\mathcal{B}^\f12} \Vert \Delta_{q'}^h \Psi^1_\Theta \Vert_{L^2}^2 dt'\right)^\f12
\end{align*}
Using the definition of $\dot{\eta}(t)$ and Definition \ref{def:CLweight} we have
$$ \left( \int_0^t \Vert \pa_y b_\Theta \Vert_{\mathcal{B}^{\frac{1}{2}}} \Vert \Delta_{q}^h \Psi^1_\Theta \Vert_{L^2}^2 dt' \right)^\f12 \lesssim 2^{-q} d_q(\Psi^1_\Theta) \Vert \Psi^1_{\Theta} \Vert_{\tilde{L}^2_{t,\dot{\eta}(t)}(\mathcal{B}^1)}. $$
Then,
\begin{align} \label{eq:J2,2} J_{2,2}^q \lesssim  C2^{-q} d_q^2 \Vert \Phi^1_\Theta \Vert_{\tilde{L}^2_{t,\dot{\eta}(t)}(\mathcal{B}^1)} \Vert \Psi^1_\Theta \Vert_{\tilde{L}^2_{t,\dot{\eta}(t)}(\mathcal{B}^1)}, \end{align}
where
$$ d_q^2 = d_q(\Psi^1_\Theta) \left(\sum_{|q-q'|\leq 4} d_{q'}(\Phi^1_\Theta) \right) $$

In a similar way, we have 
\begin{align*}
	J_{2,3}^q &= \int_0^t \abs{\psca{\Delta_q^h  (R^h(\Phi^2,\pa_y b))_{\Theta},\Delta_q^h \Psi^1_\Theta}}dt'\\ &\lesssim 2^{\frac{q}{2}} \sum_{q'\geq q-3} \int_0^t \Vert \tilde{\Delta}_{q'}^h \pa_y b_\Theta \Vert_{L^2} \Vert \Delta_{q'}^h \Phi^2_\Theta \Vert_{L^2_h(L^\infty_v)} \Vert \Delta_{q}^h \Psi^1_\Theta \Vert_{L^2} dt' \\
	&\lesssim 2^{\frac{q}{2}} \sum_{q'\geq q-3} 2^{\frac{q'}{2}} \int_0^t \Vert \pa_y b_\Theta \Vert_{\mathcal{B}^\f12} \Vert \Delta_{q'}^h \Phi^1_\Theta \Vert_{L^2} \Vert \Delta_{q}^h \Psi^1_\Theta \Vert_{L^2} dt' \\
	&\lesssim  2^{\frac{q}{2}} \sum_{q'\geq q-3} 2^{\frac{q'}{2}} \left( \int_0^t \Vert \pa_y b_{\Theta} \Vert_{\mathcal{B}^\f12} \Vert \Delta_{q'}^h \Phi^1_\Theta \Vert_{L^2}^2 dt'\right)^\f12\left( \int_0^t\Vert \pa_y b_{\Theta} \Vert_{\mathcal{B}^\f12} \Vert \Delta_{q'}^h \Psi^1_\Theta \Vert_{L^2}^2 dt'\right)^\f12.
\end{align*}
Using the definition of $\dot{\eta}(t)$ and Definition \ref{def:CLweight} we have
$$ \left( \int_0^t \Vert \pa_y b_\Theta \Vert_{\mathcal{B}^{\frac{1}{2}}} \Vert \Delta_{q}^h \Psi^1_\Theta \Vert_{L^2}^2 dt' \right)^\f12 \lesssim 2^{-q} d_q(\Psi^1_\Theta) \Vert \Psi^1_{\Theta} \Vert_{\tilde{L}^2_{t,\dot{\eta}(t)}(\mathcal{B}^1)}. $$
As a consequence, we arrive at
\begin{align} \label{eq:J2,3} J_{2,3}^q \lesssim  C2^{-q} d_q^2 \Vert \Psi^1_{\Theta} \Vert_{\tilde{L}^2_{t,\dot{\eta}(t)}(\mathcal{B}^1)}\Vert \Phi^1_{\Theta} \Vert_{\tilde{L}^2_{t,\dot{\eta}(t)}(\mathcal{B}^1)}. \end{align}

Summing the estimates \eqref{eq:J2,1}, \eqref{eq:J2,2} and \eqref{eq:J2,3} we obtain 
\begin{align} \label{eq:J2q}
	 J_{2}^q  \lesssim Cd_q^2 2^{-q} \Vert \Psi^1_{\Theta} \Vert_{\tilde{L}^2_{t,\dot{\eta}(t)}(\mathcal{B}^1)}\Vert \Phi^1_{\Theta} \Vert_{\tilde{L}^2_{t,\dot{\eta}(t)}(\mathcal{B}^1)}.
\end{align}
By summing all the resulting estimates we obtain the proof of the estimate \eqref{eq:G1}.
\end{proof}
\bigskip

We will now study the second term $G_2^q$.
\begin{proof} of the estimate \eqref{eq:G2}. Using the definition of $R^2_\Theta$, we write 
\begin{align*}
	R^2_\Theta &= - \eps^2\left(\pa_t v -\eps^2 \pa_x^2 v -\pa_y^2 v + u^{\eps} \pa_x v^{\eps} + v^\eps \pa_y v^\eps + b^\eps \pa_x c^\eps + c^\eps \pa_y c^\eps \right)_\Theta \\
	&= Q^2_\Theta + (b^\eps \pa_x c^\eps)_\Theta + (c^\eps \pa_y c^\eps)_\Theta.
\end{align*}
We have already do the proof of the estimate $\int_0^t \abs{\psca{\Delta_q^h Q^2_{\Theta}, \Delta_q^h \Psi^2_{\Theta}}_{L^2}} dt'$ ( see the proof of the estimate $(4.79)$ in \cite{AN2020}), then 
\begin{align}\label{eq:Q^2}
 \sum_{q\in \mathbb{Z}} 2^q  &\int_0^t \abs{\psca{\Delta_q^h Q^2_{\Theta}, \Delta_q^h \Psi^2_{\Theta}}_{L^2}} dt'\\
&\lesssim \Vert (\Psi^1_\Theta,\eps \Psi^2_\Theta) \Vert_{\tilde{L}^2_{t,\dot{\eta}(t)}(\mathcal{B}^{1})}  + \eps^2 \Vert \Psi^2_\Theta \Vert_{\tilde{L}^2_{t,\dot{\eta}(t)}(\mathcal{B}^1)} (\Vert \Psi^2_\Theta \Vert_{\tilde{L}^2_{t,\dot{\eta}(t)}(\mathcal{B}^1)}  \notag \\ 
&+ \eps^2\Vert (\pa_y \Psi^2_\Theta,\eps \pa_x\Psi^2_\Theta)\Vert_{\tilde{L}^2_{t}(\mathcal{B}^{\f12})}\left( \Vert (\pa_t u)_\Theta \Vert_{\tilde{L}^2_{t}(\mathcal{B}^{\f32})} + \Vert \pa_y u_\Theta \Vert_{\tilde{L}^2_{t}(\mathcal{B}^{\f32})} +  \Vert \pa_y u_\Theta \Vert_{\tilde{L}^2_{t}(\mathcal{B}^{\f52})}\right)   \notag \\
&+\Vert u^\eps_\Theta \Vert_{\tilde{L}^\infty_{t}(\mathcal{B}^{\f12})}^{\f12} \Vert \pa_y u_\Theta \Vert_{\tilde{L}^2_{t}(\mathcal{B}^2)} + \Vert u_\Theta \Vert_{\tilde{L}^\infty_{t}(\mathcal{B}^{\f32})}^\f12(\Vert \pa_y \Psi^2_\Theta \Vert_{\tilde{L}^2_{t}(\mathcal{B}^{\f12})} + \Vert \pa_y u_\Theta \Vert_{\tilde{L}^2_{t}(\mathcal{B}^{\f32})})) . \notag
\end{align} 
So, we still have to get the estimate of 
 \begin{align*}
 J_3^q &= \int_0^t \abs{\psca{\Delta_q^h(b^\eps \pa_x c^\eps)_\Theta, \Delta_q^h \Psi^2_\Theta }} dt' \\
 J_4^q &= \int_0^t \abs{\psca{\Delta_q^h(c^\eps \pa_y c^\eps)_\Theta, \Delta_q^h \Psi^2_\Theta }} dt'.
 \end{align*}
Then, we start by giving the estimate of the term $J_3^q = \int_0^t \abs{\psca{\Delta_q^h(b^\eps \pa_x c^\eps)_\Theta, \Delta_q^h \Psi^2_\Theta }} dt'.$ We write 
\begin{align*}
	J_3^q \leq \eps^2\pare{J_{31}^q + J_{32}^q},
\end{align*}
where
\begin{align*}
	J_{31}^q &= \int_0^t \abs{\psca{\Delta_q^h(b^\eps \pa_x \Phi^2)_{\Theta},\Delta_q^h \Psi^2_\Theta}_{L^2}} dt' \\
	J_{32}^q &= \int_0^t \abs{\psca{\Delta_q^h(b^\eps \pa_x c)_{\Theta},\Delta_q^h \Psi^2_\Theta}_{L^2}} dt'.
\end{align*}
It follows from Lemma $\ref{lem C+B+A}$ that 
\begin{align*} 
	\sum_{q\in\mathbb{Z}} 2^q J_{31}^q = \sum_{q\in\mathbb{Z}} 2^q\int_0^t \abs{\psca{\Delta_q^h(b^\eps \pa_x \Phi^2)_\Theta,\Delta_q^h \Psi^2_\Theta}_{L^2}} dt' \lesssim    \Vert \eps \Psi^2_\Theta \Vert_{\tilde{L}^2_{t,\dot{\eta}(t)}(\mathcal{B}^1)} \Vert \eps \Phi^2_\Theta \Vert_{\tilde{L}^2_{t,\dot{\eta}(t)}(\mathcal{B}^1)}.
\end{align*}
For the second term, Bony's decomposition for the horizontal variable gives 
\begin{align*}
J_{32}^q = \int_0^t \abs{\psca{\Delta_q^h(b^\eps \pa_x c)_{\Theta},\Delta_q^h \Psi^2_\Theta}_{L^2}} dt' \leq J_{321}^q + J_{322}^q + J_{323}^q,
\end{align*}
with
\begin{align*}
J_{321}^q &= \int_0^t \abs{\psca{\Delta_q^h(T^h_{b^\eps} \pa_x c)_{\Theta},\Delta_q^h \Psi^2_\Theta}_{L^2}} dt' \\
J_{322}^q &= \int_0^t \abs{\psca{\Delta_q^h(T^h_{\pa_x c} b^\eps)_{\Theta},\Delta_q^h \Psi^2_\Theta}_{L^2}} dt'\\
J_{323}^q &= \int_0^t \abs{\psca{\Delta_q^h(R^h(b^\eps,\pa_x c))_{\Theta},\Delta_q^h \Psi^2_\Theta}_{L^2}} dt'.
\end{align*}
Using the estimate 
$$ \Vert S_{q'-1}^h b^\eps_\Theta \Vert_{L^\infty} \lesssim \Vert b_\Theta^\eps \Vert^{\f12}_{\mathcal{B}^{\f12}} \Vert \pa_y b_\Theta^\eps \Vert^{\f12}_{\mathcal{B}^{\f12}} ,$$
and the relation \eqref{vv}, we have 
\begin{align*}
J_{321}^q = \int_0^t \abs{\psca{\Delta_q^h(T^h_{b^\eps} \pa_x c)_\Theta,\Delta_q^h \Psi^2_\Theta}_{L^2}} dt' &\lesssim \sum_{|q'-q| \leq 4} \int_0^t  \Vert S_{q'-1}^h b^\eps_\Theta \Vert_{L^\infty} \Vert \Delta_{q'}^h \pa_x c_\Theta \Vert_{L^2} \Vert \Delta_q^h \Psi^2_\Theta \Vert_{L^2} dt' \\
&\lesssim \sum_{|q'-q| \leq 4} \int_0^t \Vert b_\Theta^\eps \Vert^{\f12}_{\mathcal{B}^{\f12}} \Vert \pa_y b_\Theta^\eps \Vert^{\f12}_{\mathcal{B}^{\f12}} 2^{2q'} \Vert \Delta_{q'}^h b_\Theta \Vert_{L^2} \Vert \Delta_q^h \Psi^2_\Theta \Vert_{L^2} dt' \\
&\lesssim d_q^2 2^{-q} \Vert b_\Theta^\eps \Vert^{\f12}_{\tilde{L}^\infty_{t}(\mathcal{B}^{\f12})} \Vert \pa_y b_{\Theta} \Vert_{\tilde{L}^2_{t}(\mathcal{B}^2)}  \Vert \Psi^2_\Theta \Vert_{\tilde{L}^2_{t,\dot{\eta}(t)}(\mathcal{B}^1)}.
\end{align*}
Multiply the above inequality by $2^q$ and summing over $q \in \ZZ$, we obtain
\begin{equation} \label{eq:J321}
\sum_{q\in\mathbb{Z}} 2^q J_{321}^q \lesssim \Vert b_\Theta^\eps \Vert^{\f12}_{\tilde{L}^\infty_{t}(\mathcal{B}^{\f12})} \Vert \pa_y b_{\Theta} \Vert_{\tilde{L}^2_{t}(\mathcal{B}^2)}  \Vert \Psi^2_\Theta \Vert_{\tilde{L}^2_{t,\dot{\eta}(t)}(\mathcal{B}^1)}.
\end{equation}
In a similar way, the fact that
\begin{align*}
\Vert S_{q'-1}^h \pa_x c_\Theta \Vert_{L^\infty} \lesssim \int_0^y \Vert S_{q'-1}^h \pa_x (\pa_x b_\Theta(t,x,s) \Vert_{L^\infty} ds \lesssim 2^{\frac{q'}{2}} \Vert \pa_y b_\Theta \Vert_{\mathcal{B}^2},
\end{align*}
leads to 
\begin{align*}
J_{322}^q = \int_0^t \abs{\psca{\Delta_q^h(T^h_{\pa_x c} b^\eps)_\Theta,\Delta_q^h \Psi^2_\Theta}_{L^2}}dt' \lesssim d_q^2 2^{-q} \Vert b_\Theta^\eps \Vert^{\f12}_{\tilde{L}^\infty_{t}(\mathcal{B}^{\f12})} \Vert \pa_y b_{\Theta} \Vert_{\tilde{L}^2_{t}(\mathcal{B}^2)}  \Vert \Psi^2_\Theta \Vert_{\tilde{L}^2_{t,\dot{\eta}(t)}(\mathcal{B}^1)}.
\end{align*}
So multiply by $2^q$ and summing over $q \in \ZZ$ imply
\begin{equation} \label{eq:J322}
\sum_{q\in\mathbb{Z}} 2^q J_{322}^q \lesssim \Vert b_\Theta^\eps \Vert^{\f12}_{\tilde{L}^\infty_{t}(\mathcal{B}^{\f12})} \Vert \pa_y b_{\Theta} \Vert_{\tilde{L}^2_{t}(\mathcal{B}^2)}  \Vert \Psi^2_\Theta \Vert_{\tilde{L}^2_{t,\dot{\eta}(t)}(\mathcal{B}^1)}.
\end{equation}
For the last term $J_{323}^q$, we have 
\begin{align*}
J_{323}^q = \int_0^t \abs{\psca{ \Delta_q^h(R^h(\pa_x c, b^\eps))_\Theta,\Delta_q^h \Psi^2_\Theta}_{L^2}} &\lesssim d_q^2 2^{-q} \Vert b_\Theta^\eps \Vert^{\f12}_{\tilde{L}^\infty_{t}(\mathcal{B}^{\f12})} \Vert \pa_y b_{\Theta} \Vert_{\tilde{L}^2_{t}(\mathcal{B}^2)}  \Vert \Psi^2_\Theta \Vert_{\tilde{L}^2_{t,\dot{\eta}(t)}(\mathcal{B}^1)}.
\end{align*}
Multplying the result by $2^q$, and summing over $q \in \ZZ$, we get
\begin{equation} \label{eq:J323}
\sum_{q\in\mathbb{Z}} 2^q J_{323}^q \lesssim \Vert b_\Theta^\eps \Vert^{\f12}_{\tilde{L}^\infty_{t}(\mathcal{B}^{\f12})} \Vert \pa_y b_{\Theta} \Vert_{\tilde{L}^2_{t}(\mathcal{B}^2)}  \Vert \Psi^2_\Theta \Vert_{\tilde{L}^2_{t,\dot{\eta}(t)}(\mathcal{B}^1)}.
\end{equation}
Summing Inequalities \eqref{eq:J321}, \eqref{eq:J322} and \eqref{eq:J323} finally yields 
$$ \sum_{q\in\mathbb{Z}} 2^q J_{32}^q \lesssim \Vert b_\Theta^\eps \Vert^{\f12}_{\tilde{L}^\infty_{t}(\mathcal{B}^{\f12})} \Vert \pa_y b_{\Theta} \Vert_{\tilde{L}^2_{t}(\mathcal{B}^2)}  \Vert \Psi^2_\Theta \Vert_{\tilde{L}^2_{t,\dot{\eta}(t)}(\mathcal{B}^1)}. $$

To end, we need to estimate the last term $ J_4^q = \int_0^t \abs{\psca{\Delta_q^h(c^\eps \pa_y c^\eps)_\Theta, \Delta_q^h \Psi^2_\Theta }} dt'$. We first note that 
$$ c^\eps \pa_y c^\eps = c\pa_yc +\Phi^2\pa_y\Phi^2 + c\pa_y\Phi^2+\Phi^2\pa_yc. $$
We first deduce from the lemma \eqref{lem:vvv} that 
$$ \eps^2 \int_0^t \abs{\psca{\Delta_q^h(\Phi^2\pa_y\Phi^2)_\Theta, \Delta_q^h \Psi^2_\Theta }} dt' \leq C d_q^2 2^{-q} \Vert \Phi^1_\Theta \Vert_{\tilde{L}^2_{t,\dot{\eta}(t)}(\mathcal{B}^1)} \Vert \eps \Psi^2_\Theta \Vert_{\tilde{L}^2_{t,\dot{\eta}(t)}(\mathcal{B}^1)}  $$
We deduce also form the proof of \eqref{eq:J12q} that 
\begin{align*} \int_0^t \abs{\psca{\Delta_q^h(c\pa_yc)_\Theta, \Delta_q^h \Psi^2_\Theta }} dt' &\leq C d_q^2 2^{-q}\Vert b_\Theta \Vert_{L^\infty_t(\mathcal{B}^\f32)}^\f12 \Vert \pa_y c_{\Theta} \Vert_{\tilde{L}^2_t(\mathcal{B}^\f12)} \Vert \Psi^2_\Theta \Vert_{\tilde{L}^2_{t,\dot{\eta}(t)}(\mathcal{B}^1)} \\
&\leq C d_q^2 2^{-q}\Vert b_\Theta \Vert_{L^\infty_t(\mathcal{B}^\f32)}^\f12 \Vert \pa_y b_{\Theta} \Vert_{\tilde{L}^2_t(\mathcal{B}^\f32)} \Vert \Psi^2_\Theta \Vert_{\tilde{L}^2_{t,\dot{\eta}(t)}(\mathcal{B}^1)},
\end{align*}
and 
$$ \int_0^t \abs{\psca{\Delta_q^h(c\pa_y\Phi^2)_\Theta, \Delta_q^h \Psi^2_\Theta }} dt' \leq C2^{-q}d_q^2 \Vert b_\Theta \Vert_{L^\infty_t(\mathcal{B}^\f32)}^\f12 \Vert \pa_y \Phi^2_{\Theta} \Vert_{\tilde{L}^2_t(\mathcal{B}^\f12)} \Vert \Psi^2_\Theta \Vert_{\tilde{L}^2_{t,\dot{\eta}(t)}(\mathcal{B}^1)} .$$

And \eqref{eq:I1q} ensure that
\begin{align*}
\int_0^t &\abs{\psca{\Delta_q^h(\Phi^2\pa_x b)_\Theta, \Delta_q^h \Psi^2_\Theta }} dt' \\ &\leq Cd_q^2 2^{-q}\left( \Vert b_{\Theta} \Vert_{\tilde{L}^\infty_t(\mathcal{B}^{\frac{3}{2}})}^{\frac{1}{2}} \Vert \pa_y \Phi^2_\Theta \Vert_{\tilde{L}^2_t(\mathcal{B}^{\frac{1}{2}})} + \Vert \Phi^2_{\Theta} \Vert_{\tilde{L}^2_{t,\dot{\eta}(t)}(\mathcal{B}^1)}\right)\Vert \Psi^2_{\Theta} \Vert_{\tilde{L}^2_{t,\dot{\eta}(t)}(\mathcal{B}^1)}.
\end{align*}
As a result, multiply the above inequality by $2^q$ and summing over $q \in \ZZ$, we obtain
\begin{align} \label{J5}
\eps^2 &\sum_{q\in\mathbb{Z}} 2^q J_4^q \lesssim  \Vert (\Phi^1_\Theta,\eps \Psi^2_\Theta) \Vert^2_{\tilde{L}^2_{t,\dot{\eta}(t)}(\mathcal{B}^1)} \\ &+ \eps^2 \Vert b_\Theta \Vert_{\tilde{L}^\infty_t(\mathcal{B}^{\f32})}^\f12  \left(  \Vert \pa_y b_\Theta \Vert_{\tilde{L}^2_t(\mathcal{B}^{\f12})} + \Vert \pa_y \Phi^2_\Theta \Vert_{\tilde{L}^2_t(\mathcal{B}^{\f32})} \right)  \Vert \Psi^2_\Theta \Vert_{\tilde{L}^2_{t,\dot{\eta}(t)}(\mathcal{B}^1)}\big) +  \Vert \eps \Phi^2_\Theta \Vert^2_{\tilde{L}^2_{t,\dot{\eta}(t)}(\mathcal{B}^1)}. \notag
\end{align}
By summing all the resulting estimates we obtain the proof of the estimate \eqref{eq:G2}.
\end{proof}
For the proof of the estimates $G_3^q$ and $G_4^q$ it the same of $G_1^q$ and $G_2^q$.

\newpage
\textsc{IMB, Université de Bordeaux, 351, cours de la Libération, 33405 Talence , France}

\textit{Email address} : nacer.aarach@math.u-bordeaux.fr

\begin{thebibliography}{2}

	
\bibitem{awxy}  Alexandre R.,  Wang Y., Xu C.-J. and Yang T., Well-posedness of The Prandtl Equation in Sobolev Spaces, {\em J. Amer. Math. Soc.}, {\bf 28}, 2015, 745-784.


\bibitem{BCD2011book} H. Bahouri, J.-Y. Chemin and R. Danchin,
\emph{Fourier Analysis and Nonlinear Partial Differential Equations}, Grundlehren der Mathematischen Wissenschaften, vol. 343, Springer, Heidelberg, 2011.


\bibitem{B1981} J.-M. Bony, Calcul symbolique et propagation des singularit\'es pour les \'equations aux d\'eriv\'ees partielles non lin\'eaires, {\it Annales de l'\'Ecole Normale Sup\'erieure}, \textbf{14}, 1981, p.209-246.


\bibitem{Sadourny} P. Bougeault, R. Sadourny, Dynamique de l'atmosph\`ere et de l'oc\'ean, \'Editions de l'\'Ecole Polytechnique (2001).


\bibitem{CLT2019} C. Cao, Q. Lin and E. S. Titi, On the well-posedness of reduced 3D primitive geostrophic adjustment model with weak dissipation (preprint).


\bibitem{C1995book} J.-Y. Chemin, Fluides parfaits incompressibles, {\it Ast\'erisque}, \textbf{230}, 1995.


\bibitem{C2004} J.-Y. Chemin, Le système de Navier-Stokes incompressible soixante dix ans après Jean Leray, \textit{Actes des Journées Mathématiques à la Mémoire de Jean Leray}, Séminaires \& Congrès, \textbf{9}, Soc. Math. France, Paris, 2004, p.99-123.


\bibitem{CDGG4} J.-Y. Chemin, B. Desjardins, I. Gallagher, E. Grenier, \textit{Mathematical Geophysics: An introduction to rotating fluids and to the Navier-Stokes equations}, Oxford University Press, (2006).


\bibitem{CGP2011} J.-Y. Chemin, I. Gallagher, M. Paicu, Global regularity for some classes of large solutions to the Navier-Stokes equations


\bibitem{CL1992} J.-Y. Chemin and N. Lerner, Flot de champs de vecteurs non Lipschitziens et \'equations de Navier-Stokes, \textit{Journal of Differential Equations}, \textbf{121}, 1992, 314-328.


\bibitem{Cushman} B. Cushman-Roisin, Introduction to geophysical fluid dynamics, Prentice-Hall (1994).


\bibitem{e-1} E, W.: Boundary layer theory and the zero-viscosity limit of the Navier-Stokes equation. {\it Acta Math. Sin. (Engl. Ser.)} {\bf 16} (2000) 207-218.

\bibitem{e-2} E, W. \&  Enquist, B.: Blow up of solutions of the unsteady Prandtl's equation, {\it Comm. Pure Appl. Math.}, {\bf 50} (1997) 1287-1293.


\bibitem{EmMa} P. Embid, A. Majda, Averaging over fast gravity waves for geophysical flows with arbitrary potential vorticity, \textit{Communications in Partial Differential Equations}, \textbf{21} (1996), 619--658.


\bibitem{GV-D}  D. G\'erard-Varet, \& E. Dormy, On the ill-posedness of the Prandtl equation, {\it J. Amer. Math. Soc.},  {\bf 23} (2010), 591-609.


\bibitem{GvM2015} D. G\'erard-Varet and N. Masmoudi, Well-posedness for the Prandtl system without analyticity or mono-tonicity, \emph{Ann. Sci. \'Ecole Norm. Sup.} (4), \textbf{48} (2015), 1273-1325.


\bibitem{GvMV2019} D. G\'erard-Varet, N. Masmoudi and V. Vicol, Well-posedness of the hydrostatic Navier-Stokes equations,
arXiv:1804.04489.


\bibitem{G1982} A.E. Gill, Atmosphere-Ocean Dynamics, \emph{Academic Press New York} (1982).


\bibitem{H2004} J.R. Holton, An Introduction to Dynamic Meteorology, 4th edition, \emph{Elsevier Academic Press} (2004).


\bibitem{LCS2003} M. C. Lombardo, M. Cannone and M. Sammartino, Well-posedness of the boundary layer equations, \emph{SIAM J. Math. Anal.}, \textbf{35} (2003), 987-1004.

\bibitem{M-W} N. Masmoudi and T.~K. Wong, Local-in-time existence and uniqueness of solutions to the Prandtl equations by energy methods, {\em Comm. Pure Appl. Math.}, {\bf 68} (2015) 1683-1741.


\bibitem{oleinik}   O. A. Oleinik,    V. N. Samokhin,   {\it Mathematical Models in Boundary Layers Theory}. Chapman \& Hall/CRC, 1999.


\bibitem{PZ2011} M. Paicu, Z. Zhang, Global regularity for the Navier-Stokes equations with some classes of large initial data


\bibitem{PZZ2019} M. Paicu, P. Zhang and Z. Zhang, \textit{On the hydrostatic approximation of the Navier-Stokes equations in a thin strip}, 2019, submitted.


\bibitem{Pedlosky} J. Pedlosky,  \textit{Geophysical fluid dynamics}, Springer (1979).

\bibitem{PZ2005} R. Plougonven and V. Zeitlin, Lagrangian approach to the geostrophic adjustment of frontal anomalies in a stratified fluid, \emph{Geophys. Astrophys. Fluid Dyn.}, \textbf{99} (2005), 101–135. 

\bibitem{samm} M.  Sammartino,   R. E. Caflisch,  Zero viscosity limit for analytic solutions of the Navier-Stokes equations on a half-space, I. Existence for Euler and Prandtl equations. {\it Comm. Math. Phys.}, 192(1998) 433-461; II. Construction of the Navier-Stokes solution. {\it Comm. Math. Phys.}, {\bf 192} (1998) 463-491.

\bibitem{6} D.  Bresch,   A. Kazhikhov and J. Lemoine, On the two-dimensional hydrostatic Navier-Stokes equations. SIAM J.Math. Annal., {\bf 36} (2004), 796-814

\bibitem{30} F. Guillén-González, N. Masmoudi and M.A. Rodríguez-Bellido : Anisotropic estimates and strong solutions of the primitive equations. Differ. Integral Equ., {\bf 14}(2001), 1381–1408 
\bibitem{MP2020} M. Paicu, P. Zhang, \textit{Global existence and decay of solutions to prandtl system with small analytic data}, 2020, submitted

\bibitem{NP2020} N. Liu, P. Zhang, \textit{Global small analytic solutions of MHD boundary layer equations}, 2020, submitted
\bibitem{AN2020} A. Nacer, N. Van-Sang, \textit{Hydrostatic approximation of the 2D primitive equations in a thin strip}, 2020, submitted

\bibitem{JH} J. Hartmann.  \textit{Theory of the laminar flow of an electrically conductive liquid in ahomogeneous magnetic field},K. Dan. Vidensk. Selsk. Mat. Fys. Medd.15(6), 1?28(1937).

\bibitem{BEE2004} B. Desjardins, E. Dormy, E. Grenier. \textit{Boundary Layer Instability at the top of theEarth’s outer core}.Journal of Computational and Applied Mathematics166 (1),123?131 (2004).

\bibitem{J} J. Wesson. \textit{Tokamacs}. Oxford University Press, Oxford, 2011

\bibitem{Temam}  J. L. Lions, R. Temam and S. Wang, New formulations of the primitive equations of the atmosphere and applications, Non-linearity, {\bf 5 } (1992), 237–288.

\bibitem{Wang}  J. L. Lions, R. Temam and S. Wang, On the equations of the large-scale ocean, Non-linearity, {\bf 5 } (1992), 1007–1053

\bibitem{Lions}  J. L. Lions, R. Temam and S. Wang, Mathematical study of the coupled models of atmosphere and ocean (CAO III), J. Math. Pures Appl., {\bf 74} (1995), 105–163.

\bibitem{Ziane} R. Temam and M. Ziane, Some mathematical problems in geophysical fluid dynamics, Handbook of Mathematical Fluid-Dynamics, 2003.
\end{thebibliography}
\end{document}